\documentclass[a4paper,10pt]{article}
\usepackage[english]{babel}
\usepackage{a4wide}
\usepackage[utf8]{inputenc}
\usepackage{amssymb,amsmath,amscd,amsfonts,amsthm,bbm,mathrsfs,enumerate}
\usepackage{booktabs}
\usepackage[colorlinks=true, linkcolor=blue, citecolor=blue]{hyperref}
\usepackage{mathabx}
\usepackage{graphicx}
\usepackage{color}
\usepackage{accents}
\usepackage{subcaption}
\usepackage{enumitem}

\DeclareUnicodeCharacter{00A0}{~}

\theoremstyle{definition}
\newtheorem{theorem}{Theorem}
\newtheorem*{theorem*}{Statement}
\newtheorem{lemma}[theorem]{Lemma}
\newtheorem{corollary}[theorem]{Corollary}
\newtheorem{proposition}[theorem]{Proposition}

\newtheorem{remark}{Remark}

\newtheorem*{condition*}{Condition}
\newtheorem{assumption}{Assumption}

\DeclareMathOperator{\var}{\mathbb Var}
\DeclareMathOperator{\cov}{cov}
\DeclareMathOperator{\rank}{rank}
\DeclareMathOperator{\tr}{tr}

\DeclareMathOperator{\dist}{dist}
\DeclareMathOperator{\colspan}{span}
\DeclareMathOperator{\codim}{codim}

\DeclareMathOperator{\leb}{Leb}
\DeclareMathOperator{\conv}{conv}
\DeclareMathOperator{\vect}{vec}

\newcommand{\1}{\mathbbm 1}
\newcommand{\T}{{\mathsf T}} 

\newcommand{\ZZ}{\mathbb Z}
\newcommand{\CC}{\mathbb{C}}
\newcommand{\PP}{{{\mathbb P}}} 
\newcommand{\EE}{{{\mathbb E}}} 
\newcommand{\NN}{{{\mathbb N}}} 
\newcommand{\RR}{{{\mathbb R}}} 
\newcommand{\TT}{{{\mathbb T}}} 
\newcommand{\SSS}{{{\mathbb S}}}

\newcommand{\mcT}{{\mathscr T}}

\newcommand{\cC}{{\mathcal C}} 
\newcommand{\cD}{{\mathcal D}} 
\newcommand{\cE}{{\mathcal E}} 
\newcommand{\cF}{{\mathcal F}} 
\newcommand{\cH}{{\mathcal H}} 
\newcommand{\cI}{{\mathcal I}} 
\newcommand{\cJ}{{\mathcal J}} 
\newcommand{\cK}{{\mathcal K}} 
\newcommand{\cL}{{\mathcal L}} 
\newcommand{\cM}{{\mathcal M}} 
\newcommand{\cN}{{\mathcal N}}

\newcommand{\cR}{{\mathcal R}} 
\newcommand{\cQ}{{\mathcal Q}} 
\newcommand{\cS}{{\mathcal S}} 
\newcommand{\cV}{{\mathcal V}} 
 
\newcommand{\cU}{{\mathcal U}} 
\newcommand{\cIc}{{\mathcal I}^{\text{c}}}

\newcommand{\gmax}{\boldsymbol\gamma_{\sup}} 

\newcommand{\ssa}{{\mathsf a}} 
\newcommand{\sF}{{\mathsf F}} 
\newcommand{\ssf}{{\mathsf f}} 

\newcommand{\be}{{\boldsymbol\varepsilon}} 
 
\newcommand{\bv}{{\boldsymbol v}} 
\newcommand{\bG}{{\boldsymbol G}} 
\newcommand{\bH}{{\boldsymbol H}} 
\newcommand{\bM}{{\boldsymbol M}} 
 
\newcommand{\bS}{{\boldsymbol S}} 
\newcommand{\bU}{{\boldsymbol U}} 

\newcommand{\ceF}{\widecheck{\mathcal F}} 
\newcommand{\hG}{{\widehat G}} 
 
\newcommand{\hQ}{{\widehat Q}} 
\newcommand{\hS}{{\widehat S}}

\newcommand{\hX}{{\widehat X}} 

\newcommand{\tDelta}{{\widetilde\Delta}}

\newcommand{\tQ}{{\widetilde Q}} 
\newcommand{\tX}{{\widetilde X}} 

\newcommand{\tqs}{\widetilde{qs}}

\newcommand{\comp}{\text{comp}} 
\newcommand{\incomp}{\text{incomp}} 

\newcommand{\Eop}{{\mathcal E}_{\text{op}}} 
\newcommand{\Eden}{{\mathcal E}_{\text{Den}}} 
 
\newcommand{\HS}{{\text{HS}}} 
 
\newcommand{\bmax}{\mathfrak{b}} 
\newcommand{\Cgood}{{\mathcal C}_{\text{good}}} 
\newcommand{\Cbad}{{\mathcal C}_{\text{bad}}}

\newcommand{\ps}[1]{\langle #1 \rangle}
\newcommand{\pss}[1]{\langle\!\langle #1 \rangle\!\rangle}

\newcommand{\bs}{\boldsymbol}
\DeclareMathOperator*{\diag}{diag}
\newcommand{\eqdef}{\triangleq} 
\newcommand{\eqlaw}{\stackrel{\cL}{=}}  
\begin{document}

\title{Spectral measure of
empirical autocovariance matrices \\ of high dimensional Gaussian
stationary processes}

\author{Arup Bose\thanks{Statistics and Mathematics Unit, Indian Statistical
Institute, Kolkata. Email: \texttt{bosearu@gmail.com}.}
\and
Walid Hachem\thanks{
LIGM, CNRS, Univ Gustave Eiffel, ESIEE Paris, F-77454
 Marne-la-Vall\'ee, France.
Email: \texttt{walid.hachem@univ-eiffel.fr}.}}

\date{\today}

\maketitle

\begin{abstract} Consider the empirical autocovariance matrices at given
non-zero time lags, based on observations from a multivariate complex Gaussian
stationary time series. The spectral analysis of these autocovariance matrices
can be useful in certain statistical problems, such as those related to testing
for white noise.  We study the behavior of their spectral measure in the
asymptotic regime where the time series dimension and the observation window
length both grow to infinity, and at the same rate.  Following a general
framework in the field of the spectral analysis of large random non-Hermitian
matrices, at first the probabilistic behavior of the small singular values of
a shifted version of the autocovariance matrix is obtained. This is then
used to obtain the asymptotic behaviour of the empirical spectral
measure of the autocovariance matrices at any lag.  Matrix orthogonal
polynomials on the unit circle play a crucial role in our study.
\end{abstract}




\section{
Background, assumptions and results}
\label{sec-complex} 

\subsection{Background}
Consider a multivariate time-series sequence $(\bs x^{(N)})_{N=1,2,\ldots}$ 
where for each $N \in 
 \NN\setminus\{ 0 \}$, 
the process
$\bs x^{(N)} = (\bs
x^{(N)}_k)_{k\in\ZZ}$ is a $\CC^N$--valued centered Gaussian stationary process
in the discrete time parameter $k$. Let 
$$R^{(N)}_L = \EE \bs x^{(N)}_L (\bs x^{(N)}_0)^*$$ be the \textit{autocovariance matrix}  
of $\bs x^{(N)}$ at lag $L$  (throughout this article, $^{*}$ stands for the conjugate transpose). Let $(n_N)_N$ be an increasing sequence of positive integers 
such that 
\begin{equation}
\label{n/N} 
0 < \liminf_{N\to \infty} \frac{N}{n_N} \leq \limsup_{N\to \infty} \frac{N}{n_N} < \infty.  
\end{equation} 
Assume that for each $N$, we have the sample $\bs x^{(N)}_0,
\ldots, \bs x^{(N)}_{n_N-1}$ of the process $\bs x^{(N)}$.  Fixing an integer 
$L \geq 0$, 
the \textit{empirical autocovariance matrix} of order $L$  is given by 
\[
\widehat R^{(N)}_L = \frac{1}{n_N} \sum_{\ell = 0}^{n_N-1} 
 \bs x^{(N)}_{\ell+L} \bs x^{(N)^*}_{\ell} \quad \in \CC^{N\times N},   
\]
where the sum $\ell+L$ will be taken modulo--$n_N$. 
For any matrix $M \in \CC^{m\times m}$, let $\{ \lambda_0(M), \ldots,
\lambda_{m-1}(M) \}$ be its eigenvalues. The spectral measure of $\widehat R^{(N)}_L$ is then defined as 
\[
\mu_N = \frac 1N \sum_{\ell=0}^{N-1} 
  \delta_{\lambda_\ell(\widehat R^{(N)}_L)}. 
\]
We are interested in studying this measure as $N\to \infty$.  In the field of
multivariate time series analysis, it is classically assumed that $N$ is fixed
while the observation window length increases to $\infty$, in which case, under
standard sets of assumptions, $\mu_N$ converges weakly to the spectral measure
of $R^{(N)}_L$ in the almost sure sense. This is no more true in the asymptotic
regime that we consider in this paper, where the time series dimension and the
window length are both large and of the same order of magnitude.

Note that for $L=0$, $\widehat R^{(N)}_L$ is Hermitian and several results are known for this case under different assumptions. Our aim is to consider the cases $L \geq 1$ in which case $\widehat R^{(N)}_L$ are non-Hermitian. Generally speaking, the study of the spectral measure of non-Hermitian matrices is much harder than that for Hermitian matrices. See for example  \cite{lit-paj-rud-tom-05, rud-ver-advmath08, got-tik-ap10}. As an example, it took a tremendous amount of effort from researchers over a long period of time to establish 
the limit of the empirical spectral measure of the matrix all whose elements are real-valued iid with mean zero and variance 1 
(see \cite{tao-vu-aop10}). 

In \cite{bos-hac-20}, using the ideas from \cite{lit-paj-rud-tom-05, rud-ver-advmath08, ver-14},  
 we identified the limit spectral measure of $\widehat R^{(N)}_L$ in the particular case when  
the time series is a (complex) white noise process. 
The same setting is considered in~\cite{yao-yua-arxiv20}, but they relax the modulo--$n_N$ summation 
when constructing $\widehat R^{(N)}_L$. 
However, when the time series is not a white noise, no such result appears to exist in the literature. 

Spectral properties of the sample autocovariance matrices in a stationary time series carry information on the  
process and hence potentially, can be used for statistical inference. 
Some work in this area were initiated by \cite{BB2014free} and \cite{BB2019}. See also \cite{BB2018} for a book-level exposition. For example, plots of the empirical spectral measure of the sample autocovariance matrices for different lags can serve as graphical tests for white noise, or for the order of dependence in the time series. Some first results in this vein were proposed in \cite{bos-hac-20}. The nature of the empirical spectral measure of $\widehat R^{(N)}_L$ for different values of $L$ reflect the degree of dependence that exists in the underlying process. Theoretical support for these tests rests on the asymptotic behavior of the empirical spectral measure.  Moreover, once this behavior is identified, it can potentially be used to develop significance tests for other statistical hypothesis too.  This issue will be taken up elsewhere. 

For any matrix $M \in \CC^{N\times N}$, let $s_0(M) \geq s_1(M) \geq \cdots
\geq s_{N-1}(M)$ be its singular values arranged in a non-increasing order. It
is well-known that for a non-Hermitian matrix, say $M_N \in
\CC^{N\times N}$, as $N \to \infty$, the behavior of its spectral measure is connected to the
probabilistic behavior of the small singular values of the related matrix $M_N
- z \eqdef M_N - z I_N$ for $z \in \CC\setminus \{ 0 \}$.  With this in mind,
we shall seek solutions to the following problems. 
\begin{itemize}
\item 
The behavior of the smallest singular value  
$s_{N-1}(\widehat R^{(N)}_L - z)$ for an arbitrary $z \in \CC\setminus \{ 0 \}$.

\item For an arbitrary $\beta \in (0,1)$, the behavior of the
``small'' singular values $s_{N-\ell}(\widehat R^{(N)}_L - z)$ for
$\ell\in \{ \lfloor N^\beta\rfloor, \ldots, \lfloor N/2 \rfloor \}$ and $z \in
\CC\setminus \{ 0 \}$. Later this will help in 
controlling the magnitude of the singular values $s_{N-\ell}$ when the indices $\ell$ are close to $N^\beta$. 

\item 
The behavior of $\mu_N$ as $N \to\infty$ via the existence of a \textit{deterministic equivalent}.
\end{itemize}

\subsection{Assumptions}

We shall assume that for every $N$, $(\bs x_k^{(N)})_k$  is a stationary Gaussian process whose spectral density exists and satisfies some reasonable regularity conditions. As we shall see, this provides the opportunity to use a variety of technical tools. If the time series are not Gaussian, then the situation is way more involved technically but results similar to those in this paper are expected to hold under suitable restrictions on the time series. 
Let 
$\TT$ denote the unit circle of the complex
plane. Let $\cH^N_+$ be the set of $N\times N$ Hermitian non-negative matrices. 
Suppose that 
for each positive integer $N$, there is an integrable function $S^{(N)} : \TT \to \cH^N_+$ such that 
for each $L\in\ZZ$, 
\begin{equation}
\label{R-S} 
R^{(N)}_L = \frac{1}{2\pi} \int_0^{2\pi} 
    e^{-\imath L\theta} S^{(N)}(e^{\imath\theta}) \, d\theta. 
\end{equation} 
This $S^{(N)}$ is called the \textit{spectral density} of $\{R^{(N)}_{L}, L
\geq 0\}$ or of the corresponding stationary process
\cite[Chap.~1]{roz-livre67}, \cite{bri-livre01}.  We assume that for each $N$,
$S^{(N)}$ is non-trivial in the sense that for each non-zero $\CC^N$--valued
polynomial $p(z)$, 
\[
\int_0^{2\pi} 
 p(e^{\imath\theta})^* \, S^{(N)}(e^{\imath\theta}) \, p(e^{\imath\theta}) 
  \ d\theta > 0\ \ \text{(that is, the matrix is positive definite)}.
\]
We now turn to the more substantial assumptions on $S^{(N)}$. 
The first assumption is 
akin to uniform equicontinuity of $\{S^{(N)}\}$. For $h > 0$, let 
\[
\bs w(S^{(N)}, h) = \sup_\theta \sup_{|\psi|\leq h} 
 \left\| S^{(N)}(e^{\imath(\theta+\psi)}) - S^{(N)}(e^{\imath\theta}) \right\| 
\]
be the modulus of continuity of $S^{(N)}$ with respect to the spectral norm $\| \cdot \|$.

\begin{assumption}
\label{prop-S}
\begin{enumerate}[label=(\roman*)]
\item\label{S-equi} For any $\varepsilon > 0$, 
 there exists $h > 0$ such that 
\[
 \sup_{N\in \NN} \bs w(S^{(N)}, h) < \varepsilon. 
\]
\item\label{bnd-S} With $\displaystyle{\| S^{(N)} \|_\infty^\TT 
 = \max_{\theta} \| S^{(N)}(e^{\imath\theta}) \|}$, 
$$\displaystyle{\bM := \sup_N \| S^{(N)} \|_\infty^\TT} < \infty.$$ 
\item\label{R(0)} \ 
 $\inf_N s_{N-1}(R^{(N)}_0) > 0$.
\end{enumerate} 
\end{assumption} 
Regarding the last assumption, note that 
$$
s_{N-1}(R^{(N)}_0) = s_{N-1}\Bigl( \frac{1}{2\pi} 
 \int_0^{2\pi} S^{(N)}(e^{\imath\theta}) 
 \, d\theta \Bigr).
$$
We allow the spectral density $S^{(N)}(e^{\imath\theta})$ to be singular or
close to singular at some points of $\TT$, but within the restrictions
provided by the two following assumptions.  Assumption~\ref{S-reg}  stipulates
that $S^{(N)}(e^{\imath\theta})$ can be close to being singular only on a set
of frequencies with a small Lebesgue measure and it implies
Assumption~\ref{prop-S}--\ref{R(0)}.  Assumption~\ref{log-S} puts additional
constraint on $s_{N-1}(S^{(N)}(e^{\imath\theta}))$. Examples where
Assumptions~\ref{S-reg} and~\ref{log-S} are satisfied are provided in
Section~\ref{ex-s} below.  Let $\leb(\cdot)$ denote the Lebesgue measure.
\begin{assumption}
\label{S-reg} 
 Suppose that for any $\kappa \in (0,1)$, there exists $\delta > 0$ such that 
\[
\sup_N \leb\left\{ z \in \TT \, : \, 
  s_{N-1}(S^{(N)}(z)) \leq \delta  \right\} \leq \kappa. 
\]
\end{assumption} 
\begin{assumption}
\label{log-S}
\[
\frac 1N \int_0^{2\pi} \log s_{N-1}(S^{(N)}(e^{\imath\theta})) \ d\theta 
 \xrightarrow[N\to\infty]{} 0 .
\] 
\end{assumption} 

\subsection{Results}

Before we state our results, we wish to recall that we have used modulo--$n_N$
summation to construct $\widehat R^{(N)}_L$. This is for convenience and helps
in the details of the proofs. We believe that the results we establish continue
to hold for the sample autocovariance matrices when we define them via the
usual summation over all indices that maintain a lag $L$.  See for
instance~\cite{yao-yua-arxiv20} who relax the modulo--$n_N$ summation in the
context of the model with i.i.d.  processes considered in~\cite{bos-hac-20}. 
These results can be adapted to our model with some work. 


The first result is on a probabilistic bound on the magnitude of the smallest singular value of $(\widehat R^{(N)}_L - z)$. In problems similar to ours, often the optimal factor at the left hand side of the bound given below
is  $N^{-1} t$ instead of our $N^{-3/2} t$ when $\varepsilon > 0$ is fixed. Our weaker bound will serve our purpose. We shall elaborate on this issue during the course of the proof.

\begin{theorem}
\label{snY}
Suppose Assumptions~\ref{prop-S}, \ref{S-reg} and \ref{log-S} hold. Then, for each 
$z \neq 0$ and arbitrarily small $\varepsilon > 0$, there exists a constant $c_{\ref{snY}}$ such that for 
all small $t > 0$ and for all large $N$,  
\[
\PP\left[ s_{N-1}(\widehat R^{(N)}_L - z) \leq N^{-3/2} t \right] \leq 
 \varepsilon t + \exp(- c_{\ref{snY}} \varepsilon^2 N ). 
\]
\end{theorem} 
The behavior of the small singular values is handled by 
Theorem~\ref{intrm}.
The behavior of $s_{N-k}(\widehat R^{(N)}_L - z)$ for values of $k$ which are close to $N^\beta$ is more important.
The theorem implies that $s_{N-N^\beta}(\widehat R^{(N)}_L - z) \gtrsim N^{\beta/2 - 1}$ with large probability. 
Again, though the rate is not optimal, it will be sufficient for our needs. Further comments on this issue will be provided in the course of the proof,
see Remark~\ref{rem-intrm} below. 
\begin{theorem}
\label{intrm}
Suppose Assumptions~\ref{prop-S}, \ref{S-reg} and \ref{log-S} hold. Let $\beta\in(0,1)$. 
Then, for each $z \neq 0$, there exist two positive constants 
$c_{\ref{intrm}}$ and $C_{\ref{intrm}}$ such that for all 
$k\in[\lfloor N^\beta\rfloor, \lfloor N /2\rfloor]$, and for $N \geq N_0$, where $N_0$ is independent of $k$, 
\[
\PP\left[ s_{N-k-1}(\widehat R^{(N)}_L - z) \leq C_{\ref{intrm}} 
 \sqrt{k} / N \right] \leq \exp(- c_{\ref{intrm}} k ). 
\]
\end{theorem} 

We now turn to the 
large-$N$ behavior of 
$\mu_N$.
For this we rely on the well-known Hermitization
technique due to Girko~\cite{gir-84} (see~\cite{bor-cha-12} for a comprehensive exposition). Let  $\mu$ 
 be a probability measure on $\CC$ that integrates $\log|\cdot |$ near infinity. Then its log-potential $U_\mu(\cdot): \CC\to(-\infty,\infty]$ is defined below. 
The measure $\mu$ can be recovered from $U_\mu(z)$. 
\[
U_\mu(z)= - \int_\CC \log | w - z | \ \mu(dw) . 
\]
For the empirical spectral measure $\mu_N$,
we can write 
\begin{align*} 
U_{\mu_N}(z) &= 
 - \frac 1N \sum_{\ell=0}^{N-1} \log | \lambda_\ell(\widehat R^{(N)}_L) - z |
  = - \frac{1}{2N} \log \det (\widehat R^{(N)}_L - z) (\widehat R^{(N)}_L-z)^*
 \\ 
   &= - \int \log t \ \nu_{z,N}(dt), 
\end{align*} 
where 
$\nu_{z,N}$, the empirical measure (on $\mathbb{R}$) of the singular values of the matrix $\widehat R_L^{(N)} - z$ 
given by 
\[
\nu_{z,N} = \frac{1}{N} \sum_{\ell=0}^{N-1} 
  \delta_{s_\ell(\widehat R_L^{(N)} - z)}. 
\]
 Given a matrix $M \in \CC^{N\times N}$, denote 
its so-called \emph{Hermitized} version as 
\[
\bH(M) = \begin{bmatrix}  & M \\ 
  M^*  &  \end{bmatrix}. 
\]
 As is well-known, the spectral
measure $\check\nu_{z,N}$ (on $\mathbb{R}$)  
of 
$\bH(\widehat R_L^{(N)} - z)$ is given by 
\[
\check\nu_{z,N} = \frac{1}{2N} \sum_{\ell=0}^{N-1} 
 \left( \delta_{s_\ell(\widehat R_L^{(N)} - z)} + 
 \delta_{-s_\ell(\widehat R_L^{(N)} - z)} \right) . 
\]
The measure $\check\nu_{z,N}$ is clearly symmetric in the sense that
$\check\nu_{z,N}(B) = \check\nu_{z,N}(-B)$ for each Borel set $B \subset \RR$, 
and it is the symmetrized version of $\nu_{z,N}$.
 Observe that $\int \log t \ \nu_{z,N}(dt) = \int \log| t| \
\check\nu_{z,N}(dt)$. 

Note that all of these are \textit{random measures}. To study the 
behavior of $\mu_N$, we will find it  more convenient to study $\check\nu_{z,N}$ instead of $\nu_{z,N}$,
since the matrix $\bH(\widehat R_L^{(N)} - z)$ is Hermitian.  
This 
approach 
is formalized in the following general proposition which provides conditions under which 
the sequence of random measures $(\mu_N)$ can be approximated by some sequence
$(\bs\mu_N)$ of deterministic measures. 
\begin{proposition}
\label{herm} 
Assume that for almost every $z\in\CC$, the following two conditions hold: 
\begin{enumerate}
\item\label{ui} 
With probability one, $\log |\cdot|$ is \textit{uniformly integrable} with respect to 
$\{\check\nu_{z,N}\}_N$. 
\item\label{cvg-sing} 
There exists a \textit{tight} sequence of \textit{deterministic} symmetric probability measures 
$(\bs{\check\nu}_{z,N})_N$ on $\RR$ such that for each bounded and continuous 
function $f:\RR\to\RR$, 
\[
\int f \, d\check\nu_{z,N} - \int f \, d\bs{\check\nu}_{z,N} 
 \xrightarrow[N\to\infty]{\text{a.s.}} 0 .
\]
\end{enumerate} 
Then, there exists a tight sequence of deterministic probability measure 
$(\bs\mu_N)$ on $\CC$ such that for each bounded and continuous function 
$f:\CC\to\CC$, we have  
\[
\int f \, d\mu_{N} - \int f \, d\bs{\mu}_{N} 
 \xrightarrow[N\to\infty]{\text{a.s.}} 0 .
\]
Moreover, the logarithmic potential of $\bs\mu_N$ is 
\[
U_{\bs\mu_N}(z) = - \int \log|t| \ \bs{\check\nu}_{z,N}(dt).
\]
\end{proposition} 
This proposition is very close to \cite[Lem.~4.3]{bor-cha-12} (see also 
\cite{coo-hac-naj-ren-18}). The proof of this lemma can be adapted to our 
situation by considering converging subsequences of $(\bs{\check\nu}_{z,N})_N$. 

The uniform integrability condition needed in Proposition~\ref{herm} is ensured by
Theorems~\ref{snY} and~\ref{intrm} (again, see~\cite[Sec.~4.2]{bor-cha-12} for 
the proof details).  We are thus left to consider the asymptotic properties of
$\check\nu_{z,N}$ with a goal to comply with  Condition \ref{cvg-sing} of
Proposition~\ref{herm}. Classically, the central object that is used for this
is the resolvent (in what follows $z\in
\CC$ is arbitrary) 
\[
Q^{(N)}(z,\eta) 
= \left( \bH(\widehat R^{(N)}_L-z) - \eta I_{2N} \right)^{-1} 
\]
of $\bH(\widehat R_L^{(N)} - z)$ in the complex 
variable $\eta\in \CC_+ \eqdef \{ w \in \CC \, : \, \Im w > 0 \}$. The 
resolvent $Q^{(N)}(z,\cdot)$ is a typical example of a so-called 
matrix Stieltjes transform. Before we recall its nature, let us recall that the Stieltjes transform of any 
probability measure $\nu$ on $\RR$ is given by  
\[
g_\nu(\eta) = \int_\RR \frac{1}{\lambda - \eta} \nu(d\lambda), \quad 
 \eta \in \CC\setminus\RR .
\]

The real and imaginary parts of a square matrix $M$ are respectively given as 
\[
\Re M = \frac{M + M^*}{2} \quad \text{and} \quad 
 \Im M = \frac{M - M^*}{2\imath} .
\]
Given an integer $m > 0$, we let $\cM_+^m$ denote the set of matrices $M \in
\CC^{m\times m}$ such that $\Im M > 0$.  

The following result can be found in
the literature dealing with the moment problem and related topics, see,
\emph{e.g.}, \cite[Pages 64-65]{bol-97}, \cite{ges-tse-00},  
\cite[Prop.~2.2 and Appendix A]{hachem-loubaton-najim07}. 
\begin{proposition}[matrix Stieltjes transform]
\label{mv-st}
Let $F : \CC_+ \to \CC^{m\times m}$ be a matrix-valued function. Then, the
following facts are equivalent. 
\begin{enumerate} 
\item $F$ is the Stieltjes transform of an $\cH_+^m$--valued 
measure $\mu$ on $\RR$ such that $\mu(\RR) = I_m$. 
\item $F$ is analytic, $F(\eta) \in \cM_+^m$ for $\eta\in\CC_+$, and 
$-\imath t F(\imath t)$ converges to $I_m$ as $t\to\infty$. 
\end{enumerate} 
\end{proposition} 

Such an $F$ is called a \textit{matrix Stieltjes transform}. Let $\mathfrak
S^m$ denote the set of all such $m \times m$ matrix Stieltjes transforms.  It
is known that if $F\in\mathfrak S^m$, then \\

(1) \  
$\| F(\eta) \| \leq 1 / \Im\eta$. \\

(2) \  $m^{-1} \tr F(\eta) = g_\nu(\eta)$ for 
some probability measure $\nu$ on $\RR$. \\

An illustration of (2) is 
provided by the resolvent we just defined: it can be checked that  
$$Q^{(N)}(z,\cdot) \in \mathfrak S^{2N}\ \ \text{and}\ \  
g_{\check\nu_{z,N}}(\eta) = (2N)^{-1} \tr Q^{(N)}(z,\eta).$$ 

In the remainder of this paper, whenever we write  $M \in \CC^{2N \times 2N}$ as $M = \begin{bmatrix} M_{00} & M_{01} \\ M_{10} & M_{11}\end{bmatrix}$, it is understood that the blocks $M_{uv}$ belong to $\CC^{N\times N}$. With
this notation, we define the linear operator 
$\mcT : \CC^{2N\times 2N} \to \CC^{2\times 2}$ as 
\[
\mcT\left(\begin{bmatrix} M_{00} & M_{01} \\ M_{10} & M_{11}\end{bmatrix} 
  \right) 
 = \begin{bmatrix} \tr M_{00} / n_N & \tr M_{01} / n_N \\
  \tr M_{10} / n_N & \tr M_{11} / n_N \end{bmatrix} .
\]
Given an integer $L > 0$, we also define the $2\times 2$ Hermitian unitary 
matrix function $U_L$ on $\TT$ as  
\[
U_L(e^{\imath\theta}) = \begin{bmatrix} & e^{-\imath L \theta} \\
  e^{\imath L \theta} & \end{bmatrix} .
\]
The next technical result is on the existence of a unique solution for a functional equation. It will be used to show that 
$(\bs{\check\nu}_{z,N})_N$ approximates $(\check\nu_{z,N})_N$. Denote as
$\otimes$ the Kronecker product between matrices. 

\begin{theorem}
\label{sys} 
Let $\Sigma : \TT \to \cH_+^N$ be a continuous function, and let $z\in\CC$.
Given a function $M(\eta) \in \mathfrak S^{2N}$, the function displayed below 
is well-defined and belongs to $\mathfrak S^{2N}$ as a function of $\eta$.
\begin{multline*} 
\cF_{\Sigma,z}(M(\eta), \eta) \\ = 
\Bigl( \frac{1}{2\pi} \int_0^{2\pi} 
\left( \mcT( (I_2\otimes \Sigma(e^{\imath\theta})) M(\eta) ) + 
 U_L(e^{\imath\theta}) \right)^{-1} 
    \otimes \Sigma(e^{\imath\theta}) \ d\theta  
- \begin{bmatrix} \eta & z \\ \bar z & \eta \end{bmatrix} 
  \otimes I_N \Bigr)^{-1}. 
\end{multline*} 
 Moreover, the 
functional equation in the parameter $\eta\in\CC_+$ 
\begin{equation}
\label{impl} 
P(z,\eta) = \cF_{\Sigma,z}(P(z,\eta), \eta)
\end{equation} 
admits a \textit{unique} solution in the class 
$P(z,\cdot) \in \mathfrak S^{2N}$. Write 
$$P(z,\cdot) = 
\begin{bmatrix} P_{00}(z,\cdot) & P_{01}(z,\cdot) \\
                P_{10}(z,\cdot) & P_{11}(z,\cdot) \end{bmatrix} \ \ \text{and}\ \ \Lambda(dt) = 
  \begin{bmatrix} \Lambda_{00}(dt) & \Lambda_{01}(dt) \\ 
                 \Lambda_{10}(dt) & \Lambda_{11}(dt) \end{bmatrix}$$ 
where $\Lambda(dt)$ is the matrix measure whose matrix Stieltjes transform is $P(z,\cdot)$, where 
$P_{ii}(z,\cdot)$ is the Stieltjes transform of $\Lambda_{ii}(dt)$. The positive matrix measures $\Lambda_{00}$ and $\Lambda_{11}$ are 
symmetric. 
\end{theorem} 

One consequence of this theorem is that the probability measure $\zeta$ on
$\RR$ such that $g_\zeta(\eta) = (2N)^{-1} \tr P(z,\eta)$ is symmetric. In
fact, it can be proved that 
\[
g_\zeta(\eta) = N^{-1} \tr P_{00}(z,\eta) =N^{-1} \tr P_{11}(z,\eta).
\]   
The first part of the next theorem claims the tightness of
$(\bs{\check\nu}_{z,N})_N$.  Later we shall need the behavior of the difference
$\left(Q^{(N)}(z,\eta) - G^{(N)}(z,\eta) \right)$. This is captured in the
second part. The extra generality obtained by including the matrix $D$ is not
necessary for the purposes of the present work in which we use only the
specific choice of $D=I$. However this will be useful to derive convergence and
related properties of $(\mu_{N})_N$ and $(\bs{\mu}_{N})_N$ that we wish to
pursue in future.

\begin{theorem}
\label{eqdet} 
Suppose Assumption~\ref{prop-S} holds. 
Let $G^{(N)}(z,\eta)$ be the solution in $\mathfrak S^{2N}$, of the equation 
\[
G^{(N)}(z,\eta) = \cF_{S^{(N)},z}(G^{(N)}(z,\eta), \eta),  
\]
as specified by Theorem~\ref{sys}. 
Let $\bs{\check\nu}_{z,N}$ be the symmetric
probability measure on $\RR$ whose Stieltjes transform is 
\[
g_{\bs{\check\nu}_{z,N}}(\eta) = \frac{1}{2N} \tr G^{(N)}(z,\eta) . 
\]
Then, the sequence $(\bs{\check\nu}_{z,N})_N$ is tight. 

Let $D^{(N)} \in \CC^{2N \times 2N}$ be an arbitrary deterministic matrix such 
that $\| D^{(N)} \| = 1$. Then, 
\begin{equation}
\label{cvg-Q} 
\forall\eta\in\CC_+, \quad 
\frac{1}{2N} \tr D^{(N)} \left(Q^{(N)}(z,\eta) - G^{(N)}(z,\eta) \right) 
 \xrightarrow[N\to\infty]{\text{a.s.}} 0.  
\end{equation} 
\end{theorem} 

Taking $D^{(N)} = I$, this theorem shows in particular that
$g_{\check\nu_{z,N}}(\eta) - g_{\bs{\check\nu}_{z,N}}(\eta)
\xrightarrow[N\to\infty]{\text{a.s.}} 0$.  Since $(\bs{\check\nu}_{z,N})_N$ is
tight, we have the following corollary, whose proof employs well-known facts
about scalar Stieltjes transforms (see, \emph{e.g.}, \cite{pas-livre},
\cite[Sec.~2]{hachem-loubaton-najim07}), and is omitted. 
\begin{corollary}
\label{cor-st} 
The sequence $(\check\nu_{z,N})_{N}$ of probability measures is tight
with probability one. Furthermore, for each bounded and continuous function 
$f : \RR \to \RR$, we have  
\[
\int f d\check\nu_{z,N} - \int f d\bs{\check\nu}_{z,N} 
 \xrightarrow[N\to\infty]{\text{a.s.}} 
 0 .
\] 
\end{corollary} 

We can now conclude by characterizing the asymptotic behavior of $(\mu_N)$.
Theorems~\ref{snY} and~\ref{intrm} provide the uniform integrability condition
stated in Proposition~\ref{herm}--(\ref{ui}), see, \emph{e.g.}, the derivations
made in \cite[Sec.~4.2]{bor-cha-12}. Condition (\ref{cvg-sing}) in the
statement of Proposition~\ref{herm} is ensured by Theorem~\ref{eqdet} via
its Corollary~\ref{cor-st}. Therefore, thanks to 
Proposition~\ref{herm}, we obtain: 
\begin{theorem}
\label{msl} 
Suppose Assumptions~\ref{prop-S} to \ref{log-S} hold. Then, there exists
a tight sequence of deterministic probability measures $(\bs\mu_N)$ on 
$\CC$ such that, for each bounded and continuous function $f:\CC\to\CC$,
\[
\int f \, d\mu_{N} - \int f \, d\bs{\mu}_{N} 
 \xrightarrow[N\to\infty]{\text{a.s.}} 0. 
\]
 The function 
$\log|\cdot|$ is integrable with respect to $\bs{\check\nu}_{z,N}$ for 
each $z\neq 0$, and the measure $\bs\mu_N$ is defined by its logarithmic 
potential through the identity 
\[
U_{\bs\mu_N}(z) = - \int \log|t| \ \bs{\check\nu}_{z,N}(dt). 
\]
\end{theorem}

A natural issue is if can we say anything further about the behavior of
$(\bs\mu_N)$. In particular, its weak convergence in turn would ensure the weak
convergence of $(\mu_N)$.  An idea that goes back to~\cite{fei-zee-97} and
which has been frequently used in the literature of non-Hermitian matrices,
connects the behaviour of $(\bs\mu_N)$ to that of $N^{-1} \tr
G_{01}^{(N)}(z,\imath t)$ as $t \searrow 0$. In the present context  the
details need some significant efforts. Since this is outside the focus of this
work, it will be pursued separately.

Incidentally, 
the above idea was applied in \cite{bos-hac-20} to the particular
case where $S^{(N)}(e^{\imath\theta}) = I_N$. The connection of the present
paper with \cite{bos-hac-20} is provided by the following corollary to
Theorem~\ref{msl}. This corollary will be proved by showing that $N^{-1} \tr
G_{01}^{(N)}(z,\imath t)$ coincides with the analogue of this quantity that was
computed in \cite{bos-hac-20}. 
\begin{corollary}
\label{bh-20} 
Assume that $S^{(N)}(e^{\imath\theta}) = I_N$,  
$N/ n_N \to \gamma > 0$ as $N\to\infty$, and $L = 1$. Then the random sequence 
$(\mu_N)$ converges weakly in the almost sure sense to the deterministic 
probability measure $\bs\mu$ on $\CC$ described by \cite[Th.~2]{bos-hac-20}. 
\end{corollary}

The remainder of the paper is organized as follows.  Examples where
Assumptions~\ref{S-reg} and~\ref{log-S} hold, are provided in
Section~\ref{ex-s}.  Theorems~\ref{snY} and~\ref{intrm} are proved in
Section~\ref{prf-snY}. The results related to the behavior of the singular
values, \emph{i.e.}, Theorems~\ref{sys} and~\ref{eqdet}, are proved in 
Section~\ref{nu-N}. The proof of Corollary~\ref{bh-20} will be sketched in
this section as well.

\subsection*{Notation} The indices of the elements of a vector or a matrix start from zero. Given two integers $k$ and $m$, we write $[k:m] = \{k, k+1, \ldots, m-1\}$, this set being empty if $k\geq m$. We also write $[m] = [0:m]$. Assume $m > 0$. For $k\in[m]$, we denote as $e_{m,k}$ the $k^{\text{th}}$ canonical basis vector of $\CC^m$.  When there is no ambiguity, we write $e_k$ for $e_{m,k}$.  Given a matrix $M \in \CC^{m\times n}$, and two sets $\cI \subset [m]$ and $\cJ \subset [n]$, we denote as $M_{\cI,\cJ}$ the $|\cI| \times |\cJ|$ sub-matrix of $M$ that is obtained by retaining the rows of $M$ whose indices belong to $\cI$ and the columns whose indices belong to $\cJ$. We also write $M_{\cdot, \cJ} = M_{[m], \cJ}$ and $M_{\cI, \cdot} = M_{\cI, [n]}$.  Given a vector $v\in\CC^m$, we denote by $v_\cI$ the $\CC^{|\cI|}$ sub-vector obtained by keeping the elements of $v$ whose indices are in $\cI$. 

For any  matrix $M$,  $M > 0$  means that it is positive definite. The column span of a matrix $M$ is denoted by $\colspan(M)$.  The orthogonal projection matrix onto $\colspan(M)$ (respectively, onto the subspace orthogonal to $\colspan(M)$) is denoted by $\Pi_M$ (respectively, $\Pi^\perp_M$).
The spectral norm of a matrix and the Euclidean norm of a vector are denoted by 
$\|\cdot\|$.  The Hilbert-Schmidt norm of an operator will be denoted $\| \cdot
\|_\HS$. The notation $M \geq G$ where $M$ and $G$ are Hermitian matrices
refers to the semi-definite ordering of such matrices. 

We write $\RR_+ = [0,\infty)$.  Suppose $B$ is a Borel set of $\RR^d$. Then
$\leb(B)$ denotes its Lebesgue measure.  For any $x\in \RR^d$, $\dist(x, B)=\inf_{y\in B}||x-y||$.  
 For $B$ in a metric space $E$, $\cV_\rho^E(B)$ denotes its closed $\rho$--neighborhood. If the
underlying space $E$ is clear, we simply write $\cV_\rho(B)$.

The unit-sphere of $\CC^m$ will be denoted as $\SSS^{m-1}$. The set of vectors of $\SSS^{m-1}$ that are supported by
the (index) set $\cI \subset [m]$ will be denoted by $\SSS^{m-1}_\cI$.

The probability and the expectation with respect to the law of the vector $x$
will be denoted by $\PP_x$ and $\EE_x$.
The centered  and circularly symmetric complex Gaussian  distribution with covariance matrix $\Sigma$
will be denoted by $\cN_\CC(0, \Sigma)$.  


\section{Examples where Assumptions~\ref{S-reg} and~\ref{log-S} 
 are satisfied} 
\label{ex-s} 
In this section, we provide examples where the 
spectral density 
satisfies Assumptions~\ref{S-reg}
and~\ref{log-S}. Consider the \textit{moving average} $MA(\infty)$ model:
\begin{equation}
\label{ma} 
\bs x_k^{(N)} = \bs\xi_k^{(N)} 
   + \sum_{\ell\geq 1} A^{(N)}_\ell \bs\xi^{(N)}_{k-\ell} , 
\end{equation} 
where $\bs\xi^{(N)} = (\bs\xi^{(N)}_{k})_{k\in\ZZ}$ is an i.i.d.~process with
$\bs \xi^{(N)}_{k} \sim \cN_\CC(0, I_N)$ and  $(A^{(N)}_\ell)_{\ell\geq 1}$ is a
sequence of deterministic matrices which satisfy the minimal requirement 
$\sum_{\ell\geq 1} \| A^{(N)}_\ell \| < \infty$. The spectral density for 
this model is 
\[
S^{(N)}(e^{\imath\theta}) = \Bigl( I_N + \sum_{\ell\geq 1} 
  e^{\imath\ell\theta} A^{(N)}_\ell \Bigr) 
\Bigl( I_N + \sum_{\ell\geq 1} 
  e^{-\imath\ell\theta} (A^{(N)}_\ell)^* \Bigr) .
\]
Let us look at some particular cases: \vskip5pt

 \noindent 
(i) First suppose that 
\[
\limsup_N \sum_{\ell\geq 1} \| A^{(N)}_\ell \| < 1. 
\]
Then it is obvious that both Assumptions~\ref{S-reg} and~\ref{log-S} are satisfied and $s_{N-1}(S^{(N)}(z))$ remains bounded away from zero when $z$ runs through $\TT$. 
\vskip5pt

\noindent 
(ii) Now consider the MA($1$) case, so that $A^{(N)}_\ell = 0$
for $\ell \geq 2$, and suppose that we only have $\sup_N \| A^{(N)}_1 \| < \infty$. In
this case, Assumptions~\ref{S-reg} and~\ref{log-S}, written in terms of  $A^{(N)}_1$ are, 
respectively (the $\log$ in the second expression is always integrable),
\begin{subequations} 
\label{ma1} 
\begin{equation} 
\forall\kappa \in (0,1), \exists \delta > 0, 
\sup_N \leb\left\{ z \in \TT \, : \, 
  \left\| (z - A^{(N)}_1)^{-1} \right\| \geq 1/\delta  \right\} \leq \kappa, 
  \quad \text{and} \label{ma1-reg} 
\end{equation}
\begin{equation} 
\frac 1N \int_0^{2\pi} \log 
  \Bigl\| \Bigl(e^{\imath\theta} - A^{(N)}_1 \Bigr)^{-1} \Bigr\| \ d\theta 
 \xrightarrow[N\to\infty]{} 0. \label{ma1-log}  
\end{equation}
\end{subequations} 
These conditions
are closely connected with the \emph{pseudospectrum} of $A^{(N)}_1$
\cite{tre-emb-livre05}. The Toeplitz matrices are among the matrices for which the pseudospectra are well-understood. Suppose that $A^{(N)}_1 = \left[ a^{(N)}_{k-\ell} \right]_{0\leq k,\ell\leq N-1}$ is a Toeplitz matrix with the so-called \textit{symbol}
\[
f_N(z) = \sum_{k= 1-N}^{N-1} a^{(N)}_k z^k. 
\]
Then, by the Brown and 
Halmos theorem~\cite[Th.~4.29]{bot-gru-livre05}, Assumptions~\ref{S-reg}
and~\ref{log-S} are respectively satisfied if 
\begin{gather*} 
\forall\kappa \in (0,1), \exists \delta > 0, 
\limsup_N \leb\left\{ z \in \TT \, : \, 
  \dist( e^{\imath\theta}, \conv f_N(\TT) ) \leq\delta  \right\} \leq \kappa, 
  \quad \text{and} \\ 
\frac 1N \int_0^{2\pi}  \log \dist( e^{\imath\theta}, 
  \conv f_N(\TT) ) \ d\theta \xrightarrow[N\to\infty]{} 0 ,
\end{gather*} 
where $\conv$ denotes the convex hull of a set. 
\vskip5pt

\noindent (iii) Another particular case of the model in ~\eqref{ma} is the following.  Assume
that $N=MK$ where $M$ and $K$ are two positive integers. Assume that the
matrices $A^{(N)}_\ell$ are block-diagonal matrices of the form
\[
A^{(N)}_\ell = I_M \otimes B^{(K)}_{\ell} 
\]
where $B^{(K)}_{\ell} \in \CC^{K\times K}$. We also assume that $\bmax =
\sup_{K} \sum_{\ell\geq 1} \| B^{(K)}_{\ell} \| < \infty$. The process $\bs\xi^{(N)}$ consists of $M$ i.i.d.~streams of stationary
processes each of dimension $K$. Let us show that if $M \to\infty$, then
Assumption~\ref{log-S} is satisfied for this process. Further if $K$ is upper
bounded, then Assumption~\ref{S-reg} is satisfied. 
 
The spectral density of each stream is $\Sigma^{(K)}(e^{\imath\theta}) =
C^{(K)}(e^{\imath\theta}) C^{(K)}(e^{\imath\theta})^*$, with $C^{(K)}(z) = I_K
+ \sum_{\ell\geq 1} z^\ell B^{(K)}_{\ell}$ on the unit-disk. It is well known
that (see, \emph{e.g.}, \cite[Th.~6.1]{roz-livre67}), 
\begin{align*}
\frac{1}{2\pi} \int_0^{2\pi} 
 \left| \log\left| \det C^{(K)}(e^{\imath\theta}) \right|^{\frac 1K}
 \right| \, d\theta &\leq 
 \frac{1}{2\pi} \int_0^{2\pi} 
  \left( \det \Sigma^{(K)}(e^{\imath\theta}) \right)^{\frac 1K} d\theta 
 - \log \left| \det C^{(K)}(0) \right|^{\frac 1K} \\
&= \frac{1}{2\pi} \int_0^{2\pi}
  \left( \det \Sigma^{(K)}(e^{\imath\theta}) \right)^{\frac 1K} d\theta. 
\end{align*} 
 By Hadamard's inequality, the 
right hand side is bounded by $(1 + \bmax)^2$. Thus, 
\begin{align*}
| \log s_{K-1}( \Sigma^{(K)}(e^{\imath\theta})) | &\leq 
\left| \sum_{\ell=0}^{K-1} \log s_\ell( \Sigma^{(K)}(e^{\imath\theta})) 
 \1_{s_\ell( \Sigma^{(K)}(e^{\imath\theta})) \leq 1}\right| + \log (1+\bmax)^2 \\ 
&\leq \left| \sum_{\ell=0}^{K-1} 
   \log s_\ell( \Sigma^{(K)}(e^{\imath\theta})) \right| 
  + (K+1) \log (1+\bmax)^2 \\
&= 2 \left| \log\left| \det C^{(K)}(e^{\imath\theta}) \right| \ \right| 
  + (K+1) \log (1+\bmax)^2, 
\end{align*}
and we deduce from the last display that 
\[
\frac{1}{2\pi} \int_0^{2\pi} 
| \log s_{K-1}( \Sigma^{(K)}(e^{\imath\theta})) | \ d\theta 
\leq  2K (1+\bmax)^2 + (K+1) \log (1+\bmax)^2 . 
\]
Since $S^{(N)}(e^{\imath\theta}) = I_M \otimes \Sigma^{(K)}(e^{\imath\theta})$,
Assumption~\ref{log-S} is satisfied when $M\to\infty$. Furthermore, given a
small $\delta > 0$, we have 
\begin{align*} 
\frac{1}{2\pi} \int_0^{2\pi} 
 \1_{s_{K-1}( \Sigma^{(K)}(e^{\imath\theta})) \leq \delta} \ d\theta 
 &\leq 
\frac{1}{|\log\delta |} \frac{1}{2\pi} \int_0^{2\pi} 
| \log s_{K-1}( \Sigma^{(K)}(e^{\imath\theta})) | \ d\theta \\ 
 &\leq \frac{2K (1+\bmax)^2 + (K+1) \log (1+\bmax)^2}{|\log\delta |}. 
\end{align*} 
This shows that Assumption~\ref{S-reg} is satisfied if $K$ is bounded.


\section{Small singular values: Proofs of Theorems~\ref{snY} and~\ref{intrm}}
\label{prf-snY}

\subsection{Outline of the proofs} 
\label{outline} 

In the sequel we shall most often omit the dependency on $N$ in the notation
for simplicity. Writing $n = n_N$, define the matrix 
\[
X = 
 \begin{bmatrix}  x_0 & \cdots & x_{n-1} \end{bmatrix} = 
n^{-1/2} \begin{bmatrix} \bs x_0 &
\cdots & \bs x_{n-1} \end{bmatrix} \in \CC^{N\times n} 
\] 
with $x_k \in \CC^N$ being the $k^{\text{th}}$ column of $X$, and consider the 
$n\times n$ circulant matrix 
\[
J = \begin{bmatrix}  
0 &        &        &  1 \\
1 & \ddots &             \\
  & \ddots & \ddots      \\
  &        & 1      &  0  
\end{bmatrix}.  
\] 
Then, the sample autocovariance matrix $\widehat R_L$ can be rewritten as 
$\widehat R_L = X J^L X^*$. Let the so-called $n\times n$ \textit{Fourier matrix} 
$\sF$ be defined as 
\begin{equation}
\label{four} 
\sF = 
\frac{1}{\sqrt{n}} \begin{bmatrix}
 \exp(2\imath \pi k\ell/n) \end{bmatrix}_{k,\ell=0}^{n-1}. 
\end{equation} 
Then $J = \sF \Omega \sF^*$, where 
$$\Omega = \diag( \omega^k )_{k=0}^{n-1}\ \ \text{and}\ \ \omega = \exp(-2\imath\pi / n).$$ 
Furthermore, for $k \in [n]$, let 
\[
y_k = \frac{1}{\sqrt{n}} \sum_{\ell=0}^{n-1} e^{2\imath\pi k\ell} x_\ell  
\]
be the \textit{discrete Fourier transform} of the finite sequence 
$(x_0, \ldots, x_{n-1})$, and define the $N\times n$ matrix 
\[
Y = \begin{bmatrix} y_0 & \cdots & y_{n-1} \end{bmatrix} 
 = X {\sF} 
  \in \CC^{N\times n} . 
\]
Then, we obviously have  
\[
\widehat R_L = Y \Omega^L Y^*.  
\]
Note that the columns $y_k$ of $Y$ are ``almost'' independent, since they are the discrete Fourier transforms applied to a time
window of a Gaussian stationary process. This is why we shall heavily rely on the above expression of $\widehat R_L$.
Let us elaborate on this point before entering the core of the proof.
Write $$y_k = \tilde y_k + \check y_k,\ \  \check y_k = \EE[
y_k \, | \, Y_{k} ],$$ the latter being the conditional expectation with respect to the
$\sigma$--field $\sigma(Y_k)$ generated by the elements of the matrix $Y_k = \Bigl[ y_0 \cdots
y_{k-1} \ y_{k+1} \cdots y_{n-1} \Bigr] \in \CC^{N\times(n-1)}$. Write $S_k =
S(e^{2\imath\pi k/n})$ for brevity, where we recall that $S(e^{\imath\theta})$
is the matrix spectral density. Due to the Gaussianity, $\tilde
y_k$ and $Y_{k}$ will be independent, and we shall prove that 
$$\EE[
\tilde y_k \tilde y_k^* ] \simeq \EE[ y_k y_k^*] \simeq S_k.$$ 
For this, as is frequently done in estimation theory for stationary processes, we
shall make use of the orthogonal matrix polynomial theory with respect to the
matrix measure $(2\pi)^{-1} S(e^{\imath\theta}) d\theta$ on $\TT$. 
Assumption~\ref{log-S}, which is reminiscent of the notion of \emph{regular
measures} found in this literature \cite{sta-tot-livre92, sim-livre-11}, will 
play a major role in our analysis that will be presented in 
Section~\ref{orth-pol}. 

We now consider the proof of Theorem~\ref{snY}, starting with a well-known 
linearization trick. Write 
\[
H = \begin{bmatrix} \Omega^{-L} & Y^* \\ Y & z \end{bmatrix} 
  \in \CC^{(N+n)\times(N+n)}. 
\]
By using the well known formula for the inverse of a partitioned matrix 
that involves the Schur complements (see~\cite[\S 0.7.3]{HorJoh90}), it follows 
that $\| (Y \Omega^L Y^* - z)^{-1} \| \leq \| H^{-1} \|$.
Hence \begin{equation}
\label{lin} 
s_{N+n-1}(H) \leq s_{n-1}(Y \Omega^L Y^* - z), 
\end{equation} 
and the problem then is to control $s_{N+n-1}(H)$. 
A similar problem, considered in~\cite{ver-14}, was that of the smallest
singular value of a symmetric random matrix with iid elements above the 
diagonal, see also~\cite{ngu-12}. Even though our matrix is quite different 
from theirs, we borrow many of their ideas as well of their predecessors, such as \cite{rud-ver-advmath08}. 

First, it can be shown that for $C > 0$ large enough, $\PP[ \| Y \| \geq C ]$ is
 exponentially small (Lemma~\ref{spec_norm}) by using some standard Gaussian calculations. 
Then, as in these articles, the control over $s_{N+n-1}(H)$ will be
provided on the event $[ \| Y \| \leq C ]$ .
More specifically, we will prove that, for some constant $c > 0$,  
\[
\PP\left[ \left[ s_{N+n-1}(H)  
  \leq N^{-3/2} t \right] \cap \left[ \| Y \| \leq C \right] \right] 
 \leq \varepsilon t + \exp(- c \varepsilon^2 n ). 
\]

The smallest singular value of $H$ can be obtained from the variational 
characterization 
\[
s_{N+n-1}(H) = \inf_{u \in \SSS^{N+n-1}} \| H u \| .
\] 
A well-established method to control the smallest singular value
of a random matrix is to study the action of this matrix on the
so-called \textit{compressible} and \textit{incompressible} vectors
\cite{lit-paj-rud-tom-05, rud-ver-advmath08}. 
Let $\theta, \rho \in (0,1)$ be fixed. A vector in $\SSS^{m-1}$ is called 
$\theta$-\textit{sparse} if it does not have more than $\lfloor \theta m \rfloor$
non-zero elements. Given $\theta, \rho \in (0,1)$, we define the set of
$(\theta,\rho)$-compressible vectors as
\[
\comp(\theta,\rho) = \SSS^{m-1} \cap 
\bigcup_{\substack{\cI \subset [m] \\ 
|\cI| = \lfloor \theta m \rfloor}} 
  \cV_\rho^{\CC^m}( \SSS_{\cI}^{m-1} ). 
\]
In other words, this is the set of all unit vectors at a distance less or
equal to $\rho$ from the set of the $\theta$-sparse unit vectors. The set
$\incomp(\theta,\rho)$ of $(\theta,\rho)$-incompressible vectors is the
complementary set $\SSS^{m-1} \setminus \comp(\theta,\rho)$.

Getting back to the variational characterization, and writing 
$u = [ v^\T \ w^\T ]^\T$ with $v \in \CC^n$, we (roughly) define the set 
\[
\cS = \{ u \in \SSS^{N+n-1} \, : \, 
 \| v \| \leq \ \text{a constant} \quad \text{or} \quad 
 v / \| v \| \ \text{is compressible} \} , 
\]
and we write 
\begin{equation}
\label{comp-incomp} 
s_{N+n-1}(H) = \inf_{u \in \cS} \| Hu \| \wedge 
              \inf_{u \in \SSS^{N+n-1}\setminus\cS} \| Hu \| . 
\end{equation} 
For the first infimum, we focus on the component $v$ of
the vector $u$ because $v$ impacts the first $n$ columns of $H$ which are
nearly independent. Relying on the decomposition $y_k = \tilde y_k + \check
y_k$, we first show that $\PP[ \| Hu \| \leq c ]$ is exponentially small when
$v / \| v \|$ is a sparse vector, and then we complete the analysis by an
$\varepsilon$--net argument. 

The second infimum 
requires other arguments.
Let $h_k$ be the $k^{\text{th}}$ column of $H$, and so $H_k = \Bigl[ h_0
\cdots h_{k-1} \ h_{k+1} \cdots h_{N+n-1} \Bigr] \in \CC^{(N+n)\times(N+n-1)}$.
Following a by now well-known idea of \cite[Lem.~3.5]{rud-ver-advmath08}, the
infimum over the incompressible vectors can be controlled by managing the
distances $\dist(h_k, H_k)$ between $h_k$ and $\colspan(H_k)$. 
Given $k\in[n]$, let 
$\Omega_k = \diag(\omega^\ell)_{\ell\neq k} \in \CC^{(n-1)\times(n-1)}$. Let 
\[
G_k = \begin{bmatrix} \Omega_k^{-L} & Y_k^* \\ Y_k & z \end{bmatrix} 
 \in \CC^{(N+n-1)\times (N+n-1)}, 
\]
and partition $G_k^{-1}$ as 
\[
G_k^{-1} = \begin{bmatrix} E_k & F_k \\ P_k & D_k \end{bmatrix}, 
 \quad E_k \in \CC^{(n-1)\times (n-1)}, \ D_k \in \CC^{N\times N} .
\]
Then, after some algebra, we get the following equation (similar to what was obtained in~\cite{ver-14}): 
\[
\dist(h_k, H_{k}) = 
\frac{\left| \omega^{-kL} - y_k^* D_k y_k \right|}{
 \sqrt{1 + \| y_k^* P_k \|^2 + \| y_k^* D_k \|^2}}. 
\]
Let us assume temporarily that the $\check y_k$ are equal to zero.  Restricting
ourselves to the indices $k$ for which $s_{N-1}(S_k) \sim 1$ (which is allowed),
we can show that for these indices 
\begin{equation}
\label{fq} 
\sqrt{1 + \| y_k^* P_k \|^2 + \| y_k^* D_k \|^2} \sim 
{\| S_k^{1/2} D_k S_k^{1/2}\|_\HS}. 
\end{equation} 
Consequently, the control of $\dist(h_k, H_{k})$ can be reduced to the control
of the probability (with respect to the law of $y_k$) that,   $y_k^* D_k y_k / \|
S_k^{1/2} D_k S_k^{1/2}\|_\HS$ lies in a small ball. Due to the Gaussian
nature of $y_k$,  this task is easy and leads to a rate of 
$N^{-1} t$ in the statement of Theorem~\ref{snY}. 

Now consider the term that involves $\check y_k$. Even though $\EE\| \check y_k \|^2$ is small, its interdependence with $(P_k, D_k)$ prevents us from obtaining an approximation similar to~\eqref{fq}.  Its presence in fact is responsible for the $N^{-3/2}t$ term (instead
of $N^{-1} t$) in the statement of Theorem~\ref{snY}. \\

The proof of Theorem~\ref{intrm} is laid out on a similar canvas. Let
$k\in[\lfloor N^\beta\rfloor, \lfloor N /2\rfloor]$.  In Lemma~\ref{Htronc}
below, we show that the smallest singular value of 
$H_{\cdot, [N+n-k]}$, is a lower bound for
$s_{N-k-1}(Y\Omega^LY^* - z)$. As in the proof of Theorem~\ref{snY}, this
value can be characterized through the action of $H_{\cdot, [N+n-k]}$ on the 
compressible and incompressible vectors. The former can be
handled exactly as for $H$. The latter can also be reduced to a distance problem, and indeed this term is easier to 
tackle than earlier, thanks to the
rectangular nature of $H_{\cdot, [N+n-k]}$.

\subsection{Statistical analysis of the process $(y_k)$} 
\label{orth-pol} 
Recall the decomposition $y_k = \tilde y_k + \check y_k$, with $\check
y_k = \EE[ y_k \, | \, Y_{k} ]$. We now derive approximations for $n \EE\left[ y_k y_k^* \right]$ and 
 $n \EE\left[ \tilde y_k \tilde y_k^* \right]$. Let $\cR$ denote the $Nn \times Nn$ block-Toeplitz matrix 
\begin{equation}
\label{btoep} 
{\cR} = \EE \vect X (\vect X)^* = \frac 1n 
\begin{bmatrix} R_0     & R_{-1} &        & R_{-n+1} \\
                R_1     & \ddots & \ddots &  \\
                        & \ddots & \ddots & R_{-1} \\
                R_{n-1} &        &  R_1   & R_0 
\end{bmatrix}
= \frac 1n \begin{bmatrix} R_{k-\ell} \end{bmatrix}_{k,\ell=0}^{n-1} . 
\end{equation} 
We also write the Fourier matrix $\sF$ as 
\[
\sF = \begin{bmatrix} & \ssf_{0} & \\ &\vdots& \\ 
  & \ssf_{n-1} & \end{bmatrix},  
\]
with $\ssf_k \in \CC^{1\times n}$ being the $k^{\text{th}}$ row of ${\sF}$, and
we define the function 
\[
\ssa(e^{\imath\theta}) = 
 \begin{bmatrix} 1 & e^{-\imath\theta} & \cdots &
e^{-\imath(n-1)\theta} \end{bmatrix}^\T \in \CC^{n}. 
\]
From Equation~\eqref{R-S}, we have 
\[
\cR = \frac{1}{2\pi n} \int_0^{2\pi}    
(\ssa(e^{\imath\theta}) \otimes I_N) S(e^{\imath\theta}) 
   (\ssa(e^{\imath\theta})^* \otimes I_N) \ d\theta . 
\]
Observe that $\cR$ is invertible, since the spectral density
$S(e^{\imath\theta})$ is non-trivial, as is well-known.  Indeed, the equation
$\cR u = 0$ with 
$u = \begin{bmatrix} u_0^\T, \ldots, u_{n-1}^\T \end{bmatrix}^\T$
being a non zero vector with $u_\ell\in \CC^N$ would lead to the identity
$\int_0^{2\pi} p(e^{\imath\theta})^* S(e^{\imath\theta}) p(e^{\imath\theta}) \,
d\theta = 0$ where $p(z) = \sum_{\ell=0}^{n-1} z^\ell u_\ell$ is a non-zero
polynomial. We also have that $\| n \cR \| \leq \bM$ from
Assumption~\ref{prop-S}--\ref{bnd-S} by a similar argument. 
\begin{lemma}
\label{cov-y} 
Let Assumption~\ref{prop-S} hold true. Then, 
\[
 \max_{k\in[n]} 
 \| n \EE\left[ y_k y_k^* \right] - S^{(N)}(e^{2\imath\pi k / n}) \| 
 \xrightarrow[N\to\infty]{} 0. 
\]
\end{lemma}
\begin{proof}
Let $k\in[n]$. Since $\sF = \sF^\T$, we have 
$y_k = X \sF e_k = X \ssf_k^\T = (\ssf_k \otimes I_N) \vect X$. Hence, 
\begin{align} 
n \EE y_k y_k^* &= n (\ssf_k \otimes I_N) \cR (\ssf_k^* \otimes I_N) 
   \nonumber\\
&= \frac{1}{2\pi} \int_0^{2\pi} 
((\ssf_k \ssa(e^{\imath\theta})) \otimes I_N ) S(e^{\imath\theta}) 
 ( (\ssa(e^{\imath\theta})^* \ssf_k^*) \otimes I_N) \ d\theta \nonumber \\
&= \frac{1}{2\pi} \int_0^{2\pi} 
 \left| \ssf_k \ssa(e^{\imath\theta}) \right|^2 S(e^{\imath\theta}) \ d\theta 
= \frac{1}{2\pi} \int_0^{2\pi} 
 F_{n-1}(\theta- 2\pi k / n)  S(e^{\imath\theta}) \ d\theta , 
\label{fejer} 
\end{align} 
where $F_{n-1}$ is the Fej\'er kernel, defined as 
\begin{equation}
\label{def-fej} 
F_{n-1}(\theta) = 
 \frac 1n \Bigl| \sum_{\ell=0}^{n-1} e^{\imath\ell\theta} \Bigr|^2 = 
\frac 1n \frac{\sin (n \theta / 2)^2}{\sin (\theta/2)^2} .  
\end{equation} 
This kernel satisfies 
\[
 \frac{1}{2\pi} \int_0^{2\pi} 
 F_{n-1}(\theta) \ d\theta = 1 , \quad \text{and} \quad 
 F_{n-1}(\theta) \leq n \left(\frac{\pi^2}{n^2\theta^2} 
 \wedge \frac{\pi^2}{4} \right) \ \text{for} \ \theta\in[-\pi, \pi], 
\]
see, \emph{e.g.}, \cite[Page 136]{sim-livre-11}. For an arbitrarily small
number $\eta > 0$, we know by Assumption~\ref{prop-S}--\ref{S-equi} that there
exists $\delta > 0$ independent of $k$ and $N$ such that $\|
S(e^{\imath\theta}) - S(e^{2\imath\pi k / n}) \| \leq \eta$ if $| \theta- 2\pi
k / n | \leq \delta$.  Splitting the integral that apppears on the right side
of~\eqref{fejer} as $\int_{|\theta-2\pi k / n | \leq\delta} +
\int_{|\theta-2\pi k / n | > \delta}$ and using the properties of the Fejér
kernel provided above, along with Assumption~\ref{prop-S}--\ref{bnd-S}, we
obtain the result of the lemma after some routine derivations. 
\end{proof}


The handling of $n \EE\left[ \tilde y_k \tilde y_k^* \right]$ is more involved. 

\begin{proposition}
\label{CD->S} 
Let Assumptions~\ref{prop-S} and~\ref{log-S} hold true. Then, 
\[
\forall\delta > 0, 
 \max_{k\in[n]} 
 \| n \EE\left[ \tilde y_k \tilde y_k^* \right] 
    - S^{(N)}(e^{2\imath\pi k / n}) \| 
  \1_{S^{(N)}(e^{2\imath\pi k / n}) \geq \delta I_N} 
 \xrightarrow[N\to\infty]{} 0 .
\]
\end{proposition}

\paragraph{Proof of Proposition~\ref{CD->S}.} 
The first step is to provide a single letter expression for  
$\EE\left[ \tilde y_k \tilde y_k^* \right]$. This expression is reminiscent of the formula to calculate the partial covariance of real valued random variables.  
\begin{lemma}
\label{covcond}  For any given $k\in[n]$, 
$\displaystyle{\EE\left[ \tilde y_k \tilde y_k^* \right] 
= \left( (\ssf_k \otimes I) \cR^{-1} (\ssf_k^* \otimes I)\right)^{-1}}$. 
\end{lemma} 
\begin{proof} 
The covariance matrix $\EE\left[ \tilde y_k \tilde y_k^* \right]$ coincides
with the conditional covariance matrix $\cov( y_k \, | \, Y_k )$ of $y_k$ with
respect to $\sigma(Y_k)$. Let $Z_0$ and $Z_1$ be two random square integrable
vectors with arbitrary dimensions. Writing   
\[
\EE\begin{bmatrix} Z_0 \\ Z_1 \end{bmatrix} 
 \begin{bmatrix} Z_0^* & Z_1^* \end{bmatrix} = 
 \begin{bmatrix} \Sigma_{00} & \Sigma_{01} \\ \Sigma_{10} & \Sigma_{11}
 \end{bmatrix}
\]
with the block dimensions at the right hand side being compatible with the
dimensions of $Z_0$ and $Z_1$, it is well-known that $\cov( Z_0 \ | \ Z_1 ) =
\Sigma_{00} - \Sigma_{01} \Sigma_{11}^{-1} \Sigma_{10}$ when $\Sigma_{11}$ is
invertible. Observe that 
\[
 \EE \vect Y (\vect Y)^* = \EE \vect(X\sF) (\vect(X\sF))^* = 
 (\sF \otimes I_N) {\cR} (\sF^* \otimes I_N). 
\]
Let ${\sF}_{k} \in \CC^{(n-1) \times n}$ be the matrix that remains 
after taking the row $\ssf_k$ out of $\sF$, and let ${\cR}^{1/2}$ be the Hermitian square root of $\cR$. Then 
\begin{align*} 
& \EE\left[ \tilde y_k \tilde y_k^* \right] \\
&= 
(\ssf_k \otimes I_N) {\cR}^{1/2} \left( 
I_{(n-1)N} - 
 \cR^{1/2} ({\sF}_k^* \otimes I_N) 
\left( (\sF_k \otimes I_N) \cR 
   (\sF_k^* \otimes I_N) \right)^{-1} 
(\sF_k \otimes I_N)\cR^{1/2} \right) \\
&\phantom{=} \ \ \ \ \ \ \ \ \cR^{1/2} (\ssf_k^* \otimes I_N) \\ 
&= 
(\ssf_k \otimes I_N) \cR^{1/2} 
 \Pi_{\cR^{1/2} ( \sF_k^* \otimes I_N)}^\perp 
   \cR^{1/2} (\ssf_k^* \otimes I_N) . 
\end{align*} 
Let $\cC$ be a positively oriented circle with center zero and a radius small
enough so that $\cR$ has no eigenvalue in the closed disk delineated by $\cC$. 
This implies that the positive definite matrix 
$( \sF_k \otimes I) {\cR}( \sF_k^* \otimes I)$
does not have eigenvalues in this closed disk either. With this choice, 
the projection $\Pi_{\cR^{1/2} ( \sF_k^* \otimes I_N)}^\perp$ can be expressed
as a contour integral, which leads us to write 
\begin{align*} 
& \EE\left[ \tilde y_k \tilde y_k^* \right] \\ 
&= \frac{-1}{2\imath\pi} \oint_{\cC} 
(\ssf_k \otimes I_N) {\cR}^{1/2} 
 \left( 
  {\cR}^{1/2} ( \sF_{k}^* \otimes I_N) ( \sF_{k} \otimes I_N) {\cR}^{1/2}  
  - z \right)^{-1} {\cR}^{1/2} (\ssf_k^* \otimes I_N) \ dz  \\ 
 &= \frac{-1}{2\imath\pi} \oint_{\cC} 
 (\ssf_k \otimes I_N) {\cR}^{1/2} 
  \left( \cR - {\cR}^{1/2}  ( \ssf_k^* \otimes I_N) ( \ssf_k \otimes I_N) 
   {\cR}^{1/2} - z \right)^{-1} 
 {\cR}^{1/2} (\ssf_k^* \otimes I_N) \ dz   . 
\end{align*} 
Let $\cQ(z) = (\cR - z)^{-1}$ when $z \in\CC$ is not an eigenvalue of 
$\cR$, and let 
\[
\Sigma(z) = (\ssf_k \otimes I_N) \cR^{1/2} \cQ(z) 
\cR^{1/2} (\ssf_k^* \otimes I_N) = 
 I_N + z (\ssf_k \otimes I_N) \cQ(z) (\ssf_k^* \otimes I_N) . 
\]
Now using the Sherman-Morrison-Woodbury formula, 
we get  
\begin{align*} 
\EE\left[ \tilde y_k \tilde y_k^* \right] 
&= \frac{-1}{2\imath\pi} \oint_{\cC} 
\left( \Sigma(z) + \Sigma(z) ( I_N - \Sigma(z) )^{-1} \Sigma(z) \right) dz \\ 
&= \frac{1}{2\imath\pi} \oint_{\cC} z^{-1} 
\left( (\ssf_k \otimes I_N) \cQ(z)(\ssf_k^* \otimes I_N)\right)^{-1} dz \\
&= \left( (\ssf_k \otimes I_N) \cR^{-1} (\ssf_k^* \otimes I_N)\right)^{-1}, 
\end{align*} 
where the last equality is obtained by expressing $\cQ(z)$ in terms of the 
spectral decomposition of $\cR$ and by using that 
$(2\imath\pi)^{-1} \oint_{\cC} z^{-1} (\lambda - z)^{-1} dz = \lambda^{-1}$
when $\lambda$ is outside the closed disk enclosed by $\cC$. 
\end{proof} 

We now reinterpret the expression provided by this lemma in the light of the
matrix orthogonal polynomial theory. Let us quickly review some of the basic
results of this theory. For the proofs of these results,  the reader may
consult \cite{dam-pus-sim-08}.  

For each $N$, consider the $N\times N$ matrix valued measure $\varrho_N$ defined
on $\TT$ as 
\[
d\varrho_N(\theta) = \frac{1}{2\pi} S^{(N)}(e^{\imath \theta}) d\theta. 
\]
Given two $\CC^{N\times N}$ matrix--valued polynomials $F, G$, we define the 
$N\times N$ matrix sesquilinear function $\pss{F,G}_{\varrho_N}$ with respect to 
$\varrho_N$ as 
\[
\pss{F,G}_{\varrho_N} = 
  \int_0^{2\pi} F(e^{\imath\theta})^* d\varrho_N(\theta) G(e^{\imath\theta}) . 
\]
We now define the sequence $(\Phi^{\varrho_N}_\ell)_{\ell=0,1,2,\ldots}$ of matrix
orthogonal polynomials with respect to $\pss{\cdot,\cdot}_{\varrho_N}$.  The
following conditions, which are analogous to a Gram-Schmidt orthogonalization,
are enough to define this sequence: \begin{itemize} \item
$\Phi^{\varrho_N}_\ell(z)$ is a $\CC^{N\times N}$--valued monic matrix polynomial
of degree $\ell$, and is thus written as $\Phi^{\varrho_N}_\ell(z) = z^\ell I_N +
\text{lower order coefficients}$. 

\item For each $\ell\geq 1$, the relation 
$\pss{\Phi^{\varrho_N}_\ell, z^k}_{\varrho_N} = 0$ holds for $k = 0,1,\ldots, \ell-1$.
\end{itemize}
Since $S^{(N)}$ is non-trivial, the matrix 
$\pss{\Phi^{\varrho_N}_\ell,\Phi^{\varrho_N}_\ell}_{\varrho_N}$ is positive definite 
for each $\ell\in\NN$. Thus, one can define the sequence 
$(\varphi^{\varrho_N}_\ell)_{\ell\in\NN}$ of the normalized versions of the 
polynomials $\Phi^{\varrho_N}_\ell$ as 
\[
 \varphi^{\varrho_N}_\ell(z) = \Phi^{\varrho_N}_\ell(z) \kappa^{\varrho_N}_\ell,  
\]
where the matrices $(\kappa^{\varrho_N}_\ell)_{\ell\in\NN}$ are chosen in such a 
way that $\pss{\varphi^{\varrho_N}_\ell, \varphi^{\varrho_N}_k}_{\varrho_N} 
= \1_{k=\ell} I_N$.
This identity determines the matrices $\kappa^{\varrho_N}_\ell$ up to a right
multiplication by a unitary matrix, the convenient choice of which is 
specified in \cite[\S~3.2]{dam-pus-sim-08} and is not relevant here. 

The \textit{Christoffel-Darboux (CD)} kernel of order $\ell$ for the measure $\varrho_N$ 
is defined for $z,u\in\CC$ as
\begin{equation}
\label{CD} 
K^{\varrho_N}_\ell(z,u) 
  = \sum_{k=0}^\ell \varphi^{\varrho_N}_k(z) \varphi^{\varrho_N}_k(u)^* . 
\end{equation} 
This function is a reproducing kernel, in the sense that for each matrix 
polynomial $P(z)$ with degree less or equal to $\ell$, the following equation holds
and defines the kernel $K^{\varrho_N}_\ell$.    
\[
\int_0^{2\pi} K^{\varrho_N}_\ell(z,e^{\imath\theta}) d\varrho_N(\theta)  
 P(e^{\imath\theta}) = P(z). 
\] 
The CD kernel 
$K^{\varrho_N}_\ell(e^{\imath\theta},e^{\imath\theta})$ is invertible, and it 
satisfies the following variational formula: for each $\theta \in [0, 2\pi)$ 
and for each $N\times N$ matrix polynomial $P(z)$ of degree  at most $\ell$, 
with 
\[
P(e^{\imath\theta}) = I_N, 
\]
it holds that 
\begin{equation}
\label{CD-var} 
\pss{P,P}_{\varrho_N} \geq 
  K^{\varrho_N}_\ell(e^{\imath\theta}, e^{\imath\theta})^{-1}, 
\end{equation}
with equality if and only if $P(z) = P^{\varrho_N,\theta}_\ell(z)$, with 
\begin{equation}
\label{P=KK} 
P^{\varrho_N,\theta}_\ell(z) = K^{\varrho_N}_\ell(z, e^{\imath\theta})
  K^{\varrho_N}_\ell(e^{\imath\theta}, e^{\imath\theta})^{-1} . 
\end{equation} 
Getting back to our problem, it turns out that our covariance matrix 
$\EE[\tilde y_k \tilde y_k^*]$ can be very simply expressed in terms of a CD 
kernel: 
\begin{lemma}
\label{cov-CD}
$\displaystyle{
\EE[\tilde y_k \tilde y_k^*] = 
K^{\varrho_N}_{n-1}(e^{2\imath\pi k / n}, e^{2\imath\pi k / n})^{-1} 
}$. 
\end{lemma}
\begin{proof}
With Lemma~\ref{covcond} at hand, this result is an instance of a well-known
result, see, \emph{e.g.}, \cite{kai-vie-mor-78} for the scalar measure case.
We reproduce its proof for completeness.  Let $P(z) = \sum_{k=0}^{n-1} z^k P_k$
be an arbitrary $N\times N$ matrix polynomial of degree at most ${n-1}$.
Stack the coefficients of $P$ in the matrix ${\sf P} = \begin{bmatrix} P_0^\T &
\cdots & P_{n-1}^T\end{bmatrix}^\T$, so that $P(e^{\imath\theta}) =
(\ssa(e^{\imath\theta})^* \otimes I_N) \sf P$. We have 
\begin{align*} 
& \frac{1}{2\pi n} \int_0^{2\pi} 
 (\ssa(e^{\imath\theta})^* \otimes I_N) \cR^{-1} 
 (\ssa(e^{\imath\psi}) \otimes I_N) S(e^{\imath\psi}) 
  P(e^{\imath\psi}) \ d\psi \\ 
&= (\ssa(e^{\imath\theta})^* \otimes I_N) \cR^{-1} 
 \left( \frac{1}{2\pi n} \int_0^{2\pi}    
(\ssa(e^{\imath\psi}) \otimes I_N) S(e^{\imath\psi}) 
   (\ssa(e^{\imath\psi})^* \otimes I_N) \ d\psi \right) \sf P \\ 
&= (\ssa(e^{\imath\theta})^* \otimes I_N) \cR^{-1} \cR \sf P \\ 
&= P(e^{\imath\theta}). 
\end{align*} 
From the uniqueness of the CD kernel as a reproducing kernel, we thus obtain
that 
$$n^{-1} (\ssa(e^{\imath\theta})^* \otimes I_N) \cR^{-1}
(\ssa(e^{\imath\psi}) \otimes I_N) = K^{\varrho_N}_{n-1}(e^{\imath\theta},
e^{\imath\psi}).$$ The result follows upon observing that $\ssf_k = n^{-1/2}
\ssa(e^{2\pi\imath k/n})^*$.  
\end{proof}

Now Proposition \ref{CD->S} will be obtained by studying the
asymptotics of the CD kernel, a subject with a rich history in the literature of orthogonal
polynomials. These asymptotics are well understood in the case where
the underlying measure has a so-called regularity property \`a la Stahl and
Totik \cite{sta-tot-livre92} (see~\cite{sim-livre-11} for an extensive 
treatment of these ideas). 
We adapt the approach detailed in 
\cite[\S 2.15 and \S 2.16]{sim-livre-11}, to the matrix measure
case. This will consist of two main steps: 
\begin{itemize}
\item 
We show that the polynomial $P^{\varrho_N,\theta}_{n-1}$, for which the
variational inequality~\eqref{CD-var} for $\ell=n-1$ is an equality, 
satisfies
\begin{equation}
\label{bnd-P} 
 \max_{\theta, \psi \in [0, 2\pi)} 
   \| P^{\varrho_N,\theta}_{n-1}(e^{\imath\psi}) \|  \leq e^{n\varepsilon_N} 
 \quad \text{with} \quad  0 \leq\varepsilon_N \to 0 . 
\end{equation} 
This is where Assumption~\ref{log-S} comes into play. 

\item With the help of the variational characterization of the CD kernel, 
and by making use of the previous result coupled with a matrix version of the
so-called Nevai's trial polynomial technique \cite{nev-livre79, sim-livre-11}, we
show that 
\begin{multline} 
\forall \varepsilon, \delta > 0, \exists N_0\in\NN, \forall N > N_0, 
\forall \theta_0 \in [0, 2\pi), 
S^{(N)}(e^{\imath\theta_0}) \geq \delta I_N  \\
  \Rightarrow 
 n K^{\varrho_N}_{n-1}(e^{\imath\theta_0}, e^{\imath\theta_0})^{-1} \geq 
    (1 - \varepsilon) S^{(N)}(e^{\imath\theta_0}).
\label{K>S}
\end{multline} 
\end{itemize} 
In the following derivations, the sequence $\varepsilon_N \to 0$ can change
from one display to another. 

For a given $N \in  \NN\setminus\{ 0 \}$, consider the scalar measure $\zeta_N$ on $\TT$ 
defined as 
\[
d\zeta_N(\theta) = \frac{1}{2\pi} s_{N-1}(S^{(N)}(e^{\imath\theta})) \ d\theta . 
\]
Consider the sesquilinear form on the scalar polynomials 
\[
\ps{f,g}_{\zeta_N} = \int_0^{2\pi} f(e^{\imath\theta})^* g(e^{\imath\theta}) 
 \ d\zeta_N(e^{\imath\theta}) , 
\]
and denote as $(b^{\zeta_N}_\ell)_{\ell\in\NN}$ the sequence of monic orthogonal
polynomials with respect to $\ps{\cdot,\cdot}_{\zeta_N}$, which are the analogues
of the $\Phi_\ell^{\varrho_N}$ above.  Let $\tau^{\zeta_N}_\ell =
\ps{b^{\zeta_N}_\ell, b^{\zeta_N}_\ell}_{\zeta_N}$, and consider the orthonormal
polynomials $(\beta^{\zeta_N}_\ell)_{\ell\in\NN}$ defined as
$\beta^{\zeta_N}_\ell(z) = b^{\zeta_N}_\ell(z) / \sqrt{\tau^{\zeta_N}_\ell}$ (these 
are the analogues of the $\varphi_\ell^{\varrho_N}$). To establish~\eqref{bnd-P}, 
we first show that 
\begin{equation} 
\label{bnd-beta} 
 \max_{\ell\in[n]} \max_{\theta \in [0, 2\pi)} 
  | \beta^{\zeta_N}_{\ell}(e^{\imath\theta}) | \leq e^{n\varepsilon_N} 
 \quad \text{with} \quad \varepsilon_N \geq 0, \ \varepsilon_N \to 0 . 
\end{equation} 
Write 
\[
\sigma^2_N = \tau_0^{\zeta_N} = \int_0^{2\pi} \, d\zeta_N(e^{\imath\theta}) = 
 \frac{1}{2\pi} \int_0^{2\pi} s_{N-1}(S^{(N)}(e^{\imath\theta})) \, d\theta  . 
\]
The sequence $(\tau_\ell^{\zeta_N})_\ell$ satisfies the recursion 
$\tau_{\ell+1}^{\zeta_N} = ( 1 - |\alpha_\ell|^2) \tau_\ell^{\zeta_N}$, 
thus, 
$\tau_{\ell+1}^{\zeta_N} = \sigma^2_N \prod_{k=0}^{\ell} ( 1 - |\alpha_k|^2)$,
where $(\alpha_\ell)_{\ell=0}^\infty$ is the sequence of the so-called
Verblunsky's coefficients associated to the measure $\zeta_N$
\cite[\S 1.7 and \S 1.8]{sim-livre-11}. Moreover, by the 
celebrated Szeg\"o's theorem \cite[Th.~1.8.6]{sim-livre-11}, the non-increasing
sequence $(\tau^{\zeta_N}_\ell)_\ell$ converges as
$\ell\to\infty$ towards $\exp((2\pi)^{-1} \int_0^{2\pi} \log
s_{N-1}(S^{(N)}(e^{\imath\theta})) \, d\theta)$.  

From this, we first deduce with the help of 
Assumption~\ref{prop-S}--\ref{bnd-S} that there is a constant $C > 0$
such that 
\[
C \geq \sigma^2_N \geq \tau^{\zeta_N}_{n-1} \geq 
\exp\left(\frac{1}{2\pi} 
 \int_0^{2\pi} \log s_{N-1}(S^{(N)}(e^{\imath\theta})) \, d\theta \right).
\] 
Furthermore, by inspecting the proof of \cite[Th.~2.15.3]{sim-livre-11}, in
particular, Inequality~(2.15.21) of that proof, and by using in addition 
Inequalities~(2.15.13) then~(2.15.15) of \cite{sim-livre-11}, we get that for 
$|z| > 1$, 
\begin{align*} 
\max_{\ell\in[n]} \frac{| b^{\zeta_N}_\ell(z) |}{|z|^\ell} &\leq 
\exp\Bigl( \sum_{k=0}^{n-1} | \alpha_k | \Bigr) \\
 &\leq 
  \exp\Bigl( \frac 1n \sqrt{\frac 1n \sum_{k=0}^{n-1} | \alpha_k |^2} \Bigr) \\
 &\leq 
  \exp\Bigl( \frac 1n 
  \sqrt{-\frac 1n \log \prod_{k=0}^{n-1} ( 1 - |\alpha_k|^2)} \Bigr) \\
&= \exp\left( n \sqrt{\frac 1n \left( - \log \tau^{\zeta_N}_{n}
+ \log \sigma^2_N  \right) } \right) 
\end{align*} 
Thus, since 
\[
\frac 1n \log C \geq \frac 1n \log \sigma^2_N \geq 
 \frac 1n \log \tau^{\zeta_N}_{n} \geq \frac{1}{2\pi n} 
 \int_0^{2\pi} \log s_{N-1}(S^{(N)}(e^{\imath\theta})) \, d\theta, 
\]
we obtain from Assumption~\ref{log-S} that there exists a non-negative sequence 
$\varepsilon_N \to_N 0$ such that 
\[
\max_{\ell\in[n]} \frac{| b^{\zeta_N}_\ell(z) |}{|z|^\ell} \leq 
  \exp\left( n \varepsilon_N \right) . 
\]
By the maximum principle, for any $\theta\in[0,2\pi)$, any $\ell\in[n]$ and 
any $r > 1$, 
\begin{align*} 
| \beta^{\zeta_N}_\ell(e^{\imath\theta}) | &= 
(\tau^{\zeta_N}_\ell)^{-1/2} | b^{\zeta_N}_\ell(e^{\imath\theta}) | \leq 
(\tau^{\zeta_N}_{n})^{-1/2} \sup_{\psi\in[0,2\pi)} 
| b^{\zeta_N}_\ell(r e^{\imath\psi}) | \\
 &\leq 
r^\ell (\tau^{\zeta_N}_{n})^{-1/2} \exp\left( n \varepsilon_N \right) , 
\end{align*} 
and by Szeg\"o's theorem and Assumption~\ref{log-S} again, we 
obtain~\eqref{bnd-beta} by choosing $r$ as close to one as desired. 

Proceeding, we now consider the matrix measure 
$d\zeta_N(e^{\imath\theta}) \otimes I_N = 
(2\pi)^{-1} s_{N-1}(S^{(N)}(e^{\imath\theta})) I_N d\theta$, equipped
with the sesquilinear function 
\[
\pss{F,G}_{\zeta_N\otimes I_N} = 
\frac{1}{2\pi} \int_0^{2\pi} s_{N-1}(S^{(N)}(e^{\imath\theta})) 
F^*(e^{\imath\theta}) G(e^{\imath\theta}) \, d\theta. 
\]
It is clear that the $\ell^{\text{th}}$ normalized orthogonal polynomial for 
$\pss{\cdot,\cdot}_{\zeta_N\otimes I_N}$ is $\beta^{\zeta_N}_\ell(z) I_N$, and the
associated $\ell^{\text{th}}$ CD kernel is 
\[
K^{\zeta_N\otimes I_N}_\ell(z,u) = 
  \sum_{k=0}^\ell \beta^{\zeta_N}_k(z) (\beta^{\zeta_N}_k(u))^* I_N . 
\]
Obviously, for each Borel set $A \subset[0, 2\pi)$, $\zeta_N(A) I_N \leq
\varrho_N(A)$ in the semi-definite ordering. Therefore, by the variational
characterization of the CD kernels, we have  
\[
K^{\zeta_N\otimes I_N}_{n-1}(e^{\imath\theta}, e^{\imath\theta}) \geq 
\pss{P^{\varrho_N,\theta}_{n-1}, P^{\varrho_N,\theta}_{n-1}}_{\zeta_N\otimes I_N}^{-1} 
\geq \pss{P^{\varrho_N,\theta}_{n-1}, P^{\varrho_N,\theta}_{n-1}}_{\varrho_N}^{-1} 
= K^{\varrho_N}_{n-1}(e^{\imath\theta}, e^{\imath\theta}) 
\]
for all $\theta\in[0, 2\pi)$. Recalling the definition~\eqref{CD} of the kernel
$K^{\varrho_N}_\ell(z,u)$, we obtain from Inequality~\eqref{bnd-beta} that  
\[
 \max_{\ell \in [n]} \max_{\theta \in [0, 2\pi)}  
 \| \varphi^{\varrho_N}_\ell(e^{\imath\theta}) \|  \leq e^{n\varepsilon_N} 
 \quad \text{with} \quad   0 \leq  \varepsilon_N \to 0. 
\]
Also notice that 
\[
K^{\varrho_N}_{n-1}(e^{\imath\theta}, e^{\imath\theta}) 
 \geq \varphi^{\varrho_N}_0(e^{\imath\theta}) 
  \varphi^{\varrho_N}_0(e^{\imath\theta})^* = 
 \Bigl( \int_0^{2\pi} d\varrho_N(\theta) \Bigr)^{-1} = (R^{(N)}_0)^{-1}. 
\]
Thus, using Assumption~\ref{prop-S}--\ref{bnd-S} and recalling Definition
\eqref{P=KK} of the polynomials $P^{\varrho_N,\theta}_\ell$, we 
obtain the bound~\eqref{bnd-P}.   

We now prove~\eqref{K>S}. Let $\delta > 0$ and $\theta_0 \in (0, 2\pi)$ be 
such that $S^{(N)}(e^{\imath\theta_0}) \geq \delta I_N$. 
Consider the matrix measure
\[
d\varsigma_N(\theta) = \frac{1}{2\pi} S^{(N)}(e^{\imath\theta_0}) \ d\theta .
\]
The CD kernels for this measure are constant and satisfy the identity $\ell
K^{\varsigma_N}_{\ell-1}(e^{\imath\theta}, e^{\imath\theta})^{-1} =
S^{(N)}(e^{\imath\theta_0})$ for each integer $\ell > 0$, as can be checked
by the direct calculation of the orthogonal polynomials for 
$d\varsigma_N(\theta)$, or by the application of Lemmas~\ref{covcond} 
and~\ref{cov-CD}, with the blocks of the matrix $\cR$ being replaced with
$R_k = \1_{k=0} S^{(N)}(e^{\imath\theta_0})$. 

Let $\eta > 0$ be an arbitrarily small number. By
Assumption~\ref{prop-S}--\ref{S-equi}, we can choose $h > 0$ small enough so
that for each $N > 0$, $\bs w(S^{(N)}, h) < \eta\delta$. Let $\theta \in
[\theta_0 - h, \theta_0 + h]$. Then, for each vector $u\in\CC^N$ with $\| u \|
= 1$, we have 
\begin{align*}
u^* S^{(N)}(e^{\imath\theta}) u &= 
 u^* S^{(N)}(e^{\imath\theta_0}) u \left( 1 + 
  \frac{u^* ( S^{(N)}(e^{\imath\theta}) - S^{(N)}(e^{\imath\theta_0}) ) u}
   {u^* S^{(N)}(e^{\imath\theta_0}) u} \right)  , 
\end{align*} 
and hence, 
\[
\forall \theta \in [\theta_0 - h, \theta_0 + h], \ 
S^{(N)}(e^{\imath\theta_0}) \leq 
  (1 + 2\eta) S^{(N)}(e^{\imath\theta}). 
\]
The degree one polynomial $q(z) = 0.5( z e^{-\imath\theta_0} + 1 )$ can be
easily shown to satisfy $q(e^{\imath\theta_0}) = \sup_{\theta}
|q(e^{\imath\theta})|  = 1$, and $\sup_{\theta-\theta_0 \in [-\pi,\pi), |\theta
- \theta_0|\geq h} |q(e^{\imath\theta})| = \cos(h/2) < 1$. 
Let $\tilde\eta > 0$ be arbitrarily small, and let 
$m = n - 1 + \lfloor n\tilde\eta\rfloor$. Consider the ``trial polynomial''
$Q_m(e^{\imath\theta}) = P^{\varrho_N,\theta_0}_{n-1}(e^{\imath\theta})
q(e^{\imath\theta})^{\lfloor n\tilde\eta\rfloor}$ with degree $m$. This 
polynomial has the following features: 
\begin{itemize}
\item $Q_m(e^{\imath\theta_0}) = I_N$. 

\item By Assertion~\eqref{bnd-P} that we just proved, for 
$\theta-\theta_0 \in [-\pi,\pi), |\theta - \theta_0|\geq h$, for all large $N$,
$$\| Q_n(e^{\imath\theta} \| \leq |\cos(h/2)|^{n\tilde\eta} 
 \| P^{\varrho_N,\theta_0}_{n-1}(e^{\imath\theta}) \| \leq 
 |\cos(h/2)|^{n\tilde\eta/2}.$$  
\end{itemize}
By the variational characterization of the CD kernels, we have 
\begin{align*} 
\frac{1}{m+1} S^{(N)}(e^{\imath\theta_0}) 
&=  (K^{\varsigma_N}_m(e^{\imath\theta_0}, e^{\imath\theta_0}))^{-1} \\
&= \pss{P^{\varsigma_N, \theta_0}_m, P^{\varsigma_N, \theta_0}_m}_{\varsigma_N} \\ 
&\leq \pss{Q_m, Q_m}_{\varsigma_N} \\ 
&\leq (1+2\eta) \int_{\theta\, : \, |\theta-\theta_0| \leq h} 
  Q_m(e^{\imath\theta}) d\varrho_N(\theta) Q_m(e^{\imath\theta})^*  \\
&\phantom{=} 
  + \int_{\theta\, : \, |\theta-\theta_0| > h} 
   Q_m(e^{\imath\theta}) d\varsigma_N(\theta) Q_m(e^{\imath\theta})^* \\
&\leq (1+2\eta) 
  \pss{P^{\varrho_N,\theta_0}_{n-1}, P^{\varrho_N,\theta_0}_{n-1}}_{\varrho_N} 
 + |\cos(h/2)|^{n\tilde\eta} \bM  \\
&= (1+2\eta) 
  K^{\varrho_N}_{n-1}(e^{\imath\theta_0}, e^{\imath\theta_0})^{-1}  
 + |\cos(h/2)|^{n\tilde\eta} \bM . 
\end{align*} 
Since $\eta$ and $\tilde\eta$ are arbitrary, Assertion~\ref{K>S} holds true. 

We now have all the elements to complete the proof of  Proposition~\ref{cov-CD}. Let $k\in [n]$
be such that $S^{(N)}(e^{2\imath\pi k/n}) \geq \delta I_N$.  Thanks to
Lemma~\ref{cov-CD}, we can replace $K^{\varrho_N}_{n-1}(e^{\imath\theta_0},
e^{\imath\theta_0})^{-1}$ for $\theta_0 = 2\pi k/n$ with $\EE\tilde y_k \tilde
y_k^*$ in Assertion~\ref{K>S}. This provides a lower bound on 
$\EE\tilde y_k \tilde y_k^*$. To obtain an upper bound on this matrix,  
we recall that $\tilde y_k$ and $\check y_k$ are independent, resulting in 
$\EE \tilde y_k \tilde y_k^* \leq \EE y_k y_k^*$, and use Lemma~\ref{cov-y}. 

This completes the proof of Proposition~\ref{CD->S}. 
We note for completeness that we could have established the upper bound 
by using a Nevai's polynomial trial technique as well.

\subsection{Technical results needed in the proofs of Theorems~\ref{snY} 
and~\ref{intrm}}

The following lemma can be proved similarly to, \emph{e.g.},  
\cite[Corollary 4.2.13]{ver-livre18}. 
\begin{lemma}[covering number] 
\label{covnb}
Let $\cU \subset \CC^m$ be a $k$--dimensional subspace. Given $\rho > 0$, 
there exists a $\rho$--net of the Euclidean unit-ball of $\cU$ with a 
cardinality bounded above by $(3/\rho)^{2k}$. 
\end{lemma} 

\begin{lemma}
\label{agood}
Let $a_0, \ldots, a_{m-1}$ be non-negative real numbers such that
there exist $0 < c \leq C$ for which $c \leq m^{-1} \sum a_k$, and
$\max a_k \leq C$. Given $x \in (0, c]$, the set
$\cI(x) \subset [m]$ defined as
\[
\cI(x) = \{ k \in [m] \, : \, a_k > x \} 
\]
satisfies
\[
 | \cI | \geq \frac{c - x}{C - x} m. 
\] 
\end{lemma}
\begin{proof} We have
$\displaystyle{C |\cI| \geq \sum_{k\in\cI} a_k 
\geq m c - \sum_{k \in\cIc} a_k \geq m c - ( m - |\cI| ) x}$,
hence the result.
\end{proof}

For any random vector $\xi  = [ \xi_0, \ldots, \xi_{m-1} ]^\T\in \CC^m$, 
\textit{L\'evy's anti-concentration} function is defined for $t > 0$ as 
\[
\cL(\xi, t) = \sup_{d\in\CC^m} \PP\left[ \| \xi - d \| \leq t\right] 
\]
(when the probability $\PP$ above is taken with respect to some random vector 
$x$, we denote the associated L\'evy's anti-concentration function as 
$\cL_x(\xi, t)$). Assuming that the elements $\xi_\ell$ of $\xi$ are 
independent, letting $k \in [m]$, and defining $\xi^{(k)} = [\xi_0, \ldots, 
 \xi_k ]^\T$, the following restriction result is well-known \cite[Lemma 2.1]{rud-ver-advmath08}. 
$$\cL(\xi, t) \leq \cL(\xi^{(k)}, t).$$ 

We shall need the following rather standard results on Gaussian vectors. 

\begin{lemma}[norm and anti-concentration results for Gaussian vectors] 
\label{gauss}
Let $\xi  = [\xi_0,\ldots, \xi_{N-1} ]^\T\sim \cN_\CC(0,I_N)$, and let $\Sigma$
be an $N\times N$ covariance matrix. Then: 
\begin{enumerate}
\item\label{small-var} For $t > 0$, it holds that 
$\displaystyle{\PP\left[ \| \Sigma^{1/2} \xi \| \geq \sqrt{2 N t} \right] 
 \leq \exp(-(t / \|\Sigma\| - 1)N)}$. 

\hskip-25pt Assume that $s_{m-1}(\Sigma) \geq\alpha$ for some $m\in[N]$ and some
$\alpha > 0$. Then: 

\item\label{normgauss} There exists a constant 
$c_{\ref{gauss},\ref{normgauss}} > 0$ such that 
$\PP\left[ \| \Sigma^{1/2} \xi \| \leq \sqrt{\alpha m / 2} \right] 
  \leq \exp(- c_{\ref{gauss},\ref{normgauss}} m )$. 

\item\label{antifq}
There exists a constant $C_{\ref{gauss}, \ref{antifq}} > 0$ such 
that for each deterministic non-zero matrix $M \in \CC^{N\times N}$ and
each deterministic vector $a\in\CC^{N}$, 
\[
\cL({(\xi+a)^* M (\xi+a) / \| M \|_\HS}, t) \leq 
  C_{\ref{gauss}, \ref{antifq}} t 
\]

\item\label{antivec} There exists a constant 
$C_{\ref{gauss}, \ref{antivec}}  = C_{\ref{gauss}, \ref{antivec}}(\alpha)> 0$ 
such that 
$\cL(\Sigma^{1/2} \xi, \sqrt{m} t) \leq 
 (C_{\ref{gauss}, \ref{antivec}} t)^m$.

\item\label{larget} 
For each non-zero deterministic matrix $M \in \CC^{N\times N}$, 
\[
 \PP\left[ \| M\xi \|^2 \geq t \| M \|_\HS^2 \right] \leq \exp(1 - t / 2) .
\]
\end{enumerate} 
\end{lemma} 
This lemma is proved in the appendix.

The following results specifically concern our matrix model. The first one will
be needed to bound the spectral norm of our matrix $Y$. 

\begin{lemma}[spectral norm of $Y$] 
\label{spec_norm} 
Let Assumption~\ref{prop-S}--\ref{bnd-S} hold true. Then, there are two  
constants $c_{\ref{spec_norm}}, c'_{\ref{spec_norm}} > 0$ such that for each $t\geq c'_{\ref{spec_norm}}$,
\[
\PP\left[ \| Y \| \geq t \right] \leq \exp(-c_{\ref{spec_norm}} N t^2 ). 
\]
  \end{lemma}
\begin{proof}
Since $Y = X\sF$, it is enough to prove the lemma for $X$. By a standard $\varepsilon$-net 
argument (see \cite[Lm.~2.3.2]{tao-topics}), we know that 
\[
\PP\left[ \| X \| \geq t \right] \leq 
\sum_{u\in\bs N} \PP\left[ \| X u \| \geq t / 2 \right], 
\]
where $\bs N$ is a $1/2$--net of $\SSS^{n-1}$. Let $u \in \bs N$.  Observing
that $Xu \sim \cN_\CC(0, (u^\T\otimes I_N) \cR (\bar u \otimes I_N))$, and
recalling that $n \| \cR \| \leq \bM$, Lemma~\ref{gauss}--\ref{small-var} shows
that 
$$\PP\left[ \| X u \| \geq t / 2 \right] \leq \exp(-(\frac{t^2 n}{8 \bM N})
-1) N).$$  By Lemma~\ref{covnb} and the union bound, we thus have that
$$
\PP\left[ \| X \| \geq t \right] \leq \exp(-(\frac{t^2 n}{8 \bM N} -1) N +
(2\log 6) n).  
$$  
Choosing $t$ large enough and using~\eqref{n/N}, we get the result.  
\end{proof} 

We shall need to use a discrete frequency counterpart of
Assumption~\ref{S-reg}: 
\begin{lemma}
\label{bad-k} 
Let Assumptions~\ref{prop-S}--\ref{S-equi} and \ref{S-reg} hold true. 
Given $N \in  \NN\setminus\{ 0 \}$ and $\alpha > 0$, let $\Cgood(\alpha)$ be the subset of 
$[n]$ defined as 
\[
\Cgood(\alpha) = \{ k \in [n] \, : \, S(e^{2\imath\pi k / n}) \geq 
   \alpha I_N  \} . 
\]
Then, 
\[
\forall\kappa \in (0,1), \exists \alpha > 0, 
 | \Cgood(\alpha) | \geq (1 - \kappa) n  \quad \text{for all large } N. 
\]
\end{lemma}
\begin{proof} 
Let us identify a set $\cI \subset [n]$ with the discretization of the
unit-circle $\TT$ given as $\{ \exp(2\imath\pi k / n) \, : \, k \in \cI \}$,
and let us also denote any of these two sets as $\cI$.  Given a small real
number $h > 0$, let $\cV_h^\TT(\cI)$ be the closed $h$-neighborhood of $\cI$
on $\TT$ equipped with the curvilinear distance. We shall show that 
\begin{equation}
\label{thick} 
| \cI | \leq \frac{n}{2\pi} \leb(\cV_h^\TT(\cI)) + \frac{\pi}{h}. 
\end{equation} 
We first observe that $\cV_h^\TT(\cI)$ consists of a finite number of disjoint
closed arcs, the length of each arc being greater or equal to $2h$. Given
$\varphi_0 \in [0,2\pi)$ and $\varphi_1 \in [\varphi_0 + 2h, \varphi_0 + 2h +
2\pi)$, let $\{ e^{\imath\varphi} \, : \, \varphi_0\leq\varphi\leq\varphi_1 \}$
be one such arc. Then, there are two integers $k_{\min} \leq k_{\max}$ in $\cI$
such that $2k_{\min}\pi / n = \varphi_0 + h$ and $2k_{\max}\pi / n = \varphi_1
- h$, and all the elements of $\cI$ in this arc belong to the set $\{ k_{\min},
  \ldots, k_{\max} \}$. Since there are at most $k_{\max} - k_{\min} + 1 =
n(\varphi_1-\varphi_0)/(2\pi) - n h / \pi + 1$ of these elements, and
furthermore, since $\cV_h^\TT(\cI)$ consists of 
at most $\lfloor \pi / h \rfloor$ arcs, we obtain the inequality~\eqref{thick}. 

Fix an arbitrary $\kappa \in (0,1)$. By Assumption~\ref{S-reg}, there exists
$\delta > 0$ such that the set $B = \{ z \in \TT \, : \, s_{N-1}(S(z)) \leq
\delta \}$ satisfies $\leb(B) \leq \kappa \pi$.  Let $\alpha = \delta / 2$, and
assume that the set $\Cbad(\alpha) = [n] \setminus \Cgood(\alpha)$ is non-empty
(otherwise the result of the lemma is true).  Relying on
Assumption~\ref{prop-S}--\ref{S-equi}, let $h > 0$ be such that $\bs w(S,h)
\leq \alpha$ for all $N$. Let $k \in \Cbad(\alpha)$. By the Wielandt-Hoffmann
theorem and the triangle inequality, 
$s_{N-1}(S(e^{\imath\theta})) \leq \delta$ for each $\theta$ such that $|\theta
- 2\pi k / n | \leq h$. In other words, $\cV_h^\TT(\Cbad(\alpha)) \subset B$.
By inequality~\eqref{thick}, we thus obtain that 
\[
| \Cbad(\alpha) | \leq \frac{n}{2} \kappa + \frac{\pi}{h}, 
\]
which implies that $| \Cbad(\alpha) | \leq \kappa n$ for all large $n$. 
\end{proof} 

We now enter the proofs of Theorems~\ref{snY} and~\ref{intrm}. Given $C > 0$,
we denote as $\Eop(C)$ the event
\[
\Eop(C) = \left[ \| Y \| \leq C \right] . 
\]
In the remainder of this section, the constants will be referred to by the
letters $c$, $a$, or $C$, possibly with primes or numerical indices that refer
generally to the statements (lemmas, propositions, ...) where these constants
appear for the first time. These constants do not depend on $N$. 
The typical statements where they appear are of the type
\[
\PP\left[ \left[ \cdots \leq a \right] \cap \Eop(C) \right] \leq 
  C' a + \exp(-cN) , 
\]
or others in the same vein. Often, such inequalities hold true for all $N$
large enough. This detail will not always be mentioned. 

\subsection{Proof of Theorem~\ref{snY}} 

As mentioned in Section~\ref{outline}, our starting point for proving
Theorem~\ref{snY} is the variational characterization~\eqref{comp-incomp}. 

\subsubsection{Compressible vectors}

In the following statement and in the proof, the vectors $u\in \SSS^{N+n-1}$
will always take the form $u = [v^\T \ w^\T]^\T$ with $v \in \CC^{n}$.  
\begin{proposition}[compressible vectors] 
\label{prop-comp}
There exist constants $\theta_{\ref{prop-comp}}, \rho_{\ref{prop-comp}}, 
C_{\ref{prop-comp}} \in (0,1)$, and   
$a_{\ref{prop-comp}}, c_{\ref{prop-comp}} > 0$ such that, for the set  
\[
{\cS} := \left\{ u \in \SSS^{N+n-1} \ : \  
  \| v \| \leq C_{\ref{prop-comp}} \ \text{or} \ 
  v / \| v \| \in \comp(\theta_{\ref{prop-comp}}, \rho_{\ref{prop-comp}}) 
 \right\} 
\]
(here $v / \| v \| = 0$ if $v = 0$), we have  
\[
\PP\left[ [ \inf_{u\in{\cS}} \| H u \| \leq a_{\ref{prop-comp}} ] 
 \cap \Eop(C) \right] \leq \exp(- c_{\ref{prop-comp}} n ) . 
\]
\end{proposition} 
\begin{proof}
For each $u \in \SSS^{N+n-1}$ the inequality $\| Hu \| \leq a$ for $a > 0$ 
implies 
\begin{subequations} 
\label{Hu} 
\begin{align} 
\| Y v + z w \| &\leq a, \label{Yvzw}  \\ 
\| \Omega^{-L} v + Y^* w \| &\leq a.  \label{OvYw} 
\end{align} 
\end{subequations} 
Take $a = |z| / 2$. On the event $\Eop(C)$, we obtain from~\eqref{Yvzw} that 
\[
|z|/2 \geq |z| \| w \| - \| Y \| \| v \| \geq |z| ( 1 - \| v \|) - C \| v \|,
\]
since $\sqrt{1-x^2} \geq 1 - x$ on $[0,1]$. Hence 
\[
\| v \| \geq C_{\ref{prop-comp}} \eqdef \frac{|z| / 2}{|z| + C} , 
\]
which trivially implies that  
\[
\PP\left[ \Bigl[ \inf_{u\in\SSS^{N+n-1}\, : \, \| v\| < C_{\ref{prop-comp}}  } 
  \| H u \| \leq |z|/2 \Bigr] \cap \Eop(C) \right] = 0. 
\]
Let $\cI \subset [n]$ be such that $| \cI | = \lfloor \theta_{\ref{prop-comp}} n
\rfloor$, where $\theta_{\ref{prop-comp}}$ will be chosen later.  Let $V_\cI
\subset \CC^n$ be the set of vectors $v$ such that $\| v \| \in
[C_{\ref{prop-comp}}, 1]$ and $v  / \| v \| \in \SSS_{\cI}^{n-1}$.  Let $\bs v$
be a deterministic vector in $V_\cI$, and define the random vector $\bs w = -
z^{-1} Y \bs v$. Writing $\bs u = [ \bs v^\T \ \bs w^\T]^\T$, we shall show 
that there exists a constant $c > 0$ such that for all $t > 0$ small enough, 
\begin{equation}
\label{u-fixed} 
\PP\left[ [ \| H \bs u \| \leq t] \cap \Eop(C) \right] \leq (ct)^n . 
\end{equation} 
On the event $\Eop(C)$, we have 
\begin{eqnarray*} 
\frac{C}{|z|} \geq \| \bs w \| \geq \frac{\| Y^* \bs w \|}{\| Y \|} &=& 
\frac{\| Y^* \bs w + \Omega^{-L} \bs v - \Omega^{-L} \bs v \|}{\| Y \|} \\
 &\geq& \frac{\| \bs v \| - \| Y^* \bs w + \Omega^{-L} \bs v \|}{\| Y \|} \\ 
&\geq& \frac{C_{\ref{prop-comp}}- \| Y^* \bs w + \Omega^{-L} \bs v \|}{C}  . 
\end{eqnarray*} 
By the choice of $\bs w$, we have
$[\| H \bs u \| \leq t]=[\| \Omega^{-L} \bs v + Y^* \bs w \| \leq t]$. 
Assume that $t \leq C_{\ref{prop-comp}} / 2$, and define the interval 
$J = [ C_{\ref{prop-comp}}/(2C), C/|z|] ]$. Then, 
\[
 [ \| H \bs u \| \leq t] \cap \Eop(C)] \subset 
 [ \| \bs w \| \in [ (C_{\ref{prop-comp}} - t)/C, C/|z|] ] 
 \subset [ \| \bs w \| \in J ] . 
\]
If $J = \emptyset$, then the inequality~\eqref{u-fixed} holds trivially. 
Otherwise, we write 
\[
\PP\left[ [ \| H \bs u \| \leq t] \cap \Eop(C) \right] \leq 
\PP\left[ [ \| (Y_{\cdot, \cIc})^* \bs w \| \leq t] 
  \cap [ \| \bs w \| \in J ] \right],  
\]
and for each $\ell\in \cIc$, we write 
$$y_\ell = \bs{\tilde y}_{\ell} + 
 \bs{\check y}_{\ell}\ \ \text{where}\ \ \bs{\check y}_{\ell} = \EE [ y_\ell \, | \, Y_{\cdot,\cI}].$$ 
Similarly, we write 
$$Y_{\cdot, \cIc} = \widetilde Y_{\cdot, \cIc} + \widecheck Y_{\cdot, \cIc}\ \ \text{where}\ \ 
\widecheck Y_{\cdot, \cIc} = \EE[ Y_{\cdot, \cIc} \, | \, Y_{\cdot,\cI}].$$
Since $\bs w$ is $\sigma(Y_{\cdot,\cI})$ measurable, the variables 
$\widetilde Y_{\cdot, \cIc}$ and $\bs w$ are independent, 
and we get from the last inequality that 
\[
\PP\left[ [ \| H \bs u \| \leq t] \cap \Eop(C) \right] 
\leq \sup_{w \, : \, \| w \| \in J} 
 {\mathcal L}_{\widetilde Y_{\cdot, \cIc}}
  ( (\widetilde Y_{\cdot, \cIc})^* w, t) . 
\]
For each deterministic vector $w$ such that $\| w \| \in J$, we thus need to
consider the probability law of the Gaussian vector $(\widetilde Y_{\cdot,
\cIc})^* w$ by studying its covariance matrix $\Sigma = n \EE (\widetilde
Y_{\cdot, \cIc})^* w w^* \widetilde Y_{\cdot, \cIc}$. We first bound the 
spectral norm of $\Sigma$. Writing 
$(Y_{\cdot, \cIc})^* w = (I_{|\cIc|} \otimes w^\T) 
\overline\vect(\widetilde Y_{\cdot, \cIc})$, we have 
\[
\Sigma = n (I_{|\cIc|} \otimes w^\T) 
  \EE \overline\vect(\widetilde Y_{\cdot, \cIc}) 
  \vect(\widetilde Y_{\cdot, \cIc})^\T (I_{|\cIc|} \otimes \bar w), 
\]
thus, 
\begin{align*} 
\Sigma^\T &= n (I_{|\cIc|} \otimes w^*) 
  \EE \vect(\widetilde Y_{\cdot, \cIc}) 
  \vect(\widetilde Y_{\cdot, \cIc})^* (I_{|\cIc|} \otimes w) \\
 &\leq n (I_{|\cIc|} \otimes w^*) 
  \EE \vect(Y_{\cdot, \cIc}) 
  \vect(Y_{\cdot, \cIc})^* (I_{|\cIc|} \otimes w) 
\end{align*} 
in the semidefinite ordering. Using that $Y = X \sF$, we have 
$\vect(Y_{\cdot, \cIc}) = ((\sF_{\cdot, \cIc})^\T \otimes I_N) \vect(X)$, thus, 
\[
\Sigma^\T \leq 
  n ((\sF_{\cdot, \cIc})^\T \otimes w^*) \cR 
   (\overline{\sF_{\cdot, \cIc}} \otimes w) , 
\]
which shows that 
$$
\| \Sigma \| \leq \bM \| w \|^2 \leq C_{\max} = \bM C^2 / |z |^2.
$$ 
Recall that $y_\ell = \tilde y_\ell + \check y_\ell$, and
observe that $\sigma(Y_{\cdot,\cI}) \subset \sigma(Y_\ell)$ for each $\ell\in
\cIc$.  Therefore, for an arbitrary deterministic vector $u\in\CC^N$, we have
by Jensen's inequality 
\begin{align*}
\EE [| u^* \bs{\check y}_{\ell} |^2 ] = 
\EE\left[\left| \EE [ u^* y_{\ell} \, | \, Y_{\cdot, \cI}]\right|^2\right] 
&= \EE\left[\left| 
  \EE\left[ \EE [ u^* y_{\ell} \, | \, Y_\ell] 
    \, | \, Y_{\cdot, \cI} \right] \right|^2 \right]  \\
&\leq 
  \EE \left[ \left| \EE \left[ u^* y_{\ell}\, | \, Y_\ell \right]\right|^2
  \right] \\
 &= \EE [| u^* \check y_\ell |^2]
\end{align*} 
for such $\ell$.  Since 
$$\EE y_\ell y_\ell^* = \EE \tilde y_\ell \tilde
y_\ell^* + \EE \check y_\ell \check y_\ell^* = \EE \bs{\tilde y}_{\ell}
\bs{\tilde y}_{\ell}^* +   \EE \bs{\check y}_{\ell} \bs{\check y}_{\ell}^*,$$ we
get from the previous inequality that 
$$\EE \tilde y_\ell \tilde y_\ell^* \leq
\EE \bs{\tilde y}_{\ell}  \bs{\tilde y}_{\ell}^*.$$  Choosing
$\theta_{\ref{prop-comp}} \leq 1/4$, we have that $| \cIc | \geq 3 n / 4$.
By Lemma~\ref{bad-k}, there exists $\alpha > 0$ such that 
$| \cIc \cap \Cgood(\alpha) | \geq n / 2$. With this, we have 
\[
\tr \Sigma 
 = n w^* \EE \widetilde Y_{\cdot, \cIc} (\widetilde Y_{\cdot, \cIc})^* w 
= n \sum_{k\in\cIc} w^* \EE \bs{\tilde y}_k \bs{\tilde y}_k^*  w 
\geq n \sum_{k\in\cIc\cap\Cgood} w^* \EE \tilde y_k \tilde y_k^*  w. 
\]
By Proposition~\ref{CD->S}, we thus obtain that for all large $N$,
\[
\tr \Sigma \geq \frac{\alpha}{2} |\cIc\cap\Cgood| \| w \|^2 \geq 
n \frac{\alpha}{4} \frac{C_{\ref{prop-comp}}^2}{4 C^2}. 
\] 
 Write 
$C_{\min} = \alpha C_{\ref{prop-comp}}^2 n / (16 C^2 |\cIc|)$, and let $k$
be the largest integer in $[|\cIc|]$ such that 
$s_{k-1}(\Sigma) \geq C_{\min} / 2$. Then, we obtain from Lemma~\ref{agood} 
that 
\[
k \geq \frac{C_{\min}}{2 C_{\max} - C_{\min}} |\cIc| \geq 
\frac{C_{\min}}{2 C_{\max} - C_{\min}} \frac{3n}{4} .
\]
With this at hand, we can deduce Inequality~\eqref{u-fixed} from
Lemma~\ref{gauss}--\ref{antivec}. 

Now, still fixing $\cI$ and $t$ as above, set $\rho_{\ref{prop-comp}}$ as 
\[
\rho_{\ref{prop-comp}} = \frac{t}{3 + C|z|^{-1} (2C+1)} . 
\]
Set $a_{\ref{prop-comp}} = (|z|/2) \wedge \rho_{\ref{prop-comp}}$, and let  $u
= [v^\T \ w^\T]^\T \in\SSS^{N+n-1}$ be such that $v \in
\cV_{\rho_{\ref{prop-comp}}}^{\CC^n}(V_\cI)$ and $\| H u \| \leq
a_{\ref{prop-comp}}$.  Let $\cK_{\rho_{\ref{prop-comp}}}$ be a
$\rho_{\ref{prop-comp}}$--net of $V_\cI$. By Lemma~\ref{covnb}, we can choose
$\cK_{\rho_{\ref{prop-comp}}}$ in such a way that
$|\cK_{\rho_{\ref{prop-comp}}}| \leq (3 /
\rho_{\ref{prop-comp}})^{2\theta_{\ref{prop-comp}} n}$.  By the triangle
inequality, there is $\bs v \in \cK_{\rho_{\ref{prop-comp}}}$ such that $\| v -
\bs v \| \leq 2 \rho_{\ref{prop-comp}}$. Let $\bs w = - z^{-1} Y \bs v$, and
write $\bs u = [ \bs v^\T \ \bs w^\T]^\T$.  Since $\| Y v + z w \| \leq
\rho_{\ref{prop-comp}}$ by~\eqref{Yvzw}, we have 
$$\| Y ( v - \bs v) + z (w -
\bs w) \| \leq \rho_{\ref{prop-comp}}.$$ Thus, on the event $\Eop(C)$,
$$\| w - \bs w \| \leq |z|^{-1} (
\rho_{\ref{prop-comp}} + 2\rho_{\ref{prop-comp}} \| Y \| ) \leq |z|^{-1}  (
\rho_{\ref{prop-comp}} + 2\rho_{\ref{prop-comp}} C ).$$ 
From the inequality $\| \Omega^{-L} v + Y^* w \| \leq \rho_{\ref{prop-comp}}$ 
(see~\eqref{OvYw}), we get 
$$\| \Omega^{-L} \bs v + Y^* \bs w + \Omega^{-L} (v -
\bs v) + Y^* (w - \bs w) \| \leq \rho_{\ref{prop-comp}}.$$ Thus, 
$$\| \Omega^{-L}
\bs v + Y^* \bs w \| \leq ( 3 + C |z|^{-1} (1+2C)) \rho_{\ref{prop-comp}} = t.$$
This implies that $\| H \bs u \| \leq t$. As a consequence, 
\[
\left[ \exists u \in \SSS^{N+n-1} \, : \, 
  v \in \cV_{\rho_{\ref{prop-comp}}}(V_\cI), \  
\| H u \| \leq a_{\ref{prop-comp}} \right] 
\subset 
\left[ \exists \bs v \in \cK_{\rho_{\ref{prop-comp}}} \, : \, 
\| H [ \bs v^\T, \, - z^{-1} (Y \bs v)^\T ]^\T \| \leq t \right] . 
\]
Applying the union bound and using Inequality~\eqref{u-fixed}, we therefore 
get that 
\[
\PP\left[ 
 \Bigl[ \inf_{u \in \SSS^{N+n-1} \, : \,  \| v \| \geq C_{\ref{prop-comp}}, 
v \in \cV_{\rho_{\ref{prop-comp}}}(V_\cI)} \| H u \| \leq a_{\ref{prop-comp}} 
   \Bigr] 
  \cap \Eop(C)  \right] \leq 
 (3/\rho_{\ref{prop-comp}})^{2\theta_{\ref{prop-comp}} n} (ct)^n. 
\]
Now, considering all the sets $\cI \subset [n]$ such that 
$| \cI | = \lfloor \theta_{\ref{prop-comp}} n \rfloor$, and using the bound
${{n}\choose{m}} \leq (en/m)^m$ along with the union bound again, we get that 
\begin{multline*} 
\PP\left[ 
 \Bigl[ \inf_{u\in\SSS^{N+n-1} \, : \, \| v \| \geq C_{\ref{prop-comp}}, 
v \in \comp(\theta_{\ref{prop-comp}}, \rho_{\ref{prop-comp}})}  \| H u \| \leq 
   a_{\ref{prop-comp}} \Bigr] \cap \Eop(C)  \right] \\
 \leq 
 (e/\theta_{\ref{prop-comp}})^{\theta_{\ref{prop-comp}}n} 
  (3/\rho_{\ref{prop-comp}})^{2\theta_{\ref{prop-comp}} n} (ct)^n . 
\end{multline*} 
The proof is completed by choosing $\theta_{\ref{prop-comp}}$ small enough. 
\end{proof} 

\subsubsection{Incompressible vectors} 

We now consider the action of $H$ on the vectors $u\in\SSS^{N+n-1}$ that belong
to the complement of the set $\cS$ of Proposition~\ref{prop-comp} in the 
unit-sphere. 
\begin{proposition}[incompressible vectors for the smallest singular value] 
\label{prop-incomp} 
There exists a constant $c_{\ref{prop-incomp}} > 0$ such that for 
$\varepsilon > 0$ arbitrarily small, 
\[
\PP\left[ \left[ \inf_{u\in{\SSS^{N+n-1}\setminus\cS}} \| H u \| 
  \leq n^{-3/2} t \right] \cap \Eop(C) \right] 
  \leq \varepsilon t + \exp(- c_{\ref{prop-incomp}} \varepsilon^2 N ) . 
\]
\end{proposition} 

The remainder of this subsection is devoted to the proof of this proposition. As
in the proof of Proposition~\ref{prop-comp}, a vector $u\in \SSS^{N+n-1}$ will
be written $u = [v^\T \ w^\T]^\T$ with $v = [ v_0,\ldots, v_{n-1} ]^\T \in
\CC^{n}$.  When $u\in\SSS^{N+n-1} \setminus \cS$, the vector $\tilde v = v / \|
v \| = [ \tilde v_0,\ldots, \tilde v_{n-1}]^\T$ belongs now to
$\incomp(\theta_{\ref{prop-comp}}, \rho_{\ref{prop-comp}})$. Associate with any such vector $u$, the set 
\[
\cJ_u = \{ k\in [n] \, : \, 
 |v_k | \geq \rho_{\ref{prop-comp}} C_{\ref{prop-comp}} / \sqrt{2n} 
    \} .
\]
Then, $| \cJ_u | > \theta_{\ref{prop-comp}} n $. Indeed, the set 
$\widetilde \cJ_v = \{ k \in [n] \, : \, 
 |\tilde v_k | \geq \rho_{\ref{prop-comp}} / \sqrt{2n} \}$ 
satisfies $| \widetilde \cJ_v | > \theta_{\ref{prop-comp}} n $. To see this,
denote as $\Pi_{\widetilde \cJ_v}$ the orthogonal projector on the vectors 
supported by $\widetilde \cJ_v$, and check that 
 $\left\| \tilde v - \Pi_{\widetilde \cJ_v} (\tilde v) / 
  \| \Pi_{\widetilde \cJ_v} (\tilde v) \| \right\| \leq \rho_{\ref{prop-comp}}$.
If $| \widetilde \cJ_v | \leq \theta_{\ref{prop-comp}} n$, we get a 
contradiction to the fact that  
$\tilde v \in \incomp(\theta_{\ref{prop-comp}}, \rho_{\ref{prop-comp}})$. 
It remains to check that if $u \in \SSS^{N+n-1} \setminus \cS$, then 
$\widetilde \cJ_v \subset \cJ_u$.

Now, choose $\alpha_{\ref{prop-incomp}} > 0$ small enough so that the set
$\Cgood(\alpha)$ from Lemma~\ref{bad-k} satisfies $|
\Cgood(\alpha_{\ref{prop-incomp}}) | \geq (1 - \theta_{\ref{prop-comp}} / 2)n$.
Observe that with this choice, $| \cJ_u \cap \Cgood(\alpha_{\ref{prop-incomp}})
| \geq \theta_{\ref{prop-comp}} n / 2$. 

We now use the canvas of the proof of \cite[Lem.~3.5]{rud-ver-advmath08} to
relate the infimum over the incompressible vectors with the distance between
column subspaces of $H$. Specifically, for each 
$u\in\SSS^{N+n-1}\setminus\cS$, we have 
\[
\| H u \| \geq \max_{k\in \cJ_u \cap \Cgood(\alpha_{\ref{prop-incomp}})} 
 | v_k | \dist( h_k, H_{k} ) \\
 \geq  \frac{\rho_{\ref{prop-comp}} C_{\ref{prop-comp}}}{\sqrt{2 n}}  
    \max_{k\in \cJ_u \cap \Cgood(\alpha_{\ref{prop-incomp}})} 
    \dist( h_k, H_{k} ). 
\]
Thus, 
\[
\left[ \inf_{u\in\SSS^{N+n-1}\setminus\cS} \| H u \| \leq 
 \frac{\rho_{\ref{prop-comp}} C_{\ref{prop-comp}} t}{\sqrt{2} n^{3/2}} \right] 
\subset 
\left[ \inf_{u\in\SSS^{N+n-1}\setminus\cS} \ 
 \max_{k\in \cJ_u \cap \Cgood(\alpha_{\ref{prop-incomp}})} 
 \dist( h_k, H_{k} ) \leq \frac tn \right] . 
\]
Denoting as $\cE$ the event at the right hand side of this inclusion, we
have from what precedes that 
\[
\1_{\cE} \leq 
 \frac{2}{\theta_{\ref{prop-comp}} n} 
 \sum_{k\in\Cgood(\alpha_{\ref{prop-incomp}})} 
   \1_{\left[ \dist( h_k, H_{k} ) \leq t / n \right]} , 
\]
which implies that 
\begin{multline} 
\PP\left[\left[ \inf_{u\in\SSS^{N+n-1}\setminus\cS} \| H u \| \leq 
 \frac{\rho_{\ref{prop-comp}} C_{\ref{prop-comp}} t}{\sqrt{2} n^{3/2}} \right] 
  \cap \Eop(C) \right] \\ 
\leq \frac{2}{\theta_{\ref{prop-comp}} n} 
 \sum_{k\in\Cgood(\alpha_{\ref{prop-incomp}})} 
   \PP\left[ \left[ \dist( h_k, H_{k} ) \leq t / n \right] \cap 
  \Eop(C) \right] . 
\label{incomp-dist} 
\end{multline} 
A workable formula for $\dist(h_k, H_{k})$ is provided by the following lemma. 
\begin{lemma}
\label{lm-dist}
Let $M\in\CC^{m\times m}$, and partition this matrix as
\[
M = \begin{bmatrix} m_0 & M_{-0} \end{bmatrix} = 
 \begin{bmatrix} m_{00} & m_{01} \\ m_{10} & M_{11} \end{bmatrix} , 
\]
where $m_0$ is the first column of $M$, $M_{-0}$ is the submatrix that remains
after extracting $m_0$, and $m_{00}$ is the first element of the vector $m_0$.
Assume that $M_{11}$ is invertible. Then,
\[
\dist(m_0, M_{-0}) = \frac{|m_{00} - m_{01} M^{-1}_{11} m_{10}|} 
     {\sqrt{1 + \| m_{01} M^{-1}_{11} \|^2}} . 
\]
\end{lemma} 
\begin{proof} 
Write  $\dist(m_0, M_{-0})^2 = m_0^* ( I - M_{-0} (M_{-0}^* M_{-0})^{-1}
M_{-0}^*) m_0$, and observe that $M_{-0}^* M_{-0} = M_{11}^* M_{11} +
m_{01}^* m_{01}$. Letting $v = M_{11}^{-*} m_{01}^*$, we obtain by the 
Sherman-Morrison-Woodbury identity that 
\[
(M_{-0}^* M_{-0})^{-1} = M_{11}^{-1} \Bigl( 
I - \frac{v v^*}{1 + \| v \|^2} \Bigr) M_{11}^{-*} , 
\]
thus, 
\begin{align*} 
 I - M_{-0} (M_{-0}^* M_{-0})^{-1} M_{-0}^* 
&= I - \begin{bmatrix} m_{01} \\ M_{11} \end{bmatrix} 
 M_{11}^{-1} \Bigl( 
I - \frac{v v^*}{1 + \| v \|^2} \Bigr) M_{11}^{-*} 
 \begin{bmatrix} m_{01}^* & M_{11}^* \end{bmatrix} \\
&= \frac{1}{1 + \| v \|^2} 
 \begin{bmatrix} 1 & -v^* \\ - v & v v^* \end{bmatrix} 
 = \frac{1}{1 + \| v \|^2}
 \begin{bmatrix} 1 \\ -v \end{bmatrix} 
 \begin{bmatrix} 1 & -v^* \end{bmatrix} . 
\end{align*} 
This leads to 
\[
\dist(m_0, M_{-0})^2 = \frac{1}{1+\| v \|^2} 
  \Bigl| \begin{bmatrix} 1 & -v^* \end{bmatrix} 
 \begin{bmatrix} m_{00} \\ m_{10} \end{bmatrix}  \Bigr|^2 , 
\]
which is the required result. 
\end{proof} 

For convenience, let is recall the matrices $G_k$ and $G_k^{-1}$ from Section~\ref{outline}: 
\[
G_k = \begin{bmatrix} \Omega_k^{-L} & Y_k^* \\ Y_k & z \end{bmatrix} 
 , \ \ 
G_k^{-1} = \begin{bmatrix} E_k & F_k \\ P_k & D_k \end{bmatrix} 
 .
\]
Applying the above lemma on the matrix $H$ with the column $k$ being used 
instead of Column zero, we get that 
\begin{equation}
\label{dist-complex} 
\dist(h_k, H_{k}) = 
\frac{\left| \omega^{-kL} - y_k^* D_k y_k \right|}{
\sqrt{1 + \left\| \begin{bmatrix} 0 & y_k^* \end{bmatrix} G_k^{-1}\right\|^2}}  
= 
\frac{\left| \omega^{-kL} - y_k^* D_k y_k \right|}{
 \sqrt{1 + \| y_k^* P_k \|^2 + \| y_k^* D_k \|^2}} . 
\end{equation} 
We shall use the notation
\[
\text{Num}_k = \left| \omega^{-kL} - y_k^* D_k y_k \right|, 
 \quad \text{and} \quad 
\text{Den}_k = \sqrt{1 + \| y_k^* P_k \|^2 + \| y_k^* D_k \|^2} .
\]
\begin{lemma}
\label{PR}
The following facts hold true. 
\begin{enumerate}
\item\label{g01G} 
Assume that $k \in 
\Cgood(\alpha_{\ref{prop-incomp}})$. Then, 
there exist $c_{\ref{PR}}, C_{\ref{PR}} > 0$ such that 
\[
\PP\left[ 
 \left[ \left\| \begin{bmatrix} 0 & y_k^* \end{bmatrix} G_k^{-1}\right\| 
 \leq C_{\ref{PR}} \right] \cap \Eop(C) \right] 
 \leq \exp(- c_{\ref{PR}} N ) . 
\] 
\item\label{PRHS} 
On $\Eop(C)$, for each $k \in [n]$, we have $\| P_k \|_\HS \leq C \| D_k \|_\HS$. 
\end{enumerate} 
\end{lemma}
\begin{proof}
\begin{enumerate}
\item Consider  
\[ 
\| y_k \| = 
 \left\| \begin{bmatrix} 0 & y_k^* \end{bmatrix} G_k^{-1} G_k\right\| 
 \leq \left\| \begin{bmatrix} 0 & y_k^* \end{bmatrix} G_k^{-1} \right\| \, 
     \| G_k \| . 
\]
On $\Eop(C)$, we have $\| G_k \| \leq C + |z| + 1$. Moreover, since 
$k \in \Cgood(\alpha_{\ref{prop-incomp}})$, and since 
$\| n \EE y_k y_k^* - S(e^{2\imath\pi k / n}) \|$ is small by 
Lemma~\ref{cov-y}, there exist two constants $c', C' > 0$ such that 
\[
\PP\left[ \| y_k \| \leq C' \right] \leq \exp(- c' N)
\]
by Lemma~\ref{gauss}--\ref{normgauss}. The proof of Item~\ref{g01G} is complete. 

\item We show that for each constant non-zero vector $u \in \CC^N$,
\[
\| u^* P_k \| \leq C \| u^* D_k \|. 
\]
 The result follows then from
the identity $\| M \|_\HS^2 = \sum_{k} \| e_k^* M \|^2$, valid for each 
matrix $M$. The vectors $v^* = u^* P_k$ and $w^* = u^* D_k$ satisfy
\[
\begin{bmatrix} v^* & w^* \end{bmatrix} = \begin{bmatrix} 0 & u^*   
 \end{bmatrix} G_k^{-1}, 
\]
and hence, 
\[
\begin{bmatrix} 0 & u^* \end{bmatrix} =  
\begin{bmatrix} v^* & w^* \end{bmatrix}  
  \begin{bmatrix} \Omega_k^{-L} & Y_k^* \\ Y_k & z \end{bmatrix} . 
\]
In particular, $v^* \Omega_k^{-L} + w^* Y_k = 0$, and this shows that 
$\| v \| \leq C \| w \|$ on $\Eop(C)$.  
\end{enumerate} 
\end{proof} 
Lemma~\ref{PR} will be used to control the value of $\text{Den}_k$. 
Fix an arbitrary small real number $\be > 0$, and define the event 
\[
\Eden(\be) = \left[ \text{Den}_k^2 \geq \be \| D_k \|_\HS^2 \right].   
\]
\begin{lemma}[denominator not too large] 
\label{bnd-den} 
Assume that $k \in \Cgood(\alpha_{\ref{prop-incomp}})$. Then, there exists 
$c_{\ref{bnd-den}} > 0$ such that  for all large $n$, $\PP[ \Eden(\be) \cap \Eop(C) ] \leq \exp(-
c_{\ref{bnd-den}} \be n )$. 
\end{lemma}
\begin{proof}
Defining the event 
$\cE = \left[ \text{Den}_k^2 \geq (C_{\ref{PR}}^{-2} + 1) 
\left\| \begin{bmatrix} 0 & y_k^* \end{bmatrix} G_k^{-1}\right\|^2 \right]$, 
we have 
\[ 
\PP\left[ \cE \cap \Eop(C) \right] = 
\PP\left[ \left[ C_{\ref{PR}}^{2} \geq 
  \left\| \begin{bmatrix} 0 & y_k^* \end{bmatrix} G_k^{-1}\right\|^2 
  \right] \cap \Eop(C) \right]  
 \leq \exp(- c_{\ref{PR}} n)  
\] 
by Lemma~\ref{PR}--\ref{g01G}.  Writing $\cE'(\be) = \left[ \left\|
\begin{bmatrix} 0 & y_k^* \end{bmatrix} G_k^{-1}\right\|^2 \geq \be \| D_k
\|_\HS^2 \right]$, we have   
\begin{align*}
\PP\left[ \cE'(\be) \cap \Eop(C) \right]  
&= 
\PP\left[ \left[\| y_k^* P_k \|^2 + \| y_k^* D_k \|^2 
      \geq \be \| D_k \|_\HS^2 \right] \cap \Eop(C) \right] \\
&\leq 
\PP\left[ \| y_k^* P_k \|^2 \geq \be \| P_k \|_\HS^2 / (2 C^2) \right] 
  + \PP\left[ \| y_k^* D_k \|^2 \geq \be \| D_k \|_\HS^2 / 2 \right] , 
\end{align*} 
where we used Lemma~\ref{PR}--\ref{PRHS}.~in the first inequality. 

Note that 
$$\| y_k^* D_k \|^2 \leq 2\|\tilde y_k^* D_k\|^2 + 2 \| \check y_k^* D_k
\|^2 \ \ \text{and}\ \ \| \check y_k^* D_k\|\leq \| \check y_k \|^2 \| D_k \|_\HS.$$  Hence there exists $c_N \to\infty$ such that 
\begin{align*} 
\PP\left[ \| y_k^* D_k \|^2 \geq \be \| D_k \|_\HS^2 / 2 \right]
 &\leq 
 \PP\left[ \|\tilde y_k^* D_k\|^2 \geq \be \| D_k \|_\HS^2 / 8 \right] + 
\PP\left[ \|\check y_k^* \|^2 \geq \be / 8 \right] \\
&\leq \PP\left[ \|\tilde y_k^* D_k\|^2 \geq 
   \be \| S_k^{1/2} D_k \|_\HS^2 / (8 \bM) \right] + e^{-c_N n} \\
&\leq e^{ 1 -\be n / (16 \bM)} + e^{-c_N n} , 
\end{align*} 
where we used Lemma~\ref{gauss}--\ref{small-var} along with 
Proposition~\ref{CD->S} in the second inequality, and 
Lemma~\ref{gauss}--\ref{larget} in the third one. 
We have a similar bound for  
$\PP\left[ \| y_k^* P_k \|^2 \geq \be \| P_k \|_\HS^2 /(2C^2) \right]$. 
Consequently, there exists $c > 0$ such that 
$\PP\left[ \cE'(\be) \cap \Eop(C) \right] \leq e^{-c\be n}$ for all large $N$. 
 
Observing that $\Eden(\be) \cap \cE^{\text{c}} \subset \cE'(\be / 
(C_{\ref{PR}}^{-2} + 1))$, we obtain as
a consequence of these inequalities that there exists $c_{\ref{bnd-den}} > 0$
such that for all large $N$, 
\begin{align*} 
\PP\left[ \Eden(\be) \cap \Eop(C) \right] &\leq 
 \PP\left[\Eden(\be) \cap \cE^{\text{c}} \cap \Eop(C) \right] +  
 \PP\left[\cE \cap \Eop(C) \right]  \\ 
&\leq \PP\left[\cE'(\be / (C_{\ref{PR}}^{-2} + 1)) \cap \Eop(C) \right] 
    + \PP\left[\cE \cap \Eop(C) \right] \\
&\leq \exp(- c_{\ref{bnd-den}} \be n ). 
\end{align*} 
The proof of the lemma is now complete.
\end{proof} 

Now we get back to the expression in (\ref{incomp-dist}).  
As in Subsection~\ref{outline}, we use the shorthand notation $S_k =
S(e^{2\imath\pi k / n})$.  Given $k\in \Cgood(\alpha_{\ref{prop-incomp}})$, we
have 
\begin{align*} 
& \PP\left[ [ \dist(h_k, H_{k}) \leq t /n ] \cap \Eop(C) \right] = 
\PP\left[ \left[ \text{Num}_k / \text{Den}_k \leq  t/n \right] 
    \cap \Eop(C) \right] \\
&\leq  
 \PP\left[ \left[ \frac{\text{Num}_k}{\text{Den}_k} \leq t /n \right] 
 \cap \Eden(\be)^{\text{c}} \cap \Eop(C) \right]  
 + \PP\left[ \Eden(\be) \cap \Eop(C) \right] \\
&\leq 
 \PP\left[ \frac{n\text{Num}_k}{\| D_k \|_\HS} \leq 
         t \sqrt\be \right] + \exp(- c_{\ref{bnd-den}} \be n ) \\ 
&\leq 
 \PP\left[ \frac{n\text{Num}_k}{\| S_k^{1/2} D_k S_k^{1/2} \|_\HS} \leq 
         t \sqrt\be / \alpha_{\ref{prop-incomp}} \right] + 
  \exp(- c_{\ref{bnd-den}} \be n ) . 
\end{align*} 
Writing $\text{Num}_k = \left| \omega^{-kL} + (\tilde y_k + \check y_k)^* 
 D_k (\tilde y_k + \check y_k) \right|$, we obtain from 
Lemma~\ref{gauss}--\ref{antifq} and Proposition~\ref{CD->S} that for some constant $c > 0$, 
\begin{align*} 
\PP\left[ [ \dist(h_k, H_{k}) \leq t /n ] \cap \Eop(C) \right] 
 &\leq 
\EE_{Y_k} \cL_{\tilde y_k}\Bigl( \frac{(\tilde y_k + \check y_k)^* 
 D_k (\tilde y_k + \check y_k)}{\| S_k^{1/2} D_k S_k^{1/2} \|_\HS},  
 \frac{t \sqrt\be}{\alpha_{\ref{prop-incomp}}} \Bigr) 
 + \exp(- c_{\ref{bnd-den}} \be n ) \\ 
&\leq c t \sqrt\be + \exp(- c_{\ref{bnd-den}} \be n ). 
\end{align*} 
Proposition~\ref{prop-incomp} is obtained by combining the last inequality with
Inequality~\eqref{incomp-dist}.  

\subsubsection*{Theorem~\ref{snY}: End of proof.} 
The proof is now completed by combining Proposition~\ref{prop-comp}, Proposition~\ref{prop-incomp}, and 
Lemma~\ref{spec_norm}.  

\subsection{Proof of Theorem~\ref{intrm}} 
\label{prf-intrm} 

Throughout the proof, we fix $k\in[\lfloor N^\beta\rfloor, \lfloor N/2\rfloor]$.
None of the constants that will appear in the course of the proof will depend
on $k$. 

The following lemma refines Inequality~\eqref{lin}. 
\begin{lemma}
\label{Htronc} 
For any $k \in [N]$, we have 
\[
s_{N+n-k-1}(H_{\cdot,[N+n-k]}) \leq s_{N-k-1}(Y \Omega^L Y^* - z). 
\]
\end{lemma}
\begin{proof}
Write any vector $u \in \CC^{N+n-k}$ as $u = [v^\T \ w^\T ]^\T$ with 
$w \in \CC^{N-k}$. Writing 
\[
H_{\cdot,[N+n-k]} u = 
\begin{bmatrix} 
  \Omega^{-L} & (Y_{[N-k],\cdot})^* \\
  &\\
	Y & z (I_N)_{\cdot,[N-k]} 
\end{bmatrix} \begin{bmatrix} v \\ w \end{bmatrix} = 
\begin{bmatrix} 
  \Omega^{-L} v + (Y_{[N-k],\cdot})^* w \\
  & \\ 
	Y v + z (I_N)_{\cdot,[N-k]} w 
\end{bmatrix}  , 
\]
we have 
\begin{align*} 
s_{N+n-k-1}(H_{\cdot,[N+n-k]}) &= 
  \min_{u \in \SSS^{N+n-k-1}} \| H_{\cdot,[N+n-k]} u \| \\ 
 &\leq \min_{u \in \SSS^{N+n-k-1} \, : \, 
 v = - \Omega^L (Y_{[N-k],\cdot})^* w} \| H_{\cdot,[N+n-k]} u \| \\ 
 &\leq \min_{w \in \SSS^{N-k-1}} 
  \left\| \left( z I_N - Y \Omega^L Y^* \right)_{\cdot, [N-k]} w \right\| .  
\end{align*} 
On the other hand, using the variational characterization of the singular 
values of a matrix, see \cite[Th.~4.3.15]{HorJoh90}, we can write  
\[
 \min_{w \in \SSS^{N-k-1}} 
  \left\| \left( Y \Omega^L Y^* - z I_N \right)_{\cdot,[N-k]} w \right\| 
 \leq s_{N-k-1}(Y \Omega^L Y^* - z I_N) , 
\]
and hence the result follows.  
\end{proof} 
Similar to what we did for controlling the smallest singular value of $H$, we
use the characterization $s_{N+n-k-1}(H_{\cdot,[N+n-k]}) = \min_{u \in
\SSS^{N+n-k-1}} \| H_{\cdot,[N+n-k]} u \|$, and we partition the set
$\SSS^{N+n-k-1}$ into sets of compressible and incompressible vectors.
So, write any vector $u\in\SSS^{N+n-k-1}$ as $u = [v^\T \ w^\T
]^\T$ with $w \in \CC^{N-k}$. Let the set
$\cS\subset\SSS^{N+n-k-1}$ be as in the statement of
Proposition~\ref{prop-comp} above. We of course have 
\begin{equation}
\label{charsmall} 
s_{N+n-k-1}(H_{\cdot,[N+n-k]}) = 
 \inf_{u \in \cS} \| H_{\cdot,[N+n-k]}u \| \wedge 
          \inf_{u \in \SSS^{N+n-k-1}\setminus\cS} \| H_{\cdot,[N+n-k]}u \| . 
\end{equation} 
It can be readily checked that Proposition~\ref{prop-comp} remains true, once 
we change the values of the constants that appear in its statement as needed. 
So, finally, the contribution of the term $\inf_{u \in \cS} \|
H_{\cdot,[N+n-k]}u \|$ has been estimated, and we are left with the incompressible vectors.
Regarding these, we have the following proposition. 
\begin{proposition}[incompressible vectors for the small singular values] 
\label{incomp-intrm}
There exist constants $a_{\ref{incomp-intrm}} > 0$ and 
$c_{\ref{incomp-intrm}} > 0$ such that 
for all $N$ larger than an integer that is independent of 
$k\in[\lfloor N^\beta\rfloor, \lfloor N /2\rfloor]$, 
\[
\PP\left[ \left[ \inf_{u\in{\SSS^{N+n-k-1}\setminus\cS}} 
  \| H_{\cdot,[N+n-k]} u \| \leq a_{\ref{incomp-intrm}} \sqrt{k} / N \right] 
   \cap \Eop(C) \right] \leq \exp(- c_{\ref{incomp-intrm}} k ). 
\]
\end{proposition} 

Remember the definition of the set $\Cgood(\alpha) \subset[n]$ in
Lemma~\ref{bad-k} above.  To prove this proposition, we begin by mimicking the
argument that lead to Inequality~\eqref{incomp-dist} above. Specifically, there
exist positive constants $c$, $c'$ and $\alpha$ such that 
\begin{multline} 
\PP\left[\left[ \inf_{u\in\SSS^{N+n-k-1}\setminus\cS} 
 \| H_{\cdot,[N+n-k]} u \| \leq 
 \frac{c t}{n^{1/2}} \right] 
  \cap \Eop(C) \right] \\ 
\leq \frac{c'}{n} 
 \sum_{\ell\in\Cgood(\alpha)} 
   \PP\left[ \left[ 
  \dist( h_\ell, H_{\cdot,[N+n-k]\setminus\{\ell\}} ) \leq t \right] \cap 
  \Eop(C) \right] , 
\label{spread-intrm} 
\end{multline} 
and $|\Cgood(\alpha)|$ is of order $N$. 

This inequality calls for the following remark. 
\begin{remark}
\label{rem-intrm} 
By checking the structure of the vector $h_\ell$ and the matrix
$H_{\cdot,[N+n-k]\setminus\{\ell\}}$ (see below), one can intuitively infer
that with high probability, $\dist( h_\ell,
H_{\cdot,[N+n-k]\setminus\{\ell\}})^2$ is of order
$\codim(H_{\cdot,[N+n-k]\setminus\{\ell\}})/N = (k-1)/N$.  Taking $t =
\sqrt{k/N}$ and recalling the characterization~\eqref{charsmall}, we get that
$s_{N+n-k-1} \gtrsim \sqrt{k}/N$, which is what is predicted by
Theorem~\ref{intrm}. However, one naturally expects that these singular values
grow linearly in $k$ (as $k/N$) 
which would make
the result of this theorem sub-optimal.  Actually, this sub-optimality is due
in the first place to Inequality~\eqref{spread-intrm}, which is too conservative for
obtaining optimal bounds on the small singular values that we are dealing with
here. 

In~\cite{rud-ver-cpam09}, a tighter inequality is used to
control the smallest singular value of a rectangular matrix. Obtaining a corresponding inequality 
appears to be a quite involved task in the present context.  
\end{remark} 

\begin{proof}[Proof of Proposition~\ref{incomp-intrm}] 
We need to bound the summands at the right hand side of
Inequality~\eqref{spread-intrm}. To this end, assume for notational simplicity
that $\ell= 0$ (assuming without loss of generality that $0\in\Cgood(\alpha)$). 
In this case,  
\[
h_0 = \begin{bmatrix} 1 \\ 0_{n-1} \\ y_0 \end{bmatrix}, 
\]
and 
\[
H_{\cdot,[N+n-k]\setminus\{0\}} = \begin{bmatrix} 
0_{n-1}^\T   & (y_0)_{[N-k]}^* \\  
& \\
(\Omega^{-L})_{[1:n],[1:n]} & (Y_{[N-k],[1:n]})^* \\
& \\
Y_{\cdot, [1:n]} & 
              \begin{bmatrix} z I_{N-k} \\ 0_{k\times (N-k)}\end{bmatrix}  
\end{bmatrix} \in \CC^{(N+n)\times(N+n-k-1)}. 
\]
Write 
\[
\underline h_0 = \begin{bmatrix} 0_{n-1} \\ y_0 \end{bmatrix}  
  \in \CC^{N+n-1} , 
\]
and let $\underline H_0$ be the matrix obtained by taking out the row $0$ from 
$H_{\cdot,[N+n-k]\setminus\{0\}}$, \emph{i.e.}, 
\[
\underline H_0 = 
\begin{bmatrix} 
(\Omega^{-L})_{[1:n],[1:n]} & (Y_{[N-k],[1:n]})^* \\
& \\
Y_{\cdot, [1:n]} & 
              \begin{bmatrix} z I_{N-k} \\ 
							& \\
							0_{k\times (N-k)}\end{bmatrix}  
\end{bmatrix} \in \CC^{(N+n-1)\times(N+n-k-1)}. 
\]
Note that $\underline h_0$ and $\underline H_0$ are ``almost'' independent. However, this is not so for $h_0$ and $H_{\cdot,[N+n-k]\setminus\{0\}}$. 
Let $a \in \CC^{N+n-k-1}$ be the vector such that
$H_{\cdot,[N+n-k]\setminus\{0\}} a = \Pi_{H_{\cdot,[N+n-k]\setminus\{0\}}}
h_0$. Then, 
\[
\dist(h_0, H_{\cdot,[N+n-k]\setminus\{0\}})^2 = 
 \| h_0 - H_{\cdot,[N+n-k]\setminus\{0\}} a \|^2 
 \geq \| \underline h_0 - \underline H_0 a \|^2 
 \geq \dist(\underline h_0, \underline H_0)^2 .
\]
Consider again the decomposition $y_0 = \tilde y_0 + \check y_0$ with 
$\check y_0 = \EE[ y_0 \, | \, Y_{\cdot,[1:n]} ]$, and write 
\[
\tilde h_0 = \begin{bmatrix} 0_{n-1} \\ \tilde y_0  \end{bmatrix} , 
\quad\text{and}\quad 
\check h_0 = \begin{bmatrix} 0_{n-1} \\ \check y_0  \end{bmatrix} 
\quad \in \CC^{N+n-1} . 
\]
Then, writing $G_0 = \begin{bmatrix} \check h_0 & \underline H_0
 \end{bmatrix}$, we have 
\[
\dist(\underline h_0, \underline H_0) = 
\dist(\tilde h_0 + \check h_0, \underline H_0) \geq 
\dist(\tilde h_0 + \check h_0, G_0 ) = \dist(\tilde h_0, G_0 ).  
\]
Observe that $\tilde h_0$ and $G_k$ are independent, and that $N+n-k-1 \leq
\rank(G_0) \leq N+n-k$ with probability one.  Let $r = N+n-1 - \rank(G_0)$, and
let $A = [ V^\T \ W^\T ]^\T \in \CC^{(N+n-1)\times r}$ be an isometry
matrix such that $AA^* = \Pi_{G_0}^\perp$. Here the partitioning of $A$ is such
that $V\in\CC^{(n-1)\times r}$, and $W\in\CC^{N\times r}$. Note that 
$r \in \{k-1,k\}$ w.p.~1. 
Since $AA^* = \Pi_{G_0}^\perp$, it holds that $\underline H_0^* A = 0$,
which can be elaborated as 
\begin{align} 
0 &= (\Omega^{L})_{[1:n], [1:n]} V + (Y_{\cdot, [1:n]})^* W 
 \label{VW} \\ 
0 &= Y_{[N-k], [1:n]} U  + 
  \begin{bmatrix} \bar z I_{N-k} & 0_{(N-k)\times k} \end{bmatrix} W . 
 \nonumber
\end{align} 
Assume that $\| W \|_\HS^2  < ar$ with $a = / (1+C^2)$. 
By Equation~\eqref{VW}, on the event 
$\Eop(C)$, we have 
$\| V \|_\HS^2 \leq C^2 a r$, which contradicts the identity 
$\| V \|_\HS^2 + \| W \|_\HS^2 = r$. Thus,  
\[
\PP\left[ \left[\| W \|_\HS^2  <  a r \right] \cap \Eop(C)\right] = 0.
\]
Using this result, we now provide a control over $\dist(\tilde h_0, G_0) = 
 \| W^* \tilde y_0 \|$. It is clear that $\tilde y_0$ and $W$ are independent.
On the event $[\| W \|_\HS^2 \geq a r]$, we know by Lemma~\ref{agood} that
\[
s_{\lfloor\frac{a}{2-a} r\rfloor -1}(W)^2 \geq a /2 . 
\]
Moreover, since $0 \in \Cgood(\alpha)$, Proposition~\ref{CD->S}, with the
help of Assumptions~\ref{prop-S} and~\ref{log-S}, shows that 
$n \EE \tilde y_0 \tilde y_0^* \geq (\alpha / 2) I_N$ for all $N$ large. 
Consequently, on the event $[\| W \|_\HS^2 \geq a r]$, the conditional 
distribution of $W^* \tilde y_k$ with respect to the sigma-field $\sigma(W)$ 
is Gaussian with a covariance matrix 
$\Sigma\in\cH_+^r$ such that $s_{\lfloor\frac{a}{2-a} r\rfloor -1}(\Sigma) 
\geq a\alpha/(4n)$ for all large $N$. Put 
$t^2 = a\alpha \lfloor\frac{a}{2-a} r \rfloor / (8n) \sim r/n$. By what 
precedes and by Lemma~\ref{gauss}--\ref{normgauss}, we have 
\begin{align*}
 &  \PP\left[ \left[ 
  \dist( h_0, H_{\cdot,[N+n-k]\setminus\{0\}} )^2 \leq t^2 \right] \cap 
  \Eop(C) \right] \\
&\leq \PP\left[ \left[ \dist( \tilde h_0, G_0 )^2 \leq t^2 \right] 
     \cap \Eop(C) \right] \\
&\leq \PP\left[ \left[ \dist( \tilde h_0, G_0 )^2 \leq t^2 \right] \cap 
  \left[ \| W \|_\HS^2 \geq a r \right] \right] 
 + \PP\left[ \left[ \| W \|_\HS^2 < a r \right] \cap \Eop(C) \right] \\
&\leq \exp(- c_{\ref{gauss},\ref{normgauss}} 
 \lfloor\frac{a}{2-a} r \rfloor ). 
\end{align*} 
The proof of Proposition~\ref{incomp-intrm} is completed by using 
Inequality~\eqref{spread-intrm}. 
\end{proof}

\subsubsection*{Proof of Theorem~\ref{intrm}} 
To complete the proof, we need to just combine Lemma~\ref{Htronc} with the characterization~\eqref{charsmall}, 
apply Propositions~\ref{prop-comp} and~\ref{incomp-intrm} 
respectively to the two terms in this characterization, and use 
Lemma~\ref{spec_norm} to bound $\PP[\Eop(C)]$.


\section{Proofs regarding the singular value distribution of 
 $\widehat R_L^{(N)} - z$} 
\label{nu-N} 

\subsection{Proof of Theorem~\ref{sys}} 

We recall that $\cM_+^m = \{ M \in \CC^{m\times m}, \ \Im M > 0 \}$.   
\begin{lemma}
\label{im>0} 
Given an integer $m > 0$, let $M \in \cM_+^m$. Then $M$ is invertible, and 
$\| M^{-1} \| \leq \| (\Im M)^{-1} \|$. Moreover, $- M^{-1} \in \cM_+^m$. 
\end{lemma}
\begin{proof}
For each non-zero vector $u$, we have 
\[
\| u \| \, \| Mu \| \geq | u^* M u | = 
 | u^* \Re M u + \imath u^* \Im M u | \geq u^* \Im M u > 0 , 
\]
and hence  $M$ is invertible. For an arbitrary non-zero vector $v$, 
there is $u \neq 0$ such that $v = M u$, and we have from the former display 
that 
\[
\| M^{-1} v \| \, \| v \| \geq (M^{-1} v)^* (\Im M) M^{-1} v 
  \geq \| (\Im M)^{-1} \|^{-1} \| M^{-1} v \|^2, 
\]
and hence the second result follows. We finally have 
\[
\Im(M^{-1}) = M^{-1} \frac{M^* - M}{2\imath} M^{-*} = 
 - M^{-1} (\Im M) M^{-*} < 0.
\]
\end{proof} 
The following result is well-known. We provide its proof for completeness.
\begin{lemma}
\label{nusym}
A probability measure $\check\nu$ is symmetric if and only if its Stieltjes
transform $g_{\check\nu}$, seen as an analytic function on $\CC\setminus\RR$,
satisfies $g_{\check\nu}(-\eta) = - g_{\check\nu}(\eta)$.
\end{lemma}
\begin{proof}
The necessity is obvious from the definition of the Stieltjes
transform and from the fact that $\check\nu(d\lambda) = \check\nu(-d\lambda)$.
To prove the sufficiency, we use the Perron inversion formula, which says that
for any function $\varphi \in C_c(\RR)$, 
$$\int_\RR \varphi(x) \, \check\nu(dx) = 
\lim_{\varepsilon\downarrow 0} \frac{1}{\pi} \int_\RR \varphi(x) \, 
\Im g_{\check\nu}(x + \imath \varepsilon) \, dx.$$ 
By a simple variable change at the right hand side, and by using the
equalities $g_{\check\nu}(-\eta) = - g_{\check\nu}(\eta)$ and
$g_{\check\nu}(\bar\eta) = \bar g_{\check\nu}(\eta)$, we obtain that
$\int \varphi(x) \, \check\nu(dx) = \int \varphi(-x) \, \check\nu(dx)$, and hence the result follows.
\end{proof}
The operator $\mcT$ introduced before Theorem~\ref{sys} has the following 
properties. 
\begin{lemma}
\label{mct} 
Suppose $M\in\CC^{2N\times 2N}$ and $S \in \cH_+^N$. Then, 
$$\| \mcT (
(I_2\otimes S) M) \| \leq \| M \| \tr S / n \leq 
(N/n) \| S \| \| M \|.$$  If $\Im M \geq \rho I_{2N}$
for some $\rho > 0$, then 
$$\Im \mcT ( (I_2\otimes S) M) \geq 
\rho n^{-1} (\tr S)   I_{2}.$$  
\end{lemma}
\begin{proof}
To obtain the first result, we write 
\begin{align*} 
\mcT ( (I_2\otimes S) M) &= 
\mcT((I_2\otimes S^{1/2}) M (I_2\otimes S^{1/2}))  \\ 
 &= \frac{1}{n} \sum_{\ell\in[N]} 
\begin{bmatrix} e_\ell^* S^{1/2} \\ & e_\ell^* S^{1/2} \end{bmatrix} 
 M  
\begin{bmatrix} S^{1/2} e_\ell \\ & S^{1/2} e_\ell \end{bmatrix}.
\end{align*} 
Hence, 
$$\| \mcT ( (I_2\otimes S) M) \| 
\leq n^{-1} \| M \| \sum_{\ell\in[N]} \| S^{1/2} e_\ell \|^2 
 = \| M \| \tr S / n.$$ To obtain the second inequality, just observe that 
$\Im \mcT ( (I_2\otimes S) M) = \mcT ( (I_2\otimes S) \Im M)$ and follow the same 
derivation as above. 
\end{proof}

\subsubsection*{Proof of Theorem~\ref{sys}}

Given a function $M(\eta) \in \mathfrak S^{2N}$, let 
$$A(M(\eta),e^{\imath\theta}) =
\mcT( (I_2\otimes \Sigma(e^{\imath\theta})) M(\eta) ) + U_L(e^{\imath\theta}).$$
We shall show that $A(M(\eta),e^{\imath\theta})^{-1}$ exists, is holomorphic as
a function of $\eta$, is continuous as a function of $\theta$, and satisfies
$\Im A(M(\eta),e^{\imath\theta})^{-1} \leq 0$.  

First let $\rho > 0$ be such that
$\Im M(\eta) \geq \rho I$. By Lemma~\ref{mct}, $\Im A(M(\eta),e^{\imath\theta})
\geq \rho n^{-1} \tr \Sigma(e^{\imath\theta}) I_2$.  Thus, if
$\Sigma(e^{\imath\theta}) \neq 0$, then $A(M(\eta),e^{\imath\theta})^{-1}$
exists, is holomorphic in $\eta$, and $\Im A(M(\eta),e^{\imath\theta})^{-1} <
0$ by Lemma~\ref{im>0}. 

Otherwise, $A(M(\eta),e^{\imath\theta}) =
U_L(e^{\imath\theta})$, which is trivially invertible, and $\Im
A(M(\eta),e^{\imath\theta})^{-1} = 0$. In summary, $A(M(\eta),e^{\imath
\theta})^{-1}$ is holomorphic in $\eta$ and satisfies 
$\Im A(M(\eta),e^{\imath\theta})^{-1} \leq 0$, for each $\theta\in[0,2\pi)$. 
Moreover, the continuity of $A(M(\eta),e^{\imath\cdot})^{-1}$ follows from the 
continuity of $A(M(\eta), e^{\imath\cdot})$. 

From these properties of $A$, it follows that 
$$B(M(\eta)) = (2\pi)^{-1}
\int_0^{2\pi} A(M(\eta), e^{\imath\theta})^{-1} \otimes \Sigma(e^{\imath\theta})
d\theta$$ exists, is holomorphic as a function of $\eta$, and satisfies $\Im
B(M(\eta)) \leq 0$. Since $\Im \begin{bmatrix} \eta & z \\ \bar z & \eta
\end{bmatrix} = \Im\eta I_2$, we get by Lemma~\ref{im>0} that the function
$\cF_{\Sigma,z}(M(\eta),\eta)$ is holomorphic in $\eta\in\CC_+$, and takes
values in $\cM_+^{2N}$.  Furthermore, since $M \in \mathfrak S^{2N}$, it holds 
that $\| M(\eta) \| \leq 1/(\Im\eta)$, and it is easy to show that
$\lim_{t\to\infty} \imath t \cF_{\Sigma,z}(M,\imath t) = - I_{2N}$.  In
summary, $\cF_{\Sigma,z}(M(\eta),\eta) \in \mathfrak S^{2N}$ as a function of 
$\eta$ when $M \in \mathfrak S^{2N}$.  

Let us now establish the uniqueness of the solution of Equation~\eqref{impl} in
the class $\mathfrak S^{2N}$. Assume that $P(z,\cdot)$ and $P'(z,\cdot)$ are 
two such solutions. Then,  
\begin{align*} 
&\cF_{\Sigma,z}(P, \eta) - \cF_{\Sigma,z}(P', \eta) \\
&= \cF_{\Sigma,z}(P, \eta)( B(P') - B(P)) \cF_{\Sigma,z}(P', \eta)  \\
&= \cF_{\Sigma,z}(P, \eta) \times \\ 
&\phantom{=} 
\left\{ \frac{1}{2\pi} \int_0^{2\pi} 
\left( A(P', e^{\imath\theta})^{-1} 
\mcT( (I_2\otimes \Sigma(e^{\imath\theta})) (P-P'))  
 A(P, e^{\imath\theta})^{-1} \right) \otimes \Sigma(e^{\imath\theta}) d\theta
 \right\} \times \\
&\phantom{=} \cF_{\Sigma,z}(P', \eta) . 
\end{align*} 

Define the domain  
\[
\cD =\Big\{\eta \in \CC_+: \Im\eta > 4 \|\Sigma\|_\infty \Bigl( \frac Nn \vee \sqrt{\frac Nn} \Bigr)\Big\}.
\]
Let 
$\eta \in \cD$.
Using the Inequality $\| P \|, \| P' \| \leq 1/\Im\eta$,  along with
Lemma~\ref{mct}, it can be
checked  that 
$\| A(P,e^{\imath\theta})^{-1} \|, \| A(P',e^{\imath\theta})^{-1} \| 
 \leq 2$ when $\eta\in\cD$. Using Lemma~\ref{mct} again, we have 
\begin{equation}
\label{contr} 
\| P - P' \| = 
\| \cF_{\Sigma,z}(P, \eta) - \cF_{\Sigma,z}(P', \eta) \| 
\leq 
\frac{4}{(\Im\eta)^2} \frac Nn \| \Sigma \|_\infty^2 \| P - P' \| 
 \leq \frac 12 \| P - P' \| 
\end{equation} 
which shows that $P(z,\eta) = P'(z,\eta)$ for $\eta\in\cD$, and hence, for
$\eta\in\CC_+$. 

To show the existence of the solution, set $P_0(z,\eta) = - \eta^{-1} I_{2N}
\in \mathfrak S^{2N}$, and consider the iterations 
$$
P_{k+1}(z,\eta) = \cF_{\Sigma,z}(P_k(z,\eta), \eta).
$$ Then, $P_k(z,\cdot) \in \mathfrak S^{2N}$ 
for each $k$, and furthermore, the sequence $P_k(z,\eta)$ converges on $\cD$ 
to a function $P_\infty(z,\eta)$ which satisfies $P_\infty(z,\eta) =
\cF_{\Sigma,z}(P_\infty(z,\eta), \eta)$ by Banach's fixed point theorem. 
Furthermore, given arbitrary vectors $a, b \in\CC^{2N}$, the sequence of 
holomorphic functions $(a^* P_k(z,\cdot) b)_k$ on $\cD$ is a \textit{normal family},
thus, their limit $a^* P_\infty(z,\cdot) b$ is holomorphic on $\cD$ by the 
normal family theorem. Since $a$ and $b$ are arbitrary, $P_\infty(z,\cdot)$
is a holomorphic matrix function on $\cD$ that satisfies the properties of
a matrix Stieltjes transform stated in Proposition~\ref{mv-st}. This
shows that $P_\infty(z,\cdot)$ is the unique solution of Equation~\eqref{impl} 
in $\mathfrak S^{2N}$.

It remains to establish the last result of Theorem~\ref{sys}. Extending the 
domain of $\cF_{\Sigma,z}(P(z,\eta),\eta)$ in the parameter $\eta$ to 
$\CC\setminus\RR$, we show that $P$ is the solution of the equation 
$P = \cF_{\Sigma,z}(P,\eta)$ if and only if the matrix function 
$P' = \begin{bmatrix} - P_{00} & P_{01} \\ 
                        P_{10} & - P_{11} \end{bmatrix}$ 
is the solution of the equation $P' = \cF_\Sigma(P', - \eta)$. This can be 
demonstrated by a direct calculation: writing 
\[
T(e^{\imath\theta}, z, \eta) = \\ 
 \mcT\left( (I_2\otimes \Sigma(e^{\imath\theta})) P(z,\eta) \right) 
= \begin{bmatrix} t_{00} &  t_{01} \\ t_{10} &  t_{11} \end{bmatrix}, 
\]
we have 
\[
\left( T(e^{\imath\theta}, z, \eta) + U_L(e^{\imath\theta}) \right)^{-1} 
= 
\frac{1}{\Delta} 
\begin{bmatrix} 
 t_{11} & - (t_{01}+e^{-\imath L\theta}) \\ 
 - (t_{10}+e^{\imath L\theta}) & t_{00} \end{bmatrix},
\]
with $\Delta = t_{00} t_{11} - (t_{01} + e^{-\imath L\theta}) 
(t_{10} + e^{\imath L\theta})$. Thus, 
\[
P = \begin{bmatrix} P_{00}  & P_{01} \\ P_{10} & P_{11} \end{bmatrix} = 
 \begin{bmatrix} 
\frac{1}{2\pi} \int  \frac{t_{11}}{\Delta} \Sigma - \eta 
  & - \frac{1}{2\pi} \int \frac{(t_{01}+e^{-\imath L\theta})}{\Delta} 
   \Sigma - z \\ 
 - \frac{1}{2\pi} \int \frac{(t_{10}+e^{\imath L\theta})}{\Delta} 
  \Sigma - \bar z 
 & \frac{1}{2\pi} \int \frac{t_{00}}{\Delta}\Sigma -\eta 
  \end{bmatrix}^{-1} . 
\]
Recalling the formula for the inverse 
of a partitioned matrix (see~\cite[\S 0.7.3]{HorJoh90})
\begin{align*} 
&\begin{bmatrix} M_{00} & M_{01} \\ M_{10} & M_{11} \end{bmatrix}^{-1} \\
&= 
\begin{bmatrix} 
(M_{00} - M_{01} M_{11}^{-1} M_{10})^{-1}   & 
-M_{00}^{-1} M_{01} (M_{11} - M_{10} M_{00}^{-1} M_{01})^{-1}  \\
- (M_{11} - M_{10} M_{00}^{-1} M_{01})^{-1} M_{10} M_{00}^{-1} &  
(M_{11} - M_{10} M_{00}^{-1} M_{01})^{-1}  
\end{bmatrix} 
\end{align*} 
we get the required result by a direct checking. 

Using this result in conjunction with Lemma~\ref{nusym}, we obtain that
for each deterministic vector $u \in \CC^N$, the scalar measures 
$(u^* \Lambda_{ii} u)(dt)$ are symmetric for $i=0,1$. This shows at once that 
the measures $\Lambda_{ii}$ are symmetric. The proof of Theorem~\ref{sys} is now complete. 

\subsection{Proof of Theorem~\ref{eqdet}} 
\label{prf-eqdet} 

Recall that $\bM$ is the bound provided by
Assumption~\ref{prop-S}--\ref{bnd-S}.   
From now on we shall write $\gmax = \sup_N N/n$. 

We first establish the tightness of the sequence $(\bs{\check\nu}_{z,N})_N$. 
As is well-known~\cite{GH2003}, this is equivalent to showing that
$-\imath t g_{\bs{\check\nu}_{z,N}}(\imath t) \to 1$ as $t\to\infty$ 
uniformly in $N$. 

Given $M\in \cM_+^{2N}$, let 
$$A(M,e^{\imath\theta}) \eqdef \mcT(
(I_2\otimes S(e^{\imath\theta})) M ) + U_L(e^{\imath\theta}).$$ 
By Lemma~\ref{mct}, 
$$\| \mcT( (I_2\otimes S(e^{\imath\theta}))
G(z,\imath t) ) \| \leq \gmax \bM / t.$$ Thus, for $t \geq 2 \gmax \bM$, 
 $\| A(G(z,\imath t), e^{\imath\theta})^{-1} \| \leq 2$. 
Since 
\[
G(z,\imath t) = \cF_{S,z}(G(z,\imath t), \imath t) = \left( 
\frac{1}{2\pi} \int A(G(z,\imath t), e^{\imath\theta})^{-1} 
 \otimes S(e^{\imath\theta}) d\theta 
- \begin{bmatrix} \imath t & z \\ \bar z & \imath t \end{bmatrix} 
  \otimes I_N \right)^{-1} ,  
\]
it is clear that 
\[
\frac{-\imath t}{2N} \tr G(z,\imath t) \xrightarrow[t\to\infty]{} 1 
\] 
uniformly in $N$, thus, the sequence $(\bs{\check\nu}_{z,N})_N$ is tight.

The remainder of the proof is devoted towards establishing the
convergence~\eqref{cvg-Q}. Recalling the expression $\widehat R_L = X J^L X^*$
provided at the beginning of Section~\ref{outline}, we have 
\[
Q(z,\eta) = \left( \bH(\widehat R_L - z) - \eta \right)^{-1} = 
 \begin{bmatrix} -\eta & X J^L X^* - z \\ X J^{-L} X^* - \bar z & - \eta 
 \end{bmatrix}^{-1} = 
 \begin{bmatrix} Q_{00} & Q_{01} \\ 
                 Q_{10} & Q_{11}\end{bmatrix}. 
\] 
We begin by bounding the variance of $(2N)^{-1} \tr D Q$ at the left hand side 
of~\eqref{cvg-Q}. 

\begin{proposition}
\label{np-Q} 
Under Assumption~\ref{prop-S}--\ref{bnd-S}), for 
each deterministic matrix $B \in \CC^{N\times N}$ and each $u,v\in\{0,1\}$,
$$\var(\tr B Q_{uv}(z, \eta)) \leq 8 \gmax \bM^2 \| B \|^2 / (\Im\eta)^4.$$ 
\end{proposition}

The proof of this proposition will be based on the well-known Poincaré-Nash
(PN) inequality \cite{cha-bos-04}, \cite{pas-05}, which is also a particular
case of the Brascamp-Lieb inequality.  Let $\bv = [ v_0, \ldots, v_{m-1} ]^\T$
be a complex Gaussian random vector with $\EE \bv = 0$, $\EE \bv \bv^\T = 0$,
and $\EE [ \bv \bv^* ] = \Sigma$. Let $\varphi = \varphi(v_0, \ldots, v_{m-1},
\bar v_0, \ldots, \bar v_{m-1} )$ be a $C^1$ complex function which is
polynomially bounded together with its derivatives. Then, writing 
\[
\nabla_{\bv}\varphi = [ \partial\varphi / \partial v_0, \ldots, 
 \partial\varphi / \partial v_{m-1} ]^\T \ \ \text{and}\ \ 
 \nabla_{\bar{\bv}}\varphi = [ \partial\varphi / 
 \partial \bar v_0, \ldots, \partial\varphi / \partial \bar v_{m-1} ]^\T,
\] 
the \textit{PN inequality} is 
\begin{equation}
\label{np} 
 \var\left( {\varphi}(\bv) \right) \leq 
 \EE \left[ \nabla_{\bv} \varphi( {\bv} )^T \ \Sigma \ 
 \overline{\nabla_{\bv} \varphi( {\bv} )} 
 \right] 
 +
 \EE \left[ \left( \nabla_{\bar{\bv}}  \varphi( \bv ) \right)^* 
 \ \Sigma \ \nabla_{\bar{\bv}} \varphi({\bv}) 
 \right].
\end{equation} 
If we rewrite $Q(z,\eta)$ as $Q(z,\eta) = Q^X$ to emphasize the
dependence of the resolvent on the matrix $X$, then, given
a matrix $\Delta \in \CC^{N\times n}$, the resolvent identity implies that
\begin{align*} 
& Q^{X+\Delta} - Q^X \\ 
 &= - Q^{X+\Delta} \left( (Q^{X+\Delta})^{-1} - 
  (Q^X)^{-1} \right) Q^X \\
 &= - Q^{X+\Delta} 
\begin{bmatrix} & (X+\Delta) J^L (X+\Delta)^* - X J^LX^* \\
(X+\Delta) J^{-L} (X+\Delta)^* - X J^{-L} X^* \end{bmatrix} Q^X . 
\end{align*} 
Using this equation, we can obtain the expression of $\partial a^* Q_{uv} b /
\partial \bar x_{ij}$, where $a,b\in\CC^N$, $u,v \in \{0,1\}$, $i\in [N]$, and
$j \in [n]$. Indeed, taking $\Delta = e_{N,i} e_{n,j}^*$ in the former
expression, we get after a simple derivation that
\begin{equation}
\label{deriv} 
\frac{\partial a^* Q_{uv} b}{\partial \bar x_{ij}}
 = - \begin{bmatrix} a^* Q_{u1} X J^{-L} \end{bmatrix}_{j} 
     \begin{bmatrix} Q_{0v} b \end{bmatrix}_i  
  - \begin{bmatrix} a^* Q_{u0} X J^{L} \end{bmatrix}_{j} 
     \begin{bmatrix} Q_{1v} b \end{bmatrix}_i  . 
\end{equation}

\begin{proof}[Proof of Proposition~\ref{np-Q}]   
We apply Inequality~\eqref{np} by respectively replacing $\bv$ and $\varphi$ with 
$\vect X$ and $\tr B Q_{uv}$ (seen as a function of $X$). 

Given $k,i\in [N]$, and $j\in [n]$, we have 
\[
\frac{\partial [ e_k^* B Q_{uv} e_k]}{\partial \bar{x}_{ij}} = 
 - \begin{bmatrix} B Q_{u1} X J^{-L} \end{bmatrix}_{kj} 
     \begin{bmatrix} Q_{0v} \end{bmatrix}_{ik}   
  - \begin{bmatrix} B Q_{u0} X J^{L} \end{bmatrix}_{kj} 
     \begin{bmatrix} Q_{1v} \end{bmatrix}_{ik} , 
\]
and thus, 
\[
\frac{\partial \tr B Q_{uv} }{\partial \bar{x}_{ij}} = 
 - \begin{bmatrix} Q_{0v} B Q_{u1} X J^{-L} \end{bmatrix}_{ij} 
  - \begin{bmatrix} Q_{1v} B Q_{u0} X J^{L} \end{bmatrix}_{ij} . 
\]
Let us focus on the second term at the right hand side of
Inequality~\eqref{np}. Observing that $\EE [x_{i_1 j_1} \bar x_{i_2 j_2} ] =
n^{-1} [R_{j_1 - j_2}]_{i_1, i_2}$, and recalling the expression of the
block-Toeplitz matrix $\cR$ given by Equation~\eqref{btoep}, we have 
\begin{align*}
& \sum_{i_1, i_2\in[N]} \sum_{j_1,j_2\in [n]} 
\EE \Bigl[\frac{\partial \overline{\tr BQ_{uv}}}{\partial \bar{x}_{i_1j_1}} 
\EE [x_{i_1 j_1} \bar x_{i_2 j_2} ] 
\frac{\partial \tr B Q_{uv}}{\partial \bar{x}_{i_2j_2}} \Bigr] \\
&\leq  
 2 \EE \vect( Q_{0v}BQ_{u1} X J^{-L} )^* \cR \vect( Q_{0v}BQ_{u1} X J^{-L} ) + 
 2 \EE \vect( Q_{1v}BQ_{u0} X J^{L} )^* \cR \vect( Q_{1v}BQ_{u0} X J^{L} ) \\
&\leq \frac{2 \bM}{n} \left( 
 \EE \| Q_{0v} B Q_{u1} X J^{-L} \|_\HS^2 + 
 \EE \| Q_{1v} B Q_{u0} X J^{L} \|_\HS^2 \right) \\
&\leq \frac{4 \bM \| B \|^2}{(\Im\eta)^4 n} \EE \| X \|_\HS^2 \\ 
&\leq \frac{4 \gmax \bM^2 \| B \|^2}{(\Im\eta)^4}.  
\end{align*} 
The first term at the right hand side of 
Inequality~\eqref{np} is treated similarly, leading to the bound given in
the statement, and the proof is complete.
\end{proof} 

In order to establish the convergence~\eqref{cvg-Q}, using Proposition~\ref{np-Q} and the Borel-Cantelli lemma, it will
be enough to show that 
\begin{equation}
\label{cvg-EQ} 
\forall\eta\in\CC_+, \quad 
\frac{1}{2N} \tr D^{(N)} \left(\EE Q^{(N)}(z,\eta) - G^{(N)}(z,\eta) \right) 
 \xrightarrow[N\to\infty]{} 0.  
\end{equation} 
Following the general canvas of \cite{pel-pel-16}, we approximate our process $\bs x^{(N)}$ by a Moving Average process with a
finite memory. We shall
construct from this MA process a resolvent that will be more easily manageable
than $Q^{(N)}$. 

With reference to Theorem~\ref{sys}, first consider a discrete analogue of the integral within the
expression of $\cF_{\Sigma,z}$. 
A straightforward adaptation of its proof yields the following proposition 
and we omit the details. 
\begin{proposition}\label{prop:Fsigma}
Let $\Sigma : \TT \to \cH_+^N$ be a continuous function, and let $z\in\CC$.
Then, the conclusions of Theorem~\ref{sys} remain true if the function
$\cF_{\Sigma,z}$ there is replaced with 
\begin{multline*} 
\ceF_{\Sigma,z}(M(\eta), \eta) = \\ 
\Bigl( \frac 1n \sum_{\ell\in[n]} 
\left( \mcT( (I_2\otimes \Sigma(e^{2\imath\pi\ell/n})) M(\eta) ) + 
 \begin{bmatrix} & e^{-2\imath\pi\ell L/n} \\
 e^{-2\imath\pi\ell L/n} \end{bmatrix} \right)^{-1} 
    \otimes \Sigma(e^{2\imath\pi\ell/n}) 
- \begin{bmatrix} \eta & z \\ \bar z & \eta \end{bmatrix} 
  \otimes I_N \Bigr)^{-1} . 
\end{multline*} 
\end{proposition} 

Now, given an integer constant $K > 0$, let us define the function 
$\hS^{(N,K)}$ on $\TT$ as 
\[
\hS^{(N,K)}(e^{\imath\theta}) = \frac{1}{2\pi} \int_0^{2\pi} 
 F_K(e^{\imath(\theta-\psi)}) S^{(N)}(e^{\imath\psi}) \ d\psi , 
\]
where $F_K$ is the Fej\'er kernel (see Equation~\eqref{def-fej}).
This function has the following properties: 
\begin{enumerate}
\item\label{hS>0} By the non-negativity of the Fej\'er kernel, 
 $\hS^{(N,K)}$ is a spectral density.  
\item\label{hS-laurent} By replacing $F_K(e^{\imath(\theta-\psi)})$ 
with the first expression of this kernel provided by~\eqref{def-fej}, and
by developing the integrand above, we obtain that  
 $\hS^{(N,K)}$ is a Laurent trigonometric polynomial of the form 
 $\hS^{(N,K)}(e^{\imath\theta}) = \sum_{\ell=-K}^{K} 
 e^{\imath\ell\theta} \widetilde R_\ell^{(N,K)}$. 
\item\label{hS->S}
With Assumptions~\ref{prop-S}--\ref{S-equi} and~\ref{prop-S}--\ref{bnd-S}, we have  
\begin{equation}
\label{cvg-K} 
\sup_{N} \left\| \hS^{(N,K)} - S^{(N)} \right\|^\TT_\infty
 \xrightarrow[K\to\infty]{} 0 .
\end{equation} 
Relation (\ref{cvg-K}) can be established by splitting the integral that defines 
$\hS^{(N,K)}(e^{\imath\theta})$ into two pieces as $\int_0^{2\pi} =
\int_{\psi:|\theta-\psi|\leq \delta} + \int_{\psi:|\theta-\psi| > \delta}$ for
a properly chosen $\delta > 0$, and by using the properties of the Fej\'er
kernel provided after Equation~\eqref{def-fej}.  
\end{enumerate}
Consider the implicit 
equation 
\[
\hG^{(N,K)}(z,\eta) = \ceF_{\hS^{(N,K)},z}(\hG^{(N,K)}(z,\eta),\eta) \ \ \text{(in} \ \ \mathfrak S^{2N}).  
\]
From Proposition \ref{prop:Fsigma}, the solution $\hG^{(N,K)}(z,\cdot)$ exists and is unique.
The following three propositions will be proved in Section~\ref{prf-svd}. 


\begin{proposition}
\label{hG-G} 
For each $\eta$ such that $\Im\eta \geq C_{\ref{hG-G}}$ with 
$C_{\ref{hG-G}} = C_{\ref{hG-G}}(\gmax, \bM) > 0$, 
\[
\limsup_N \| \hG^{(N,K)}(z,\eta) - G^{(N)}(z,\eta) \| 
 \xrightarrow[K\to\infty]{} 0. 
\] 
\end{proposition} 

Let $(\hat x^{(N,K)}_k)_{k\in\ZZ}$ be a $\CC^N$--valued 
stationary centered Gaussian process with the spectral density 
$n^{-1} \hS^{(N,K)}$. Define  
$$\hX^{(N,K)} = \begin{bmatrix} \hat x^{(N,K)}_0 & 
 \cdots & \hat x^{(N,K)}_{n-1} \end{bmatrix} \ \ \text{and}\ \ 
\hQ^{(N,K)}(z,\eta) = (\bH(\hX^{(N,K)} J^L (\hX^{(N,K)})^* - z) 
 - \eta I)^{-1}.$$ We then have the following proposition. 
\begin{proposition}
\label{hQ-hG} 
Fix $K > 0$.
Then, for an arbitrary deterministic matrix $D^{(N)} \in
\CC^{2N\times 2N}$ with $\| D^{(N)} \| = 1$,  we have   
\[
\frac 1N \left| \tr D^{(N)} 
  \left( \EE \hQ^{(N,K)}(z,\eta) - \hG^{(N,K)}(z,\eta) \right) \right| 
 \leq \frac{C K}{\sqrt{N}}, 
\]
where $C > 0$ depends only on $\eta, \bM$, and $\gmax$.  
\end{proposition} 

We note here that the bound provided in the statement of this proposition is
not optimal but is good enough for our purpose. 

\begin{proposition}
\label{Q-hQ} 
With $D^{(N)}$ as in the previous proposition, 
\[
\limsup_N \frac 1N 
 \left| \tr D^{(N)} \left( \EE Q^{(N)}(z, \eta) - \EE \hQ^{(N,K)}(z, \eta) 
  \right)\right| \xrightarrow[K\to\infty]{} 0 .
\] 
\end{proposition} 

\subsubsection*{Theorem~\ref{eqdet}: end of the proof.} 
We write 
\begin{align*} 
&\frac 1N \tr D^{(N)} (\EE Q^{(N)} - G^{(N)}) \\
 &= \frac 1N \tr D^{(N)} (\EE Q^{(N)} - \EE \hQ^{N,K} ) + 
\frac 1N \tr D^{(N)} (\EE \hQ^{N,K} - \hG^{(N,K)}) \\
&\phantom{=} + 
\frac 1N \tr D^{(N)} (\hG^{(N,K)} - G^{(N)})  \\
 &\eqdef \chi_1(N,K) + \chi_2(N,K) + \chi_3(N,K) .
\end{align*} 
Fix an arbitrarily small $\varepsilon > 0$. Let $K_0, N_0 > 0$ be such that,
by Propositions~\ref{Q-hQ} and~\ref{hG-G}, 
\[
|\chi_1(N,K_0)|, |\chi_3(N,K_0)| \leq \varepsilon \quad 
 \text{for all} \ N \geq N_0 \ \text{and} \ \Im\eta > C_{\ref{hG-G}}. 
\]
By Proposition~\ref{hQ-hG}, $\chi_2(N,K_0) \to_{N\to\infty} 0$. Thus, 
$(2N)^{-1} \tr D^{(N)} (\EE Q^{(N)} - G^{(N)}) \to_N 0$ first for 
$\Im\eta > C_{\ref{hG-G}}$, and hence for each $\eta\in\CC_+$ by 
analyticity. Thus~\eqref{cvg-EQ} is established. This completes the proof of
Theorem~\ref{eqdet}.

\subsection{Remaining proofs for Section~\ref{prf-eqdet}} 
\label{prf-svd} 

\subsubsection*{Proof of Proposition~\ref{hG-G}} 

Let $\bar S^{(N,K)}(e^{\imath\theta}) = \hS^{(N,K)}(e^{2\imath\pi k / n})$ and
$\bar U_L(e^{\imath\theta}) = U_L(e^{2\imath\pi k / n})$ for $\theta\in [2\pi
k/n, 2\pi(k+1)/n)$ be the respective stepwise continuous versions of the
functions $\hS^{(N,K)}$ and $U_L$ with step size $2\pi /n$. 
Within this proof, we re-denote the function $\cF_{\Sigma,z}$ defined in the
statement of Theorem~\ref{sys} as $\cF_{\Sigma,z,U_L}$ to stress the 
dependence on $U_L$. With this notation, it is obvious that 
$\ceF_{\hS^{(N,K)},z} = \cF_{\bar S^{(N,K)},z,\bar U_L}$. 

In the rest of the proof, we often drop the superscripts $^{(N)}$ and $^{(N,K)}$ and the
subscript $_L$ for brevity.  Given $M\in \mathfrak S^{2N}$, put
$A_{\bS,\bU}(M,e^{\imath\theta}) \eqdef \mcT( (I_2\otimes
\bS(e^{\imath\theta})) M ) + \bU(e^{\imath\theta})$, where $(\bS,\bU) = (S,U)$
or $(\bar S,\bar U)$. Write $B_{\bS,\bU}(M) = (2\pi)^{-1} \int_0^{2\pi}
A_{\bS,\bU}(M,e^{\imath\theta})^{-1} \otimes \bS(e^{\imath\theta}) d\theta$. 

We also assume that $K$ is large enough so that 
\[
\sup_N \| \bar S \|^\TT_\infty \leq 2 \bM .
\]
By dropping the unnecessary parameters from the notations, we write 
\begin{align*} 
&G - \hG \\
 &= \cF_{S,U}(G) - \cF_{\bar S,\bar U}(\hG) = 
 \cF_{S,U}(G)\left( B_{\bar S,\bar U}(\hG) - B_{S,U}(G)\right) 
               \cF_{\bar S,\bar U}(\hG) \\
&= \cF_{S,U}(G)\left( \frac{1}{2\pi}\int_0^{2\pi} 
 \left(A_{\bar S,\bar U}(\hG, e^{\imath\theta})^{-1} \otimes 
 \bar S(e^{\imath\theta}) 
 - A_{S,U}(G, e^{\imath\theta})^{-1} \otimes S(e^{\imath\theta})\right) 
  d \theta  \right) \cF_{\bar S,\bar U}(\hG) \\
&= \cF_{S,U}(G)\left( \frac{1}{2\pi}\int 
 \left(A_{\bar S,\bar U}(\hG)^{-1} \otimes \bar S - A_{S,U}(\hG)^{-1} \otimes S
 \right.\right. \\
&\phantom{=} 
 \quad\quad\quad\quad\quad\quad\quad\quad
\left.\left. 
 + (A_{S,U}(\hG)^{-1} - A_{S,U}(G)^{-1} ) \otimes S\right) 
  d \theta  \right) \cF_{\bar S,\bar U}(\hG) \\
&= \cF_{S,U}(G)\left( \frac{1}{2\pi}\int 
 \left(A_{\bar S,\bar U}(\hG)^{-1} \otimes (\bar S-S) 
 + ( A_{\bar S,\bar U}(\hG)^{-1} - A_{S,U}(\hG)^{-1}) \otimes S \right.\right. \\
&\phantom{=} 
 \quad\quad\quad\quad\quad\quad\quad\quad
\left.\left. + (A_{S,U}(\hG)^{-1} - A_{S,U}(G)^{-1} ) \otimes S\right) 
  d \theta  \right) \cF_{{\bar S,\bar U}}(\hG) \\
&\eqdef \cF_{{S,U}}(G)\left( 
\frac{1}{2\pi}\int ( \chi_1 + \chi_2 + \chi_3 ) d\theta \right) 
  \cF_{{\bar S,\bar U}}(\hG) .  
\end{align*} 
By Lemma~\ref{mct}, we have that $\| \mcT( (I_2\otimes \bS(e^{\imath\theta}))
\bG ) \| \leq 2 \gmax \bM / \Im\eta$ for any of the possibilities for $\bS$ and
for $\bG = G, \hG$.  Thus, for 
\[
2 \gmax \bM / \Im\eta \leq 1/2, 
\]
we have  $\| A_{\bS}(\bG, e^{\imath\theta})^{-1} \| \leq 2$. 
Therefore,
\[
\| \chi_1 \| \leq 2 \| \bar S - S \|^\TT_\infty . 
\]
Moreover, 
\begin{align*} 
\chi_2 &= ( A_{\bar S,\bar U}(\hG)^{-1} - A_{S,U}(\hG)^{-1}) \otimes S \\
 &= \left( A_{\bar S,\bar U}(\hG)^{-1} \left( 
 \mcT( (I_2\otimes (\bar S - S)) \hG )  + U - \bar U \right) 
 A_{S,U}(\hG)^{-1} \right) \otimes S 
\end{align*} 
satisfies 
\[
\| \chi_2 \| \leq 4 \bM \gmax (\Im\eta)^{-1} \| \bar S - S \|^\TT_\infty 
 + 4 \bM \| U - \bar U \| 
\]
for the same values of $\eta$. 
By mimicking the calculation that lead to Inequality~\eqref{contr}, we obtain
\[
\| \chi_3 \| \leq 
\frac{4}{(\Im\eta)^2} \gmax \bM^2 \| G - \hG \| . 
\]
Using the inequality $\| \cF_{\bS}(\bG) \| \leq 1/(\Im\eta)$, we thus arrive at 
\begin{multline*} 
\left( 1 - \frac{4\gmax \bM^2}{(\Im\eta)^4}    \right) 
\| G^{(N)} - \hG^{(N,K)} \| \\
 \leq \frac{2}{(\Im\eta)^2} 
   \left( 1 + \frac{4\gmax\bM}{\Im\eta}\right) 
   \| \bar S^{(N,K)}-S^{(N)} \|^\TT_\infty
 + \frac{4\bs M}{(\Im\eta)^2} \| \bar U_L - U_L \|^\TT_\infty . 
\end{multline*} 
Thus, if $\Im\eta > (8\gmax \bM^2)^{1/4} \vee (4\gmax \bs M)$, then 
\[
\| G^{(N)} - \hG^{(N,K)} \| \leq C \| \bar S^{(N,K)}-S^{(N)} \|^\TT_\infty
 + C' \| \bar U_L - U_L \|^\TT_\infty 
\]
for some constants $C, C' > 0$. Now the proof can be completed by using the
convergence~\eqref{cvg-K}.

\subsubsection*{Proof of Proposition~\ref{hQ-hG}} 

From now on, $C$ is a positive constant that depends on $\eta$, $\bM$, and $\gmax$ at most, and can change from line to
line. Recalling Properties~\ref{hS>0} and \ref{hS-laurent} of the density
$\hS^{(N,K)}$ that were stated in Section~\ref{prf-eqdet} above, our first step
is to apply the well-known operator version of the Fej\'er-Riesz theorem 
(see \cite[Sec.~6.6]{ros-rov-livre85}) to $\hS^{(N,K)}$. This implies that for each 
$(N,K)$, there exists an $N\times N$ matrix trigonometric polynomial 
\[
P^{(N,K)}(e^{\imath\theta}) = 
  \sum_{\ell=0}^{K} e^{\imath\ell\theta} B_\ell^{(N,K)} 
\]
such that $\hS^{(N,K)}(e^{\imath\theta}) = P^{(N,K)}(e^{\imath\theta})
P^{(N,K)}(e^{\imath\theta})^*$. 

Letting $\bs\xi^{(N)} = (\bs\xi^{(N)}_{k})_{k\in\ZZ}$ be an i.i.d.~process with
$\bs \xi^{(N)}_{k} \sim \cN_\CC(0, I_N)$, the process 
$(\hat x^{(N,K)}_k)_{k\in\ZZ}$ that we used to construct the resolvent
$\hQ^{(N,K)}$ can be defined as 
\[
\hat x^{(N,K)}_k = \frac{1}{\sqrt{n}} 
  \sum_{\ell=0}^{K} B_\ell^{(N,K)} \bs\xi^{(N)}_{k-\ell}. 
\] 
Define the finite sequence of random vectors 
$(\tilde x^{(N,K)}_k)_{k\in[n]}$ as 
\[
\tilde x^{(N,K)}_k = \frac{1}{\sqrt{n}} 
  \sum_{\ell=0}^{K} B_\ell^{(N,K)} \bs\xi^{(N)}_{(k-\ell)\!\! \mod n} 
\]
(thus, $\tilde x^{(N,K)}_k$ is the analogue of $\hat x^{(N,K)}_k$ obtained
through a circular convolution). Define 
\[
\tX^{(N,K)} = \begin{bmatrix} \tilde x^{(N,K)}_0 & 
 \cdots & \tilde x^{(N,K)}_{n-1} \end{bmatrix} \in \CC^{N\times n} , 
\]
and
\[
\tQ^{(N,K)}(z,\eta) = 
 (\bH(\tX^{(N,K)} J^L (\tX^{(N,K)})^* - z) - \eta I)^{-1}.
\]
Observing that $\rank(\hX-\tX) \leq K$, we get  
$$\rank(\bH(\hX J^L \hX^* - z)- \bH(\tX J^L \tX^* - z)) \leq 4K.$$
Thus, 
for each matrix $D \in \CC^{2N\times 2N}$, the inequality 
\begin{equation}
\label{fin-rk} 
 \left| \tr D (\hQ(z,\eta) - \tQ(z,\eta) ) \right| 
 \leq \frac{4 K\| D \|}{\Im\eta} 
\end{equation} 
holds (see \cite[Lemma 2.6]{SilBai95}). We can thus work with $\tQ$ in place of $\hQ$ for establishing Proposition~\ref{hQ-hG}. 

For $k \in [n]$, let 
\[
w_k^{(N,K)} = \frac{1}{\sqrt{n}} 
    \sum_{\ell=0}^{n-1} e^{2\imath\pi k\ell/n} \tilde x_\ell^{(N,K)}   
\]
be the discrete Fourier transform of the finite sequence 
$(\tilde x_0, \ldots, \tilde x_{n-1})$, and define the matrix 
\[
W^{(N,K)} = \begin{bmatrix} w_0^{(N,K)} & \cdots & w_{n-1}^{(N,K)} 
  \end{bmatrix} = \tX^{(N,K)} {\sF} \in \CC^{N\times n} 
\]
where $\sF$ is the Fourier matrix defined in~\eqref{four}.  Since the $\tilde
x_\ell$ are built through a circular convolution, the vectors $w_k$ are
independent, and $w_k \sim \cN_\CC(0, n^{-1} \hS(e^{2\imath\pi k / n}))$.
Recalling that $J = \sF \Omega \sF^*$ with $\Omega =
\diag(\omega^\ell)_{\ell=0}^{n-1}$ and $\omega = \exp(-2\imath\pi / n)$, we 
can write 
\[
\tQ(z,\eta) = (\bH(\tX J^L \tX^* - z) - \eta I)^{-1} = 
 (\bH(W \Omega^L W^* - z) - \eta I)^{-1} .
\]
The remainder of the proof will be devoted towards showing that 
\begin{equation}
\label{tQ-hG} 
\left\| \EE \tQ^{(N,K)}(z,\eta) - \hG^{(N,K)}(z,\eta) \right\| \leq 
 C N^{-1/2} , 
\end{equation}
taking advantage of the independence of the columns of $W$.  This bound, used
in conjunction with the bound~\eqref{fin-rk}, immediately leads to the result
of Proposition~\ref{hQ-hG}. 

The proof of Inequality~\eqref{tQ-hG} relies on the NP inequality that we used
above, as well as on the well-known Integration by Parts (IP) formula for
Gaussian functionals \cite{gli-jaf-(livre87), kho-pastur93}.  Recalling the
definition of the vector $\bv$ and the functional $\varphi$ after 
Proposition~\ref{np-Q} above, the IP formula reads as 
\[
\EE v_k \varphi( \bv  ) = 
\sum_{\ell = 0}^{n-1} \left[ \Sigma \right]_{k \ell} 
\EE \left[ \frac{\partial \varphi(\bv)}{\partial \bar v_{\ell}} \right].
\]
Write 
$\tQ(z,\eta) = \begin{bmatrix} \tQ_{00}  & \tQ_{01} \\
  \tQ_{10}  & \tQ_{11} \end{bmatrix}$, 
and $W = \begin{bmatrix} w_{ij} \end{bmatrix}_{i\in[N],j\in[n]}$. By
reproducing verbatim the derivation that we made to obtain the
Identity~\eqref{deriv}, we have 
\[
\frac{\partial a^* \tQ_{uv} b}{\partial \bar w_{ij}}
 = - \begin{bmatrix} a^* \tQ_{u1} W \Omega^{-L} \end{bmatrix}_{j} 
     \begin{bmatrix} \tQ_{0v} b \end{bmatrix}_i  
  - \begin{bmatrix} a^* \tQ_{u0} W \Omega^{L} \end{bmatrix}_{j} 
     \begin{bmatrix} \tQ_{1v} b \end{bmatrix}_i  . 
\]
In the subsequent derivations, we write $\hS_k = \hS(e^{2\imath\pi k / n})$.
Let $a,b$ be two constant vectors in $\CC^N$.  Write $b = [ b_0, \ldots,
b_{N-1} ]^\T$, and let 
$\alpha_{uv}(\ell) = \EE \begin{bmatrix} [ a^* \tQ_{uv} W ]_\ell 
  [ W^* b ]_\ell \end{bmatrix}$. By the IP formula, we have  
\begin{align*}
& \alpha_{uv}(\ell) \\
&= 
\sum_{i\in[N]} \EE \begin{bmatrix} 
 [ a^* \tQ_{uv} ]_i w_{i\ell} [ W^* b ]_\ell \end{bmatrix}  
= \frac 1n \sum_{i,m\in[N]} [ \hS_\ell ]_{im} 
 \EE \frac{\partial ( [ a^* \tQ_{uv} ]_i [ W^* b ]_\ell)}
  {\partial \bar w_{m\ell}}  \\
&= \frac 1n \sum_{i,m} [ \hS_\ell ]_{im} \EE\Bigl[ 
 - [ a^* \tQ_{u1} W \Omega^{-L} ]_\ell [\tQ_{0v}]_{mi} [ W^* b ]_\ell 
 - [ a^* \tQ_{u0} W \Omega^{L} ]_\ell [\tQ_{1v}]_{mi} [ W^* b ]_\ell \\ 
& \quad\quad \quad\quad \quad\quad \quad \ 
 + [ a^* \tQ_{uv} ]_i b_m \Bigr] \\
&= 
 - \EE \Bigl[ [ a^* \tQ_{u1} W ]_\ell [ W^* b ]_\ell 
    \omega^{-\ell L} \frac{\tr \tQ_{0v}\hS_\ell}{n} \Bigr] 
 - \EE \Bigl[ [ a^* \tQ_{u0} W ]_\ell [ W^* b ]_\ell 
    \omega^{\ell L} \frac{\tr \tQ_{1v}\hS_\ell}{n} \Bigr] \\
&\phantom{=} + \frac{\EE [ a^* \tQ_{uv} \hS_\ell b ]}{n} . 
\end{align*} 
Write $\tqs_{uv}(\ell) = \EE \tr \tQ_{uv} \hS_\ell / n$. We shall isolate the
terms $\tqs_{0v}(\ell)$ and $\tqs_{1v}(\ell)$ in the last display by resorting 
to the following lemma. 
\begin{lemma}
The two inequalities,  
$\var \tqs_{uv}(\ell) \leq C / n^2$, and 
$\var [ a^* \tQ_{uv} W ]_\ell [ W^* b ]_\ell \leq C \| a \|^2 \| b \|^2 / n$ hold. 
\end{lemma}
\begin{proof}
The first bound is obtained by repeating almost word for word the 
proof of Proposition~\ref{np-Q}. 
To obtain the second bound, we also use the NP inequality again. We start by 
writing 
\begin{align*}
\frac{\partial [ a^* \tQ_{uv} W ]_\ell [ W^* b ]_\ell}{\partial \bar w_{ij}} 
&= \sum_{k} 
\frac{\partial[ a^* \tQ_{uv}]_k w_{k\ell} [ W^* b ]_\ell}{\partial\bar w_{ij}} 
  \\
&= [ a^* \tQ_{uv} W ]_\ell b_i \1_{j=\ell}  
   - [ a^* \tQ_{u1} W\Omega^{-L} ]_j [ Q_{0v} W ]_{i\ell} [ W^* b ]_\ell  \\
&\phantom{=} 
  - [ a^* \tQ_{u0} W\Omega^{L} ]_j [ Q_{1v} W ]_{i\ell}  [ W^* b ]_\ell .
\end{align*}  
We focus on the second term at the right hand side of
Inequality~\eqref{np}, treating separately the three terms at the right hand
side of the last display. Starting with the first term, we get 
\begin{multline*} 
\frac 1n \sum_{i_1, i_2\in[N]} 
\EE \overline{[ a^* \tQ_{uv} W ]_\ell b_{i_1}} [\hS_\ell]_{i_1,i_2} 
{[ a^* \tQ_{uv} W ]_\ell b_{i_2}}  \\ 
= \frac 1n b^* \hS_\ell b \, \EE | [ a^* \tQ_{uv} W ]_\ell |^2 
\leq \| a \|^2 \| b \|^2 C / n. 
\end{multline*} 
Turning to the second of these terms, we have 
\begin{align*} 
&\frac 1n \sum_{i_1, i_2\in[N]} \sum_{j\in [n]} 
\EE | [ W^* b ]_\ell |^2 \, | [ a^* \tQ_{u1} W\Omega^{-L} ]_{j}|^2 
  \overline{[ Q_{0v} W ]_{i_1\ell}} [\hS_j]_{i_1,i_2} [ Q_{0v} W ]_{i_2\ell} \\
&\leq \frac{\bM}{n}  
 \EE | [ W^* b ]_\ell |^2 \, \| a^* \tQ_{u1} W \|^2
  \| [ Q_{0v} W ]_{\cdot,\ell} \|^2 \\ 
&\leq \| a \|^2 \| b \|^2 C / n^2, 
\end{align*} 
where the last inequality can be obtained by applying, \emph{e.g.}, 
Lemma~\ref{spec_norm} along with standard inequalities. 
The third term can be handled similarly. 
\end{proof} 

Thanks to these bounds and to Cauchy-Schwarz inequality, we obtain
the identity 
\[
\alpha_{uv}(\ell) = -\alpha_{u0}(\ell) \omega^{\ell L} \tqs_{1v}(\ell) 
  - \alpha_{u1}(\ell) \omega^{-\ell L} \tqs_{0v}(\ell) 
  + n^{-1} \EE [ a^* \tQ_{uv} \hS_\ell b ] + \varepsilon, 
\]
with $\| \varepsilon \| \leq C \| a \| \| b \| n^{-3/2}$. 
This leads to the system of equations 
\[
\begin{bmatrix} 
1+ \omega^{\ell L} \tqs_{10}(\ell) &  \omega^{-\ell L} \tqs_{00}(\ell) \\
 \omega^{\ell L} \tqs_{11}(\ell) & 1+\omega^{-\ell L} \tqs_{01}(\ell) 
\end{bmatrix} 
\begin{bmatrix} \alpha_{u0}(\ell) \\ \alpha_{u1}(\ell) \end{bmatrix} 
= 
\frac 1n \begin{bmatrix}
  \EE [ a^* \tQ_{u0} \hS_\ell b ] \\ 
  \EE [ a^* \tQ_{u1} \hS_\ell b ] 
\end{bmatrix} 
+ \bs\varepsilon, 
\]
with $\| \bs\varepsilon \| \leq C \| a \| \| b \| n^{-3/2}$. 
The matrix at the left hand side of this expression, that we denote as 
$T_\ell$, is written as 
\[
T_\ell = I_2 + \mcT\left( (I_2\otimes \hS_\ell) 
 \left( \tQ^\T 
 \begin{bmatrix} & \omega^{-\ell L} \\ \omega^{\ell L} \end{bmatrix} 
 \right)\right) .
\]
Assume that $K$ is large enough so that $\| \hS_\ell \| \leq 2\bM$ for all 
$\ell\in[n]$, and take $\Im\eta \geq 4 \gmax \bM$. Then by Lemma~\ref{mct} we get that 
$\| T_\ell - I_2 \| \leq 1/2$. Thus, the determinant 
\[
d(\ell) = (1 + \omega^{\ell L} \tqs_{10}(\ell)) 
     (1 + \omega^{-\ell L} \tqs_{01}(\ell)) - \tqs_{00}(\ell) \tqs_{11}(\ell)
\]
of $T_\ell$ is such that $| d(\ell) |$ is bounded away from zero uniformly in
$N$ and $\ell$.  Solving our system, and reusing henceforth the notations 
$\varepsilon$ and $\bs\varepsilon$ at will, we get that 
\begin{equation} 
\label{alpha} 
\begin{bmatrix} \alpha_{u0}(\ell) \\ \alpha_{u1}(\ell) \end{bmatrix} 
= 
\frac 1n \frac{1}{d(\ell)} 
\begin{bmatrix}
(1 + \omega^{-\ell L} \tqs_{01}(\ell)) &  
 - \omega^{-\ell L} \tqs_{00}(\ell)  \\
- \omega^{\ell L} \tqs_{11}(\ell) &  (1 + \omega^{\ell L} \tqs_{10}(\ell)) 
\end{bmatrix}
\begin{bmatrix}
\EE [ a^* \tQ_{u0} \hS_\ell b ] \\
\EE [ a^* \tQ_{u1} \hS_\ell b ] 
\end{bmatrix}
+ \bs\varepsilon, 
\end{equation} 
with $\| \bs\varepsilon \| \leq C \| a \| \| b \| n^{-3/2}$. 

Now, keeping in mind the identity 
$\tQ \Bigl( \bH(W\Omega^L W^*) - \begin{bmatrix} \eta & z \\ \bar z & \eta
 \end{bmatrix} \otimes I_N \Bigr) = I$, our purpose is to find an 
approximant of the matrix 
\[
\tQ \bH(W\Omega^L W^*) = 
 \begin{bmatrix} 
 \tQ_{01} W \Omega^{-L} W^* & \tQ_{00} W \Omega^{L} W^* \\
 \tQ_{11} W \Omega^{-L} W^* & \tQ_{10} W \Omega^{L} W^*
\end{bmatrix} . 
\]
To that end, we write 
\[
\EE [a^* \tQ_{u0} W \Omega^L W^* b ] = 
 \sum_{\ell\in[n]} \omega^{\ell L} \alpha_{u0}(\ell) , 
\quad\text{and}\quad 
\EE [a^* \tQ_{u1} W \Omega^{-L} W^* b ] = 
 \sum_{\ell\in[n]} \omega^{-\ell L} \alpha_{u1}(\ell) ,   
\]
and we use Equation~\eqref{alpha} to obtain 
\begin{align} 
 & \begin{bmatrix} 
\EE a^* \tQ_{u1} W \Omega^{-L} W^* b & \EE a^* \tQ_{u0} W \Omega^{L} W^* b 
 \end{bmatrix} \nonumber \\ 
 &= 
 a^* \EE \begin{bmatrix} \tQ_{u0} & \tQ_{u1} \end{bmatrix} 
 \times 
 \frac 1n \sum_{\ell\in[n]} \frac{1}{d(\ell)} 
 \begin{bmatrix} - \tqs_{11}(\ell) & \omega^{\ell L} + \tqs_{01}(\ell) \\
   \omega^{-\ell L} + \tqs_{10}(\ell) & - \tqs_{00}(\ell)
    \end{bmatrix} 
 \otimes (\hS_\ell b) 
   + \bs\varepsilon 
\label{int-C} 
\end{align} 
with $\| \bs\varepsilon \| \leq C \| a \| \| b \| n^{-1/2}$ (note that we lost
a factor of $n^{-1}$ because of the summation $\sum_{\ell\in[n]}$). 
Let $U_{L,\ell} \eqdef U_L(e^{2\imath\pi\ell/n}) = 
 \begin{bmatrix} & \omega^{\ell L} \\ \omega^{-\ell L} \end{bmatrix}$, and 
define the matrix 
\begin{align*} 
C(z,\eta) &\eqdef 
\frac 1n \sum_{\ell\in[n]} \frac{1}{d(\ell)} \begin{bmatrix} 
- \tqs_{11}(\ell) & \omega^{\ell L} + \tqs_{01}(\ell) \\ 
\omega^{- \ell L} + \tqs_{10}(\ell) & - \tqs_{00}(\ell) 
\end{bmatrix} \otimes \hS_\ell  \\ 
 &= \frac 1n \sum_{\ell\in[n]} 
\left( \begin{bmatrix} \tqs_{00}(\ell) & \tqs_{01}(\ell) \\
 \tqs_{10}(\ell) & \tqs_{11}(\ell) \end{bmatrix} + U_{L,\ell} 
 \right)^{-1} \otimes \hS_\ell \\ 
 &= \frac 1n \sum_{\ell\in[n]}  
\left( \mcT( (I_2\otimes \hS_\ell) \EE\tQ  ) + U_{L,\ell} \right)^{-1} 
    \otimes \hS_\ell  . 
\end{align*} 
Then, Equation~\eqref{int-C} can be rewritten as 
\begin{align*} 
\EE a^* \tQ_{u1} W \Omega^{-L} W^* b 
  &= a^* \EE \begin{bmatrix} \tQ_{u0} & \tQ_{u1} \end{bmatrix} 
    C \begin{bmatrix} b \\ 0 \end{bmatrix} + \varepsilon, \\ 
\EE a^* \tQ_{u0} W \Omega^{L} W^* b &= 
   a^* \EE \begin{bmatrix} \tQ_{u0} & \tQ_{u1} \end{bmatrix} 
    C \begin{bmatrix} 0 \\ b \end{bmatrix} + \varepsilon', 
\end{align*} 
with $|\varepsilon|, |\varepsilon'| \leq C \| a \| \| b \| n^{-1/2}$. 
Let $\bs a$ and $\bs b$ be two constant vectors in $\CC^{2N}$. Recalling 
the expression of $\tQ \bH(W\Omega^L W^*)$ above, the last display can be 
written compactly as 
\[
\left| \bs a \left( \EE \tQ \bH(W\Omega^L W^*) -  \EE \tQ C \right) \bs b 
 \right| \leq C \| \bs a \| \| \bs b \| n^{-1/2},
\]
equivalently, 
\[ 
\left\| \EE \tQ \bH(W\Omega^L W^*) -  \EE \tQ C \right\| \leq C n^{-1/2}. 
\] 
Since $\EE\tQ \in \mathfrak S^{2N}$, it is easy to prove, mostly by mimicking 
the first part of the proof of Theorem~\ref{sys}, that the matrix function 
\[
R(z, \eta)
 = \Bigl( C(z,\eta) - \begin{bmatrix} \eta & z \\ \bar z & \eta \end{bmatrix} 
  \otimes I_N \Bigr)^{-1} 
\]
is well defined for $\eta\in\CC\setminus\RR$, and 
$R(z, \cdot)\in\mathfrak S^{2N}$. In particular, 
$\| R(z,\eta) \| \leq 1 / \Im\eta$ for $\eta\in \CC_+$. We therefore have 
\begin{align*}
\| \EE \tQ - R \| &= \| \EE \tQ( R^{-1} - \tQ^{-1} ) R \| 
 = \| ( \EE \tQ C - \EE \tQ \bH(W\Omega^L W^*) ) R \| \\ 
 &\leq \| ( \EE \tQ C - \EE \tQ \bH(W\Omega^L W^*) \| \, \| R \| \\
 &\leq C n^{-1/2}.  
\end{align*}

To complete the proof of Proposition~\ref{hQ-hG}, it remains to control the
norm $\| R - \hG \|$. Remembering that $\hG$ is defined through the implicit
equation in Proposition~\ref{hG-G}, we use the contraction
property of $\ceF_{\hS,z}(\cdot, \eta)$ to gain this control.  By mimicking the
derivation that led to Inequality~\eqref{contr}, we obtain that if $K$ is large
enough so that $\| \hS \|^\TT_\infty \leq 2\bM$, and if 
$\Im\eta$ is large enough, then 
\[
\| \ceF_{\hS,z}(M, \eta) - \ceF_{\hS,z}(M',\eta) \| \leq 
 \frac 12 \| M - M' \|. 
\]
Notice that $R = \ceF_{\hS,z}(\EE\tQ, \eta)$. Therefore, if $\Im\eta$ is large 
enough, we have 
\[
\| R - \hG \| = \| \ceF_{\hS, z}(\EE \tQ,\eta) - \ceF_{\hS, z}(R,\eta) 
  + \ceF_{\hS, z}(R,\eta) - \ceF_{\hS, z}(\hG,\eta) \|  
 \leq \frac 12 \| \EE\tQ - R \| + \frac 12 \| R - \hG \| ,  
\] 
leading to 
\[
\| R - \hG \| \leq \| \EE\tQ - R \| \leq C n^{-1/2} . 
\]
Finally $\| \EE\tQ - \hG \| \leq \| \EE \tQ - R \| + \| R - \hG \| 
 \leq C n^{-1/2}$. The proof of Proposition~\ref{hQ-hG} can now be completed 
by combining this bound with the inequality~\eqref{fin-rk}.

\subsubsection*{Proof of Proposition~\ref{Q-hQ}} 

We can assume that the processes $(x^{(N)}_k)_k$ and $({\hat x}^{(N,K)}_k)_k$
that constitute the columns of the matrices $X^{(N)}$ and $\hX^{(N,K)}$
respectively, are generated by applying the filters with Fourier transforms
$n^{-1/2}S^{(N)}(e^{\imath\theta})^{1/2}$ and
$n^{-1/2}\hS^{(N,K)}(e^{\imath\theta})^{1/2}$ respectively to the same
i.i.d.~process $\bs\xi^{(N)} = (\bs\xi^{(N)}_{k})_{k\in\ZZ}$ with $\bs
\xi^{(N)}_{k} \sim \cN_\CC(0, I_N)$. This being the case, we have  
\[
\EE \| x^{(N)}_k - \hat x^{(N,K)}_k \|^2 = 
 \frac 1n \frac{1}{2\pi} \int_0^{2\pi} 
  \| S^{(N)}(e^{\imath\theta})^{1/2} - \hS^{(N,K)}(e^{\imath\theta})^{1/2} 
   \|_\HS^2 d\theta. 
\]
Hence, by \eqref{cvg-K}, we get that 
\[
\sup_N \EE \| x^{(N)}_k - \hat x^{(N,K)}_k \|^2 \xrightarrow[K\to\infty]{} 0 .
\]
For any two square matrices $M_1$ and $M_2$ of same order, by Cauchy-Schwarz 
inequality, 
$| \tr M_1 M_2 | \leq \| M_1 \|_\HS \| M_2 \|_\HS$. Thus, by the resolvent identity, 
\begin{align*}
&\frac{1}{N^2}
  \left| \tr D ( \EE Q(z, \eta) - \EE \hQ(z, \eta) ) \right|^2 \\ 
&\leq \frac{1}{N^2}
 \EE \left| \tr D Q(z,\eta)( \bH(\hX J^L \hX^*) - \bH(X J^L X^*) ) 
                                                \hQ(z,\eta) \right|^2 \\
&\leq \frac{1}{(\Im\eta)^4} \frac 1N 
 \EE \| \bH(\hX J^L \hX^*) - \bH(X J^L X^*) \|_\HS^2 .
\end{align*} 
Writing 
$\EE \| X J^L X^* - \hX J^L \hX^* \|_\HS^2 = 
\EE \| (X - \hX) J^L X^* + \hX J^L (X - \hX)^* \|_\HS^2$, it is enough to
bound $\EE \| (X - \hX) J^L X^* \|_\HS^2$. Given a constant $\kappa > 0$, we 
have 
\begin{align*}
\frac 1n \EE \| (X - \hX) J^L X^* \|_\HS^2 &\leq  
\frac 1n \EE \| (X - \hX) \|_\HS^2 \| X \|^2 \\
&\leq \frac{\kappa^2}{n} \EE \| (X - \hX) \|_\HS^2 + 
\frac 1n \EE \| (X - \hX) \|_\HS^2 \| X \|^2 \1_{\| X \| > \kappa} \\
&\leq \kappa^2 \EE \| x_k - \hat x_k \|^2 
  + \frac 1n (\EE (\| X - \hX \|_\HS^4)^{1/2} 
  (\EE \| X \|^4 \1_{\| X \| > \kappa})^{1/2}  . 
\end{align*}  
With the help of Lemma \ref{spec_norm}, the second term in the last expression
can be made as small as desired, independently of $N$,  when $N$ is large
enough, by choosing $\kappa$ large enough. The first term converges to 
zero as $K\to\infty$ as shown above. 

\subsection{Corollary~\ref{bh-20}: sketch of the proof} 

We have here $G(z,\eta) = \cF_{I_N,z}(G(z,\eta),\eta)$ for $\eta\in\CC_+$.
We also know from the proof of Theorem~\ref{sys} that if we start with 
$P_0(z,\eta) = - \eta^{-1} I_{2N}$, then the iterates 
$P_{k+1}(z,\eta) = \cF_{I,z}(P_{k}(z,\eta),\eta)$ converge to 
$G(z,\eta)$ uniformly on the compacts of $\CC_+$ in the parameter $\eta$. 
Writing $\mcT(P_k) = \begin{bmatrix} p_{00,k} & p_{01,k} \\ 
  p_{10,k} & p_{11,k} \end{bmatrix}$ where the argument $(z,\eta)$ is omitted, 
by developing the expression of $\cF_{I_N,z}(P_k,\eta)$, we get that
\[
P_{k+1} = 
 \begin{bmatrix} 
\frac{1}{2\pi} \int  \frac{p_{11,k}}{\Delta_k(e^{\imath\theta})}d\theta - \eta 
  & - \frac{1}{2\pi} \int \frac{(p_{01,k}+e^{-\imath \theta})}
  {\Delta_k(e^{\imath\theta})}d\theta  - z \\ 
 - \frac{1}{2\pi} \int \frac{(p_{10,k}+e^{\imath \theta})}
  {\Delta_k(e^{\imath\theta})}d\theta  
   - \bar z 
 & \frac{1}{2\pi} \int \frac{p_{00,k}}
  {\Delta_k(e^{\imath\theta})} d\theta -\eta 
  \end{bmatrix}^{-1} \otimes I_N , 
\]
where $\Delta_k(e^{\imath\theta}) = p_{00,k} p_{11,k} - 
 (p_{01,k}+e^{-\imath \theta}) (p_{10,k}+e^{\imath \theta})$. 
Setting from now on $\eta = \imath t$ with $t > 0$, it is obvious that 
$p_{00,0} = p_{11,0} = \imath h_0$ for $h_0 = 1/t > 0$, and 
$p_{01,0} = \bar p_{10,0}$ ($=0$ here). Assuming that 
$p_{00,k} = p_{11,k} = \imath h_k$ for some $h_k > 0$ and 
$p_{01,k} = \bar p_{10,k}$, it is not difficult to show by developing the
last display that the same properties hold for $P_{k+1}$. Passing to the limit,
we obtain that 
$g_{00}(z,\imath t) = g_{11}(z,\imath t) = \imath h$ for some $h > 0$, and
$g_{01}(z,\imath t) = \bar g_{10}(z,\imath t)$, where we wrote
$\mcT(G) = \begin{bmatrix} g_{00} & g_{01} \\ 
  g_{10} & g_{11} \end{bmatrix}$. 
With this at hand, the equation 
$G(z,\imath t) = \cF_{I,z}(G(z,\imath t),\imath t)$ becomes 
\begin{align*} 
G &= \Bigl( \frac{1}{2\pi} \int_0^{2\pi} 
  \frac{1}{\tDelta(e^{\imath\theta})} \begin{bmatrix} 
   - \imath h & g_{01} + e^{-\imath  \theta} \\
   \bar g_{01} + e^{\imath  \theta} & - \imath h \end{bmatrix} d\theta 
  - \begin{bmatrix} \imath t & z \\ \bar z & \imath t \end{bmatrix}  
 \Bigr)^{-1} \otimes I_N \\
 &= \frac nN \begin{bmatrix} \imath h & g_{01} \\ \bar g_{01} & 
 \imath h \end{bmatrix} \otimes I_N
\end{align*} 
where $\tDelta(e^{\imath\theta}) = h^2 + |g_{01} + e^{-\imath  \theta}|^2$.  
After some calculation, this equation can be equivalently restated in the form of 
the following system of two equations:
\begin{align*}
\frac Nn &= 
 \frac{1}{2\pi} \int_0^{2\pi} 
  \frac{h^2 + |g_{01}|^2 + g_{01} e^{\imath\theta}}{\tDelta(e^{\imath\theta})} 
  d\theta 
 + t h - \bar z g_{01} , \\
0 &= \frac{1}{2\pi} \int_0^{2\pi} 
  \frac{h e^{-\imath\theta}}{\tDelta(e^{\imath\theta})} d\theta 
 -  z h - t g_{01} . 
\end{align*} 
This coincides with the system of equations (33a) and (33b) of \cite{bos-hac-20}.

\begin{appendix}
\section{Proof of Lemma~\ref{gauss}} 
\begin{enumerate}
\item 
Using Markov's inequality,  an obvious integration with respect to the 
exponential distribution, and using the inequality 
$-\log(1-x) \leq 2x$ for $x\in [0,1/2]$, we have 
\begin{align*}
&\PP\left[ \| \Sigma^{1/2} \xi \| \geq \sqrt{2 N t} \right] \\ 
&= \PP\left[ \sum_{\ell=0}^{N-1} s_\ell(\Sigma) | \xi_{\ell} |^2 
    \geq 2 N t \right] 
= \PP\left[ \exp{\frac{\sum_{\ell=0}^{N-1} s_\ell(\Sigma) | \xi_{\ell} |^2}
   {2\|\Sigma\|}}  \geq \exp \frac{N t}{\|\Sigma\|} \right] \\
&\leq e^{- N t/\|\Sigma\|}  
 \EE\left[ \exp{\frac{\sum_{\ell=0}^{N-1} s_\ell(\Sigma) | \xi_{\ell} |^2}
   {2\|\Sigma\|}}  \right] 
= e^{- \frac{N t}{\|\Sigma\|}}  
 \prod_{\ell=0}^{N-1} \left( 1 - s_\ell(\Sigma)/(2\|\Sigma\|) \right)^{-1} \\
&\leq \exp \Bigl( - \frac{Nt}{\|\Sigma\|} + 
  \sum_{\ell=0}^{N-1} \frac{s_\ell(\Sigma)}{\|\Sigma\|} \Bigr) \\ 
&\leq e^{-(t /\|\Sigma\| - 1)N} . 
\end{align*} 
 
\item 
We have $\| \Sigma^{1/2} \xi \|^2 \eqlaw 
 \sum_{\ell=0}^{N-1} s_\ell(\Sigma) | \xi_{\ell} |^2 
\geq \alpha \sum_{\ell=0}^{m-1} |\xi_\ell|^2$.
Thus, by a calculation similar to
above, 
\[
\PP[ \| \Sigma^{1/2} \xi \| \leq \sqrt{\alpha m / 2} ] \leq 
 \PP[ \sum_{\ell=0}^{m-1} |\xi_\ell|^2 \leq m / 2] \leq 
\exp(- c_{\ref{gauss},\ref{normgauss}} m).
\]

\item We obviously have that 
\[
(\xi+a)^* M (\xi+a) = (\xi+a)^* \Re M (\xi+a) +
\imath \, (\xi+a)^* \Im M (\xi+a),
\]
 and both $(\xi+a)^* \Re M (\xi+a)$ and
$(\xi+a)^* \Im M (\xi+a)$ are real.  Furthermore, 
\[
\| M \|_\HS^2 = \| \Re M
\|_\HS^2 + \| \Im M \|_\HS^2.
\]
Let us assume that $\| \Re M \|_\HS \geq \|
M \|_\HS / \sqrt{2}$, otherwise, we replace $\Re M$ with $\Im M$. From these
facts, we deduce that 
\[
\cL((\xi+a)^* M (\xi+a) / \| M \|_\HS, t) \leq
\cL((\xi+a)^* \Re M (\xi+a) / \| \Re M \|_\HS, \sqrt{2} t).
\]
By a spectral factorization of the Hermitian matrix $\Re M$, 
\[
(\xi+a)^* \Re M (\xi+a) / \| \Re M \|_\HS \eqlaw \sum_{\ell\in [N]}
\beta_\ell |\xi_\ell + u_\ell|^2,
\]
where the $u_\ell$ are deterministic
complex numbers, and the $\beta_\ell$ are deterministic reals that satisfy
$\sum \beta_\ell^2 = 1$, since they are the eigenvalues of $\| \Re M
\|_\HS^{-1} \Re M$. The random variable $|\xi_\ell+u_\ell|^2$ is non-central chi-squared with two
degrees of freedom and has the density 
\[
f_\ell(x) =
\exp(-(x+|u_\ell|^2)) \bs I_0(2|u_\ell|\sqrt{x})
\]
on $\RR_+$,  where $\bs I_0$
is the modified Bessel function. Since these densities are bounded by one
\cite{abr-ste-64}, we  can get the result from \cite[Th.~1.2]{rud-ver-imrn15}.  

\item By the restriction property of L\'evy's anti-concentration function,  
\[
\cL(\Sigma^{1/2} \xi, \sqrt{m} t) \leq \sup_{(d_0,\ldots, d_{m-1})\in \CC^m} 
\PP\left[ \sum_{\ell= 0}^{m-1} 
 \left| s_\ell(\Sigma)^{1/2} \xi_\ell - d_\ell \right|^2 
 \leq m t^2 \right] . 
\]
We furthermore use the following well-known \textit{tensorization} result (see \cite[Lemma 2.2]{rud-ver-advmath08}):  
Suppose $\{ w_0, \ldots, w_{m-1} \}$ is a 
collection of independent non-negative random variables such that 
there is a constant $c > 0$ such that for each $t \geq 0$,  
\[
\PP[ w_\ell \leq t ] \leq c t.
\]
Then there exists a constant
$c' > 0$ so that 
\[
\PP[ \sum_{\ell=0}^{m-1} w_\ell^2 \leq m t^2 ] \leq (c' t)^m.
\]  
In the present case, for each $\ell\in[m]$ and each $d_\ell \in\CC$, $w_\ell^2=
[| s_\ell(\Sigma)^{1/2} \xi_\ell - d_\ell |]^2$ is a non-central chi-squared
random variable with two degrees of freedom and has a density bounded by a
constant that depends only on $\alpha$, as can be checked from the previous 
item. Thus, the tensorization argument applies and the result follows. 

\item By a singular value decomposition of $M$, we obtain that 
\[
\PP[ \| M\xi \|^2 \geq t \| M \|_\HS^2 ] = 
 \PP[ \sum_{\ell=0}^{N-1} \sigma_\ell^2 | \xi_\ell |^2 \geq t ]
\]
where 
$\sum\sigma_\ell^2 = 1$. Writing 
\[
\PP[ \sum_{\ell} \sigma_\ell^2 | \xi_\ell |^2 \geq t ] = 
 \PP[ \exp( \sum_{\ell} \sigma_\ell^2 | \xi_\ell |^2 / 2 ) \geq 
   \exp(t/2) ]
\]
and following the arguments given for Item~\ref{small-var}, we easily get the
result.  
\end{enumerate} 
Proof of Lemma~\ref{gauss} is now complete. 
\end{appendix} 


\def\cprime{$'$} \def\cdprime{$''$} \def\cprime{$'$} \def\cprime{$'$}
  \def\cprime{$'$} \def\cprime{$'$}

\end{document}